\documentclass[a4paper,reqno]{amsart}
\usepackage{tikz-cd}
\usepackage{amssymb}
\usepackage{amsmath,amscd,color}
\usepackage{combelow}
\usepackage[all]{xy}
\usepackage{enumerate}
\usepackage{mathrsfs}
\usepackage{tensor}
\usepackage{stackrel}
\usepackage{hyperref}
\usepackage{mathrsfs}
\usepackage{hyperref}

\definecolor{tocolor}{rgb}{.1,.1,.5}
\definecolor{urlcolor}{rgb}{.2,.2,.6}
\definecolor{linkcolor}{rgb}{.1,.1,.6}
\definecolor{citecolor}{rgb}{.6,.2,.1}
\hypersetup{ colorlinks=true, urlcolor=urlcolor,
  linkcolor=linkcolor, citecolor=citecolor}
\definecolor{darkgreen}{rgb}{0.0, 0.5, 0.0}

\providecommand{\U}[1]{\protect\rule{.1in}{.1in}}
\newtheorem{theorem}{Theorem}[section]
\newtheorem*{theorem*}{Theorem}
\newtheorem*{claim*}{Claim}

\newtheorem{corollary}[theorem]{Corollary}

\newtheorem{definition}[theorem]{Definition}

\newtheorem{example}[theorem]{Example}
\newtheorem{lemma}[theorem]{Lemma}
\newtheorem{proposition}[theorem]{Proposition}
\newtheorem*{proposition*}{Proposition}
\newtheorem{remark}[theorem]{Remark}
\numberwithin{equation}{section}

\newcommand{\id}{\mathrm{Id}}
\newcommand{\G}{\mathcal{G}}

\renewcommand{\U}{\mathcal{U}}

\renewcommand{\O}{\mathcal{O}}
\newcommand{\K}{\mathcal{K}}
\renewcommand{\H}{\mathcal{H}}

\newcommand{\ka}{\mathfrak{k}}

\renewcommand{\gg}{\mathfrak{g}}

\newcommand{\Rep}{V}

\newcommand{\Ato}{\Rightarrow}

\newcommand{\al}{\alpha}                
\newcommand{\be}{\beta}                 
\newcommand{\ga}{\gamma}                

\newcommand{\s}{\mathbf{s}}             
\renewcommand{\t}{\mathbf{t}}           

\newcommand{\tto}{\rightrightarrows}    
\newcommand{\timesst}{\tensor[_\s]{\times}{_\t}} 
\newcommand{\timestp}{\tensor[_\t]{\times}{_p}} 
\newcommand{\timessp}{\tensor[_\s]{\times}{_p}} 
\newcommand{\ract}{\curvearrowright}
\newcommand{\lact}{\curvearrowleft}

\newcommand{\Lie}{\mathscr{L}}
\newcommand{\dd}{\operatorname{d}}
\renewcommand{\d}{\dd}
\newcommand{\D}{\operatorname{D}}

\newcommand{\bas}{\mathrm{bas}}
\newcommand{\Der}{\operatorname{Der}}
\newcommand{\pr}{\operatorname{pr}}
\newcommand{\Aut}{\operatorname{Aut}}

\newcommand{\ad}{\operatorname{ad}}
\newcommand{\rank}{\operatorname{rank}}

\newcommand{\im}{\operatorname{Im}}

\newcommand{\mult}{\mathrm{M}}
\newcommand{\imult}{\mathrm{IM}}
\newcommand{\can}{\mathrm{can}}

\newcommand{\gl}{\operatorname{\mathfrak{gl}}}

\newcommand{\Hol}{\operatorname{Hol}}
\newcommand{\Ad}{\operatorname{Ad}}
\newcommand{\R}{\mathbb{R}}
\newcommand{\Z}{\mathbb{Z}}
\newcommand{\X}{\mathfrak{X}}
\newcommand{\diffto}{\xrightarrow{\raisebox{-0.2 em}[0pt][0pt]{\smash{\ensuremath{\sim}}}}}
\newcommand{\rmap}{\longrightarrow}

\newcommand{\eps}{\varepsilon}
\newcommand{\bO}{\overline{\Omega}}

\begin{document}
\title{Multiplicative Ehresmann connections}

\author{Rui Loja Fernandes}
\address{Department of Mathematics, University of Illinois at Urbana-Champaign, 1409 W. Green Street, Urbana, IL 61801 USA}
\email{ruiloja@illinois.edu}

\author{Ioan M\u{a}rcu\cb{t}}
\address{Radboud University Nijmegen, IMAPP, 6500 GL, Nijmegen, The Netherlands}
\email{i.marcut@math.ru.nl}

\thanks{RLF was partially supported by NSF grants DMS-2003223 and DMS-2303586.}

\begin{abstract}
We develop the theory of multiplicative Ehresmann connections for Lie groupoid submersions covering the identity, as well as their infinitesimal counterparts. We construct obstructions to the existence of such connections, and we prove existence for several interesting classes of Lie groupoids and Lie algebroids, including all proper Lie groupoids. We show that many notions from the theory of principal bundle connections have analogues in this general setup, including connections 1-forms, curvature 2-forms, Bianchi identity, etc. In \cite{FerMa22} we provide a non-trivial application of the results obtained here to construct local models in Poisson geometry and to obtain linearization results around Poisson submanifolds.
\end{abstract}
\maketitle

\setcounter{tocdepth}{1}
\tableofcontents

\section{Introduction}

The notion of principal bundle connection is ubiquitous in differential geometry and its applications. In this work we study a far reaching extension of this notion. In order to explain it, recall that given a principal $G$-bundle $p:P\to M$ one has an associated gauge groupoid $\G:=P\times_G P\tto M$ and its anchor is a groupoid submersion $\Phi:=(\t,\s):\G\to\H$ onto the pair groupoid $\H:=M\times M\tto M$. Then there is a 1-to-1 correspondence:
\[ 
\left\{\txt{principal bundle\\ \ connections for $p:P\to M$ \\ \, } \right\}
\tilde{\longleftrightarrow}
\left\{\txt{multiplicative Ehresmann\\  \ connections for $\Phi:\G\to\H$ \\ \,}\right\}
\]
Here, by a \emph{multiplicative} Ehresmann connection we mean a distribution $E$ in $\G$ such that $T\G=E\oplus\ker\d\Phi$, which is a \emph{subgroupoid} of the tangent groupoid $T\G\tto TM$. 

We study in this paper multiplicative Ehresmann connections for
any surjective, submersive, Lie groupoid map $\Phi:\G\to\H$ covering the identity. We will see that such connections share many properties with principal bundle connections. One reason for this similarity is that given such a groupoid map its kernel $\K=\ker\Phi$ is a bundle of Lie groups and one obtains for each $x\in M$ a principal bundle $\Phi:\s_{\G}^{-1}(x)\to \s_{\H}^{-1}(x)$ with structure group $\K_x$. We will see that a  multiplicative Ehresmann connection $E$ for $\Phi$ gives rise to a family of principal connections, one for each principal bundle $\Phi:\s_{\G}^{-1}(x)\to \s_{\H}^{-1}(x)$. One should think of this family as a \emph{leafwise component} of $E$. 

Recall that a Cartan connection on a Lie groupoid $\G\tto M$ is an Ehresmann connection for the source map $\s:\G\to M$, which is at the same time a multiplicative distribution. Such connections have been studied extensively, sometimes under different names (see, e.g., \cite{AC13,Blaom06,Blaom12,Behrend05,CrSchStr,Tang06}). For a bundle of Lie groups $p:\K\to M$, they coincide with our notion of multiplicative Ehresmann connection, if we think of the bundle projection $p$ as a groupoid morphism onto the identity groupoid. We will see that a multiplicative Ehresmann connection for $\Phi:\G\to\H$ also gives rise to a Cartan connection $E^\K$ on the kernel $\K=\ker\Phi$. We think of $E^\K$ as the \emph{kernel component} of $E$.

The existence of Cartan connections is, in general, rather mysterious. There are examples of proper groupoids (even transitive!) groupoids which do not admit any Cartan connection. By contrast, we will see that multiplicative Ehresmann connections for a morphism $\Phi:\G\to \H$ often (but not always!) exist. 
For example, they always exist for the anchor map $(\t,\s):\G\to M\times M$ of a transitive Lie groupoid, because of the 1-to-1 correspondence with principal connections, explained in the first paragraph. Also, we will prove the following:

\begin{theorem*}
\label{thm:Morita}
If $\G$ is a proper Lie groupoid, any morphism $\Phi:\G\to \H$ admits a multiplicative Ehresmann connection.
\end{theorem*}

The reason why this result holds is that multiplicative Ehresmann connections, unlike Cartan connections, can be transported along Morita equivalences. However, this requires to consider multiplicative Ehresmann connections for slightly more general objects than groupoid morphisms. Given a Lie groupoid $\G$ with Lie algebroid $A$ we will define multiplicative Ehresmann connections for any bundle of ideals $\ka\subset A$ which are invariant under the conjugation action of $\G$. When $\Phi:\G\to\H$ is a morphism as above and $\ka$ is the bundle of ideals induced by its kernel $\K$ (i.e., the Lie algebroid of $\K$), the two notions coincide. 

A bundle of ideals $\ka$ for a Lie groupoid $\G$ is called \emph{partially split} if it admits a multiplicative Ehresmann connection $E$. Note that this more general notion appears naturally in applications. We were originally led to introduce them in the study of local models around Poisson submanifolds in Poisson geometry.
There one finds a constant rank multiplicative 2-form $\omega\in\Omega^2_M(\G)$ whose kernel defines a bundle of ideals. A multiplicative Ehresmann connection for this bundle of ideals is necessary to construct a groupoid coisotropic embedding of $(\G,\omega)$ into a symplectic groupoid. The infinitesimal version of this embedding is precisely a local model for a Poisson submanifold. We will not discuss any of these applications here and we instead refer the reader to \cite{FerMa22}.

Just like for principal connections, rather than specifying a multiplicative Ehresmann connection via a distribution $E\subset T\G$ one can give instead a \emph{multiplicative connection 1-form}. This is a multiplicative 1-form $\al\in\Omega^1_\mult(\G;\ka)$ with coefficients in the bundle of ideals $\ka$ which additionally must satisfy:
\[ \al(\xi^L)=\xi, \quad (\xi\in\Gamma(\ka)), \]
where $\xi^L\in\X(\G)$ is the left-invariant vector field associated to the section $\xi$. The \emph{curvature of the connection} is then a multiplicative 2-form with coefficients in $\ka$, denoted $\Omega=D\al\in\Omega^2_\mult(\G;\ka)$. It measures the failure in $E$ being an involutive distribution. 
One advantage of specifying a multiplicative Ehresmann connection via a connection 1-form is that is leads immediately to its infinitesimal version. Namely, given a Lie algebroid $A\Ato M$, integrable or not, and a bundle of ideals $\ka\subset A$, an \emph{IM (infinitesimal multiplicative) Ehresmann connection} is given by an IM 1-form on $A$ with coefficients in $\ka$ (see, e.g., \cite{CrSaSt15}), whose symbol $l:A\to \ka$ satisfies:
\[ l(\xi)=\xi, \quad (\xi\in\Gamma(\ka)). \]

The multiplicative Ehresmann connections studied in this paper have been studied before in \cite{LSX09}, for the special case of a groupoid extension in the context of the theory of non-abelian gerbes. There, by a groupoid extension the authors mean a groupoid map  $\Phi:\G\to\K$ between Hausdorff Lie groupoids which is locally trivial. It implies, in particular, that whenever it admits a connection, its kernel is a locally trivial bundle of Lie groups. Many of the results we obtain in Section \ref{sec:multiplicative:connections} recover, in this special case, results of \cite{LSX09}. Our result mentioned above concerning transporting multiplicative connections along Morita equivalences also extends a theorem from \cite{LSX09}, which establishes Morita invariance of the existence of connections. Note that the notion of Morita equivalence of groupoid extensions, used in \cite{LSX09}, is more restrictive than the Morita equivalence of Lie groupoids used here. In particular, Theorem \ref{thm:Morita} does not follow from the results of \cite{LSX09}.

In \cite[Section 6.5]{BD19} the authors consider multiplicative tensors $K\in\Omega^1(\G,T\G)$ satisfying $K^2=K$, which they call \emph{multiplicative projections}. These include, as a special case, multiplicative Ehresmann connections, and the so-called Fr\"olicher-Nijenhuis  bracket of $K$ (see \cite{BD19}) coincide with our curvature 2-forms. Also, the notion of matched pair of Lie algebroid from \cite{BD19} can be interpreted as a version of our IM connections.

There are many questions related to our theory of connections and its applications that we do not discuss in this paper. Besides the applications to Poisson geometry \cite{FerMa22} and to gerbes \cite{LSX09} already mentioned, it is also natural to consider Ehresmann connections for morphisms of groupoids over different bases or more general types of bundles of ideals \cite{JO14}. Another intriguing question is to define appropriate moduli spaces of flat multiplicative connections. We plan to return to some these questions in future work.

This paper is organized as follows. In Section \ref{sec:multiplicative:connections} we introduce multiplicative Ehresmann connections, and we discuss:
\begin{itemize}
    \item[\tiny$\bullet$] alternative characterizations of connections;
    \item[\tiny$\bullet$] obstructions to their existence;
    \item[\tiny$\bullet$] necessary and sufficient criteria for completeness;
    \item[\tiny$\bullet$] curvature, structure equation and Bianchi's identity; 
    \item[\tiny$\bullet$] various notions of flatness and relationship to semi-direct products.
\end{itemize}
In Section \ref{examples:groupoids}, we give many classes of Lie groupoids and morphisms that admit multiplicative Ehresmann connections. Section \ref{sec:Morita} is dedicated to the proof of Morita invariance and to show the existence of multiplicative Ehresmann connections for proper groupoids. In this section, we also illustrate our results with an application to the theory of $S^1$-gerbes over a manifold, recovering the classical result of Murray about representing the Dixmier-Douady class in real cohomology. In Section \ref{section: IM Ehresmann}, we introduce IM Ehresmann connections, and we discuss:
\begin{itemize}
    \item[\tiny$\bullet$] alternative characterizations of IM connections;
    \item[\tiny$\bullet$] obstructions to their existence;
    \item[\tiny$\bullet$] a theory of coupling forms for IM connections, generalizing the classical coupling description of symplectic fibrations;
    \item[\tiny$\bullet$] curvature of IM connections; 
    \item[\tiny$\bullet$] various notions of flatness and relationship to semi-direct products of Lie algebroids.
\end{itemize}
In Section \ref{sec:examples}, we give many classes of Lie algebroids and morphisms that admit IM Ehresmann connections, including non-integrable algebroids. At the end of the paper, we have included an appendix with background and results on multiplicative forms and IM forms with coefficients, needed throughout the paper. The results in Section \ref{sec:differentiation:forms:coefficients} concerning covariant differentiation of such forms seem to be new.

\medskip
{\bf Acknowledgments.} We would like to thank Henrique Bursztyn for bringing to our attention his work \cite{BD19}. We specially would like to thank Camille Laurent Gengoux for many comments on a first version of this paper posted in the arXiv, and for pointing out the connections to his joint work with Mathieu Sti\'enon and Ping Xu \cite{LSX09}, which we were unaware of.

\bigskip
{\bf Conventions and notations.} We denote a Lie groupoid by $\G\tto M$, with source/target $\s,\t:\G\to M$ and multiplication $m:\G\timesst \G\to \G$. We denote a Lie algebroid by $A\Ato M$, with Lie bracket $[\cdot,\cdot]_A:\Gamma(A)\times\Gamma(A)\to\Gamma(A)$ and anchor $\rho_A:A\to TM$. Our conventions are as in \cite{CFM21}, so the Lie algebroid of $\G$ is the vector bundle $A=\ker(\d\t)|_M$ with Lie bracket induced from the bracket of left-invariant vector fields and anchor $\rho=\d\s$. Also,
Proper Lie groupoids are assumed Hausdorff, while general Lie groupoids are not necessarily Hausdorff, unless stated otherwise. 

\section{Multiplicative Ehresmann connections}
\label{sec:multiplicative:connections}

\subsection{Ehresmann connection for a Lie groupoid submersion}
Let $\G\tto M$ and $\H\tto M$ be two Lie groupoids, and 
\[\Phi:\G\to \H\]
be a Lie groupoid map covering $\id_M$ which is a surjective submersion. We are investigating the existence problem for the following objects.

\begin{definition}
\label{def:Ehresmann connection:morphism}
A \textbf{multiplicative Ehresmann connection} for $\Phi$ is a ``horizontal'' distribution 
\[T\G=\ker(\dd \Phi)\oplus E\] 
which is also a Lie subgroupoid $E\tto TM$ of the tangent groupoid $T\G\tto TM$.
\end{definition}

Notice that the kernel of a Lie groupoid map $\Phi:\G\to \H$ as above 
\[\K:=\ker \Phi = \Phi^{-1}(\textrm{units}) \subset \G\]
is a bundle of Lie subgroups of $\G$ which is \textbf{normal}, i.e., conjugation by any $g\in \G$ gives a Lie group isomorphism:
\[\Ad_g:\K_{\s(g)}\diffto \K_{\t(g)},\quad k\mapsto gkg^{-1}.\]

Let $A\Ato M$ and $B\Ato M$ denote the Lie algebroids of $\G\tto M$ and $\H\tto M$, respectively, and let $\phi:A\to B$ be the Lie algebroid map induced by $\Phi$. Then the Lie algebroid $\ka\Ato M$ of $\K$ is a bundle of Lie algebras that fits into the short exact sequence of Lie algebroids:
\[ \xymatrix{0\ar[r]& \ka\ar[r]& A\ar[r]^{\phi}& B\ar[r]& 0}.\]
The ``conjugation action" induces an action on the isotropies:
\begin{equation}
\label{eq:representation:grpd} 
\Ad_g:\ker\rho|_{\s(g)}\diffto \ker\rho|_{\t(g)}, \quad g\cdot \al=\frac{\d}{\d t}\Big|_{t=0} g\, \exp_{\s(g)}(t\alpha)\, g^{-1},
\end{equation}
where $\exp_{x}:\ker\rho|_x\to \G_{x}$ is the Lie group exponential. Since $\K$ is invariant under conjugation, we obtain a representation of $\G$ on $\ka$. 

Finally, note that the subbundle $K:=\ker (\d\Phi)$ can be recovered from $\ka$ by using (left or right) translations:
\[K_g=(\ker \d\Phi)|_g=\dd L_g(\ka|_{\s(g)})=\dd R_g(\ka|_{\t(g)})\subset \ker(\dd \s)\cap \ker(\dd \t).\]

\begin{remark}
\label{rem:LSX}
Connections in the sense of Definition \ref{def:Ehresmann connection:morphism} were introduced in \cite{LSX09} under the extra assumption that $\Phi$ is a locally trivial fibration. The authors call them ``connections for Lie groupoid extensions''.
\end{remark}

\subsection{Partially split bundles of ideals}
It turns out that multiplicative Ehresmann connections can be defined in a more general setting, without the presence of a Lie groupoid morphism.

\begin{definition}
Given a Lie groupoid $\G\tto M$ with Lie algebroid $A\Ato M$, we call a vector subbundle $\ka\subset A$ a \textbf{bundle of ideals} of $\G$ if $\ka\subset \ker\rho$ and $\ka$ is invariant under the $\G$-action by conjugation \eqref{eq:representation:grpd}.
\end{definition}

The fact that $\ka$ is invariant under conjugation implies that \[
 \al\in\Gamma(A), \be\in\Gamma(\ka) \quad \Longrightarrow\quad [\al,\be]\in\Gamma(\ka).
 \]
When $\G\tto M$ is source-connected, this condition is equivalent to $\ka\subset \ker \rho$ being a bundle of ideals (see, e.g., Appendix B \cite{Marcut14}). So this condition defines the notion of a \textbf{bundle of ideals} for any Lie algebroid $A$, integrable or not. We will explore this in Section \ref{section: IM Ehresmann}, where we study the infinitesimal analog of multiplicative Ehresmann connections. 

Given a bundle of ideals $\ka\subset A$ for $\G$, by using (left or right) translations we obtain the involutive distribution on $\G$:
\begin{equation}
    \label{eq:bundle:K}
    K_g:=\dd L_g(\ka_{\s(g)})=\dd R_g(\ka_{\t(g)})\subset \ker(\dd \s)\cap \ker(\dd \t).
\end{equation}
Note that $K\tto 0_M$ is a subgroupoid of the tangent groupoid $T\G\tto TM$.

\begin{definition}
\label{def:Ehresmann connection:ideals}
A bundle of ideals $\ka$ of $\G$ is said to be \textbf{partially split} if it admits a \textbf{multiplicative Ehresmann connection}, that is, if there is a wide Lie subgroupoid $E\subset T\G$ such that
\[T\G=E\oplus K\] 
where $K$ is the subgroupoid \eqref{eq:bundle:K}.
\end{definition}

\begin{remark}
Note that, given a bundle of ideals $\ka\subset A$, we still have a short exact sequence of Lie algebroids:
\[ \xymatrix{0\ar[r]& \ka\ar[r]& A\ar[r]& B\ar[r]& 0}.\]
where $B:=A/\ka$. However, at the groupoid level, there may not exist a closed embedded subgroupoid $\K\subset \G$ integrating $\ka$ and even $B$ may fail to be integrable (see Example \ref{example:Lie-Poisson:sphere}). So in this more general setup a Lie groupoid morphism $\Phi:\G\to\H$ as in Definition \ref{def:Ehresmann connection:morphism} may not exist. If it exists, then $K=\ker\dd\Phi$.
\end{remark}


\begin{remark}
\label{rem:nomenclature}
We will be using the following nomenclature:
\begin{itemize}
    \item[\tiny$\bullet$] A {\bf multiplicative Ehresmann connection for a morphism} $\Phi:\G\to\H$, as in Definition \ref{def:Ehresmann connection:morphism}. We will always assume that it is a surjective, submersive, Lie groupoid morphism covering the identity, without explicitly mentioning it;
    \item[\tiny$\bullet$] A {\bf multiplicative Ehresmann connection for a bundle of ideals} $\ka$ of a groupoid $\G$, as in Definition \ref{def:Ehresmann connection:ideals};
    \item[\tiny$\bullet$] A {\bf Cartan connection for a groupoid} $\G$: by this we mean a multiplicative distribution $E\subset T\G$ such that
    \[ T\G=\ker\d\s\oplus E. \]
\end{itemize}
The last notion has been discussed and used extensively in the literature, sometimes under different names (see, e.g., \cite{AC13,Blaom06,Blaom12,Behrend05,CrSchStr,Tang06}). For a general groupoid, the source map $\s:\G\to M$ cannot be viewed as groupoid morphism, and this notion is distinct from the notions introduced above. However, for a bundle of Lie groups $p:\G\to M$, where $\s=\t=p$, one can view the projection as a groupoid morphism onto the identity groupoid $M\tto M$. The corresponding bundle of ideals is the Lie algebroid $A$ of $\G$. So a Cartan connection for a bundle of Lie groups $p:\G\to M$ is the same as a multiplicative Ehresmann connection for the morphism $p$, or a multiplicative Ehresmann connection for the bundle of ideals $A$.
\end{remark}

\subsection{The partially split condition}

We will now look for alternative characterizations of partially split bundles of ideals. We start by observing that $K$ is a actually a semi-direct product:

\begin{lemma}
The subgroupoid $K\tto 0_M$ of $T\G\tto TM$ is canonically isomorphic to the semi-direct product $\G\times_M\ka\tto M$, with multiplication:
\[ (g,v)\cdot (h,w)=(gh,h^{-1}\cdot v+w), \quad \text{if}\quad \s(g)=\t(h), v\in \ka_{\s(g)}, w\in\ka_{\s(h)}. \]
\end{lemma}

\[
\vcenter{
\xymatrix@R=8pt{
{}\save[]+<-30pt,0cm>*\txt{$\G\times_M\ka\simeq$}\restore K\, \ar@<0.15pc>[dd] \ar@<-0.15pc>[dd] \ar@{^{(}->}@<-0.10pc>[rr]  \ar[dr]&  &  T\G  \ar@<0.15pc>[dd] \ar@<-0.15pc>[dd] \ar[dl]\\
 & 
\G \ar@<0.15pc>[dd] \ar@<-0.15pc>[dd] & \\
0_M\, \ar[dr] \ar@{^{(}-}[r] &\ar[r]  & TM \ar[dl]\\
 & M}
 }
\]

\begin{proof}
One checks immediately that the map:
\[ \G\times_M\ka\to K , \quad (g,v)\mapsto \d L_g(v), \]
is a groupoid isomorphism.
\end{proof}

The Lie groupoids $T\G\tto TM$ and $K\tto 0_M$ are examples of VB groupoids with cores $A\to M$ and $\ka\to M$, respectively. Under the duality operation in the category of VB groupoids (see \cite{Mackenzie05,BCdH16,GrMe17}), these have duals the VB groupoids $T^*\G\tto A^*$ and $K^*\tto \ka^*$, with cores $T^*M \to M$ and $0_M\to M$, respectively. The groupoid $K^*\tto \ka^*$ is isomorphic to the action groupoid $\G\ltimes\ka^*\tto \ka^*$ resulting from the dual action of $\G$ on $\ka^*$. 
Moreover, by applying duality to the inclusion map, the restriction map $T^*\G\to K^*$ is a groupoid morphism:
\[
\vcenter{
\xymatrix@R=8pt{
{}\save[]+<-30pt,0cm>*\txt{$\G\ltimes\ka^*\simeq$}\restore K^*\, \ar@<0.15pc>[dd] \ar@<-0.15pc>[dd] \ar[dr] &  &  T^*\G\ar@{->>}[ll] 
\ar@<0.15pc>[dd] \ar@<-0.15pc>[dd] \ar[dl]\\
 &  \G  \ar@<0.15pc>[dd] \ar@<-0.15pc>[dd]  \\
\ka^* \ar[dr] & \ar@{->>}[l] & A^* \ar[dl] \ar@{-}[l]\\
 & M
}}
\]




The following proposition gives useful characterizations of the partially split condition
(for the terminology, see the appendix): 

\begin{proposition}
\label{prop:partially:split:grpd}
Let $\ka$ be a bundle of ideals for $\G$. The following structures are in 1-to-1 correspondence: 
\begin{enumerate}[(i)]
\item VB subgroupoids $E\subset T\G$ that are complementary to $K$:
\[ T\G=K\oplus E; \]
\item VB groupoid morphisms $\Theta:\G\ltimes\ka^* \to T^*\G$ that are splittings of the projection $p:T^*\G\to \G\ltimes \ka^*$:
\[ p\circ\, \Theta=\id;\]
\item $\ka$-valued, multiplicative, 1-forms $\alpha\in\Omega^1_\mult(\G,\ka)$ that restrict to the identity on $\ka\subset T_M\G$:
\[ \al|_\ka=\id;\]
\item linear, closed, multiplicative, 2-forms $\omega\in\Omega^2_\mult(\G\ltimes \ka^*)$ (see Definition \ref{definition:linear:multipl}) that restrict to the canonical symplectic form on $\ka\times_M\ka^*\subset T_M(\G\ltimes \ka^*)$:
\[ \omega((v_1,\xi_1),(v_2,\xi_2)) = \xi_2(v_1)-\xi_1(v_2), \quad \text{if} \quad(v_k,\xi_k)\in\ka\times_M\ka^*.\]
\end{enumerate}
\end{proposition}

\begin{proof}
The equivalence between (i) and (ii) follows from the discussion above.

Let $A$ be the Lie algebroid of $\G$. A bundle map 
\[
\xymatrix@C=10pt@R=10pt{
\G\ltimes \ka^* \ar@<0.15pc>[dd] \ar@<-0.15pc>[dd] \ar[rr]^{\Theta} \ar[dr]&  &  T^*\G  \ar@<0.15pc>@{-}[d] \ar@<-0.15pc>@{-}[d] \ar[dr]\\
 & 
\G\ar[rr]^(.3){\id}  \ar@<0.15pc>[dd] \ar@<-0.15pc>[dd] & \ar@<0.15pc>[d] \ar@<-0.15pc>[d] &\G  \ar@<0.15pc>[dd] \ar@<-0.15pc>[dd] \\
\ka^* \ar[dr] \ar@{-}[r] &\ar[r]  & A^* \ar[dr]\\
 & M\ar[rr] &  &M
}
\]
is a VB groupoid morphism if and only if its dual is a VB groupoid morphism:
\[
\xymatrix@C=10pt@R=10pt{
T\G \ar@<0.15pc>[dd] \ar@<-0.15pc>[dd] \ar[rr]^{\Theta^\vee} \ar[dr]&  &  \G\times_M \ka  \ar@<0.15pc>@{-}[d] \ar@<-0.15pc>@{-}[d] \ar[dr]\\
 & 
\G\ar[rr]^(.3){\id}  \ar@<0.15pc>[dd] \ar@<-0.15pc>[dd] & \ar@<0.15pc>[d] \ar@<-0.15pc>[d] &\G  \ar@<0.15pc>[dd] \ar@<-0.15pc>[dd] \\
TM \ar[dr] \ar@{-}[r] &\ar[r]  & 0_M \ar[dr]\\
 & M\ar[rr] &  &M
}
\]
But a VB groupoid morphism $\Theta^\vee:T\G\to \G\times_M \ka$ covering the identity is the same thing as a $\ka$-valued, multiplicative 1-form $\alpha\in\Omega^1_\mult(\G,\ka)$ (see Example \ref{ex:degree1:E-valued:M:forms}). This gives a correspondence $\Theta\leftrightarrow \al$, where the conditions $p\circ\, \Theta=\id$ and $\al|_\ka=\id$ correspond to each other, proving the equivalence between (ii) and (iii).

By Lemma \ref{lemma:decomposition:linear:forms}, there is a 1-to-1 correspondence between linear, closed forms $\omega\in \Omega^2_{\mult}(\G\ltimes \ka^*)$ and vector bundle maps $\Theta:\G\ltimes \ka^* \to T^*\G$ covering the identity, determined by the relation:
\[\omega=\Theta^*(\omega_{\can}).\]
Clearly, if $\Theta$ is a groupoid map, it follows that $\omega$ is multiplicative. Conversely, if $\omega$ is multiplicative, then we have a VB groupoid morphism:
\[
\xymatrix{
T(\G\ltimes \ka^*)\ar[r]^{\omega^\flat}  \ar@<0.15pc>[d] \ar@<-0.15pc>[d]  & T^*(\G\ltimes \ka^*)  \ar@<0.15pc>[d] \ar@<-0.15pc>[d] \\
T\ka^*\ar[r] & (A\ltimes \ka^*)^*
}
\]
Via the assignment of Lemma \ref{lemma:decomposition:linear:forms}, the map $\Theta$ is the composition: 
\[\G\ltimes \ka^*\hookrightarrow T_{\G}(\G\ltimes \ka^*) \stackrel{\omega^{\flat}}{\rmap}T_{\G}^*(\G\ltimes \ka^*)\rmap T^*\G.\]
Hence, $\Theta$ is a groupoid morphism. Under the correspondence $\Theta\leftrightarrow \omega$, the extra conditions in (ii) and (iv) correspond to each other, proving the equivalence between these two items. 
%
%
\end{proof}

The characterization of multiplicative Ehresmann connections in terms of multiplicative 1-forms (item (iii) in the previous proposition) also appears in \cite[Thm 4.12]{LSX09} for the case of Lie groupoid extensions (cf.~Remark \ref{rem:LSX}). We will be making extensive use of it.

\begin{definition}
Given a partially split bundle of ideals $\ka$ with a multiplicative Ehresmann connection $E\subset T\G$, the corresponding $\ka$-valued 1-form $\alpha\in\Omega^1_\mult(\G,\ka)$, given by Proposition \ref {prop:partially:split:grpd} (iii), is called
the \textbf{multiplicative connection 1-form}.
\end{definition}


Next we deduce two obstructions for a bundle of ideals to be partially split. We start by observing that for bundles of Lie groups multiplicative connections induce ordinary connections, a remark that will also be useful in future sections. This remark follows, e.g., by observing that a multiplicative connection on a bundle of Lie groups is an example of a Cartan connection (see Remark \ref{rem:nomenclature}), hence it induces an infinitesimal Cartan connection on its Lie algebroid (see \cite{Blaom16} for more details and further references). We give an independent proof:

\begin{proposition}
\label{prop:connection:partial:split}
Let $p:\K\to M$ be a bundle of Lie groups endowed with a multiplicative Ehresmann connection: $T\K=E^{\K}\oplus \ker\d p$. Then the corresponding bundle of Lie algebras $\ka\to M$ has an induced linear connection $\nabla$ which preserves the Lie bracket:
\[\nabla_{X}[\xi,\eta]_{\ka}=[\nabla_{X}\xi,\eta]_{\ka}+[\xi,\nabla_{X}\eta]_{\ka},\]
for all $X\in\X(M)$, $\xi,\eta\in\Gamma(\ka)$. 
The corresponding linear Ehresmann connection  $E^{\ka}\subset T\ka$ is related to $E^{\K}$ via the exponential map $\exp:\ka\to \K$ on a neighborhood $U\subset\ka$ of the zero section: 
\begin{equation}\label{eq:exp:connection}
 (\d\exp)(E^{\ka}|_U)=E^{\K}|_{\exp(U)}.   
\end{equation}
\end{proposition}




\begin{proof}
The connection is defined by:
\begin{equation}\label{eq:of:nabla}
\nabla_X\xi=[\widetilde{X},\xi^{L}]|_M,\quad X\in \X(M),\ \xi\in \Gamma(\ka),
\end{equation}
where $\widetilde{X}$ is the horizontal lift with respect to $E^{\K}$ and $\xi^{L}\in \X(\K)$ is the left-invariant vector field corresponding to $\xi$. Multiplicativity of $E^\K$ implies that $[\widetilde{X},\xi^{L}]$ is a left-invariant vector field, so the Jacobi identity shows that the connection preserves the Lie bracket on $\ka$.

We now claim that if $U\subset\ka$ is a neighborhood of the zero section where $\exp$ is a diffeomorphism, then \eqref{eq:exp:connection} holds. For this, we will show that the left-hand side is a multiplicative distribution. It then follows that both sides are (local) Lie subgroupoids of $T\K$ with the same Lie algebroid, so they must coincide. Hence, we assume that $\K=\ka$ and multiplication is given by the (fiberwise) Baker-Campbell-Hausdorff series $\mathrm{BCH}(\cdot,\cdot)$. 

Notice that:
\[ E^{\ka}_u=\big\{(\d_x s)(v): v\in T_xM,\, s\in\Gamma(\ka),\, s(x)=u,\, (\nabla s)|_x=0\big\}. \]
So it is enough to check that for $s_1,s_2\in \Gamma(\ka)$:
\[ (\nabla s_1)|_x=(\nabla s_2)|_x=0\quad \Longrightarrow \quad (\nabla \textrm{BCH}(s_1,s_2))|_x=0. \]
This follows from the expression of $\mathrm{BCH}(\cdot,\cdot)$ as a sum of commutators because $\nabla$ preserves the brackets.
\end{proof}

\begin{corollary}
\label{cor:connection:partial:split}
Given a partially split bundle of ideals $\ka$ for a Lie groupoid $\G$ each multiplicative Ehresmann connection $E\subset T\G$ determines a linear connection $\nabla^E$ on the bundle of Lie algebras $\ka\to M$ which preserves the Lie bracket.
\end{corollary}

\begin{proof}
Let  $\G(\ka)$ be the bundle of simply connected Lie groups integrating $\ka$ and $i:\G(\ka)\hookrightarrow \G$ the immersion integrating the inclusion. Given a multiplicative Ehresmann connection $E$ on $\G$ for $\ka$ the pullback $(\d i)^{-1}(E)\subset T \G(\ka)$ is a multiplicative Ehresmann connection for the bundle of Lie groups $p:\G(\ka)\to M$. Hence, the result follows from the proposition. 
\end{proof}

\begin{corollary}\label{corollary:locally:trivial:groupoid}
If a bundle of ideals $\ka\subset A$ for $\G$ is partially split, then $\ka$ is a locally trivial bundle of Lie algebras.
\end{corollary}

The corollary implies that $\G(\ka)$ is a Hausdorff manifold (recall that bundles of Lie algebras may fail to have Hausdorff integrations). 

The partial split condition also places restrictions on the isotropy Lie algebras:

\begin{corollary}\label{corollary:complement:groupoid}
If a bundle of ideals $\ka\subset A$ for $\G$ is partially split, then for each $x\in M$ the isotropy Lie algebra splits into a direct sum of ideals: 
\[ \gg_x=\ka_x\oplus \mathfrak{l}_x. \]
\end{corollary}

\begin{proof}
Observe that if $\ka\subset A$ is partially split for $\G$ then for each $x\in M$ the ideal $\ka_x\subset\gg_x$ is partially split for $\G_x^0$, the connected component of the isotropy Lie group. This follows, e.g., by restricting a multiplicative connection 1-form $\alpha\in\Omega^1_\mult(\G,\ka)$ to $\G_x^0$. The corollary now follows from the fact that an ideal $\ka\subset \gg$ for a connected Lie group $G$ is partially split if and if only it admits a complementary ideal in $\gg$, as discussed in the example from Subsection \ref{ex:Lie:groups}. 
\end{proof}

\begin{remark}
In the case of connections for Lie groupoid extensions, the associated linear connection $\nabla$ and Proposition \ref{prop:connection:partial:split} was also deduced in \cite{LSX09} (see Corollary 4.5, Lemma 4.6 and Proposition 4.21 there).
\end{remark}

\subsection{Completeness}
\label{sec:complete}

In this section we will assume that all Lie groupoids are Hausdorff. Given a multiplicative Ehresmann connection for a groupoid homomorphism $\Phi:\G\to \H$ covering the identity we can ask if it is complete, i.e., if one can lift any path in $\H$ to an horizontal path in $\G$. In this section we prove the following result which completely settles this question. 

\begin{theorem}
\label{thm:completeness}
Let $\Phi:\G\to \H$ be a surjective, submersive, Lie groupoid homomorphism covering the identity. A multiplicative Ehresmann connection for $\Phi$ is complete if and only the kernel $\K=\ker\Phi$ is a locally trivial bundle of Lie groups.
\end{theorem}

\begin{remark}
Given a surjective submersion $\Phi:M\to N$ between two manifolds one has (see, e.g., \cite{dH16,Fr19}):
\begin{enumerate}[(i)]
    \item $\Phi$ is locally trivial if and only if $\Phi:M\to N$ admits a complete Ehresmann connection;
    \item $\Phi$ is proper if and only if every Ehresmann connection on $\Phi:M\to N$ is complete.
\end{enumerate}
One can think of the theorem as a multiplicative version of (i). On the other hand, it it is not hard to see that a surjective, submersive, groupoid morphism $\Phi:\G\to \H$ covering the identity is proper if and only if its kernel $\K$ is proper (i.e., $p:\K\to M$ is a proper map). 
\end{remark}

The rest of this section is devoted to the proof of the theorem. As before, set
\[ K:=\ker(\d\Phi)\subset T\G. \] 
We also denote by $p:=\s_\K=\t_\K$ the common source and target of the bundle of Lie groups $\K\to M$. Denoting by $i:\K\hookrightarrow \G$ the inclusion, notice that 
\[ E^\K:=(\d i)^{-1}(E)=E\cap T\K, \]
gives a multiplicative Ehresmann connection on the bundle $p:\K\to M$:
\[ T\K=\ker(\d p)\oplus E^\K,\text{ where }E^\K:=E\cap T\K. \]
 In terms of the nomenclature of Remark \ref{rem:nomenclature}, $E^\K$ is a Cartan connection for $\K$. The resulting linear connection on the Lie algebroid $\ka$ of $\K$ is precisely the linear connection $\nabla^E$ of Proposition \ref{prop:connection:partial:split}.

The following result reduces the proof of Theorem \ref{thm:completeness} to the case of a bundle of Lie groups.


\begin{proposition}
\label{prop:complete:Ehresmann:connection}
A multiplicative Ehresmann connection for a groupoid homomorphism $\Phi:\G\to \H$ is complete if and only if the multiplicative Ehresmann connection induced on its kernel $p:\K\to M$ is complete.
\end{proposition}

\begin{proof}
In one direction the statement is obvious. 

For the converse, assume that $\Phi:\G\to \H$ admits a multiplicative Ehresmann connection $E\subset T\G$ such that the induced Ehresmann connection on the kernel $E^\K\subset T\K$ is complete. We observe that the following properties hold, independent of completeness:
\begin{enumerate}[(i)]
    \item If $g:I\to \G$ is a horizontal path, then $t\mapsto g(t)^{-1}$ is also horizontal;
    \item If $g_1,g_2:I\to\G$ are horizontal paths with $\s(g_1(t))=\t(g_2(t))$ for all $t\in I$, then $t\mapsto g_1(t)g_2(t)$ is also horizontal;
    \item if $g:I\to \G$ is a horizontal lift of a path $h:I\to \H$ and $k:I\to \K$ is a horizontal lift of $\gamma:=\s\circ h:I\to M$, then
    \[ \overline{g}:I\to \G,\quad \overline{g}(t):=g(t)k(t), \]
    is also a horizontal lift of $h:I\to \H$;
    \item any two horizontal lifts $g_1,g_2:[a,b]\to \G$ of $h:[a,b]\to \H$ are related by:
    \[g_2(t)=g_1(t)k(t) \quad (t\in[a,b]),\]
    where $k:[a,b]\to \K$ is the horizontal lift of $\s\circ h:[a,b]\to M$ with $k(a)=g_1(a)^{-1}g_2(a)$.
\end{enumerate}
Indeed, (i) and (ii) follow immediately from the fact that $E\tto TM$ is a subgroupoid of $T\G\tto TM$. Then (iii) follows from (i) and the fact that the connection $E^\K$ on $\K\to M$ is the restriction of $E$ to $T\K$. Finally, (iv) follows from (i)-(iii) and the uniqueness of horizontal lifts.

Now, using these facts, we have the following standard argument showing that for every path $h:[0,1]\to \H$ and any $g_0\in \G$ with $\Phi(g_0)=h(0)$ there exists a horizontal lift $g:[0,1]\to\G$ starting at $g_0$. First, we always have an horizontal lift $g:[0,t_0)\to\G$, defined for some $t_0>0$. Then we can take any horizontal lift $ \overline{g}:(t_0-\eps,t_0+\eps)\to \G$ of $h:(t_0-\eps,t_0+\eps)\to \H$. Then we can write:
\[  \overline{g}(t)=g(t)k(t),\quad t\in (t_0-\eps,t_0),\]
where $k:(t_0-\eps,t_0)\to \K$ is a horizontal lift of $\s\circ h:(t_0-\eps,t_0)\to M$.
Since $E^\K\subset T\K$ is complete, the path $k:(t_0-\eps,t_0)\to \K$ extends to a (unique) horizontal lift of $\s\circ h:[0,1]\to M$ on the whole interval $[0,1]$. But then we obtain an horizontal extension $g:[0,t_0+\eps)\to \G$ by setting $g(t):= \overline{g}(t)k(t)^{-1}$, for $t\in (t_0-\eps,t_0+\eps)$. We conclude that $g(t)$ extends to whole interval $[0,1]$.
\end{proof}

To complete the proof of Theorem \ref{thm:completeness} we will show the following:

\begin{proposition}
\label{prop:completeness:bundle of groups}
Let $p:\K\to M$ be a bundle of Lie groups. The following are equivalent:
\begin{enumerate}[(a)]
  \item $\K$ admits a complete multiplicative Ehresmann connection;
    \item $p:\K\to M$ is a locally trivial bundle of groups.
\end{enumerate}
Moreover, in this case any multiplicative Ehresmann on $\K$ is complete. 
\end{proposition}


\begin{proof}
Suppose a multiplicative Ehresmann connection exists which is complete. Then if we fix $x_0\in M$ and take a ball $B_\eps(x_0)$ inside some chart, we define:
\[ \phi:B_\eps(x_0)\times \K_{x_0}\to \K,\quad (x,k)\mapsto \widetilde{\gamma}_x^k(1), \]
where $\widetilde{\gamma}_x^k:[0,1]\to \K$ denotes the horizontal lift through $k$ of the line segment $\gamma_x:[0,1]\to B_\eps(x_0)$ joining $x_0$ to $x$. Since the Ehresmann  connection is multiplicative, for each $x\in B_\eps(x_0)$ the map $k\mapsto \phi(x,k)$ is a group isomorphism, so $\phi$ defines a local trivialization around $x_0$.

For the converse, assume that $\K$ is locally trivial, and fix an open cover $\{U_i\}$ over which $\K$ trivializes. On each $U_i$ consider the trivial multiplicative Ehresmann connection (see the example from Subsection \ref{example:os:products}) with corresponding multiplicative connection 1-form $\al_i\in \Omega^1_\mult(\K|_{U_i};\ka)$. Then choosing a partition of unity $\{\chi_i\}$ subordinate to this cover, we obtain a multiplicative connection 1-form 
\[ \al=\sum_i\chi_i\al_i\in \Omega^1_\mult(\K;\ka).\]

Next, we show that any multiplicative Ehresmann connection  restricted to the connected component of the identity, denoted $\K^0$, is complete. This will not use that $\K$ is locally trivial. 
Let $\gamma:[0,1]\to M$ be a path. The lift at $1_{\gamma(0)}$ is defined as $1_{\gamma(t)}$ for all $t\in [0,1]$. Therefore there is some neighborhood $U\in \K^0_{\gamma(0)}$ such that for any $g\in U$ the horizontal lift $\widetilde{\gamma}^{g}$ is defined on $[0,1]$. By a well-known result in Lie theory, we have:
\[ \K^0_{\gamma(0)}=\bigcup_{n\ge 1} U^n. \]
Hence, given any $g\in \K^0_{\gamma(0)}$ we can factor it as a product:
\[ g=g_1\cdots g_N, \quad g_i\in U. \]
Then the horizontal lift of $\gamma$ through $g$ is defined on $[0,1]$ and is given by: 
\[
\widetilde{\gamma}^{g}(t):=\widetilde{\gamma}^{g_1}(t)\cdots \widetilde{\gamma}^{g_N}(t),\]
(see item (ii) in the proof of Proposition \ref{prop:complete:Ehresmann:connection}). This proves the claim for $\K^0$.

For the general case, since $\K$ is locally trivial bundle of Lie groups, it follows that $\K/\K^0\to M$ is a covering map. Hence any path $\gamma:[0,1]\to M$ has a complete lift $\widetilde{\gamma}^{[g]}:[0,1]\to \K/\K^0$, for any initial value $[g]\in \K/\K^0|_{\gamma(0)}$. Next, we apply Proposition \ref{prop:complete:Ehresmann:connection} to the groupoid homomorphism:
\[\Phi:\K\to \K/\K^0,\]
with kernel $\K^0$. Then we know that $\widetilde{\gamma}^{[g]}$ lifts to a horizontal path $\widetilde{\gamma}^{g}:[0,1]\to \K$, with any starting condition $g\in \K_{\gamma(0)}$. This proves completeness.

Note that the specific form of the Ehresmann connection did not play a role, and therefore we obtain the last claim of the statement.
\end{proof}

From the proof, we deduce:
\begin{corollary}
\label{cor:connected:Lie:groups}
A bundle of connected Lie groups is locally trivial if and only if it admits a multiplicative Ehresmann connection. Any such connection is complete. 
\end{corollary}

Note also that there are bundle of groups which are not locally trivial and admit multiplicative Ehresmann connections (necessarily non-complete). For example, bundles of discrete groups admit unique multiplicative Ehresmann connections. 

\begin{remark}
In \cite[Prop 4.4]{LSX09} the authors show that connections for a Lie groupoid extension are always complete. This also follows from our results. Indeed, the proof of our Proposition \ref{prop:complete:Ehresmann:connection} shows that a multiplicative connection on a bundle of Lie groups which is a trivial fibration must be complete: for completeness all that is used in the proof is that $\K/\K^0\to M$ is a covering map, and this follows if $\K\to M$ is a locally trivial fibration.
\end{remark}

\subsection{Curvature}
\label{sec:curvature}
In this section we fix a partially split bundle of ideals $\ka$ for $\G$ with a multiplicative Ehresmann connection $E\subset T\G$. We denote by
$h:T\G\to E$ the horizontal projection relative the decomposition
\[ T\G=K\oplus E.\]
We denote by $\nabla:=\nabla^E$ the induced linear connection on the bundle $\ka\to M$ given by Corollary \ref{cor:connection:partial:split}, and by $\nabla^{\s}:=\s^*\nabla$ the induced connection on $\s^*\ka=\G\times_M\ka\to \G$. This allows us to differentiate forms in $\G$ with coefficients in $\ka$ (see Section \ref{sec:differentiation:forms:coefficients}).
We introduce the {\bf exterior covariant derivative} of $E$ as the operator $\D:\Omega^k(\G,\ka)\to \Omega^{k+1}(\G,\ka)$ given by
\[ 
(\D\omega)(X_1,\dots,X_{k+1}):=(\d^{\nabla}\omega)(h(X_1),\dots,h(X_{k+1})).
\]

\begin{definition}
The {\bf curvature} of the multiplicative Ehresmann connection $E$ is the 2-form:
\[ \Omega:=\D\al\in\Omega^2(\G,\ka), \]
where $\al\in\Omega^1_{\mult}(\G,\ka)$ is the multiplicative connection 1-form of $E$.
\end{definition}

The next properties of the curvature are reminiscent of properties of principal connections. In fact, as we will see in Subsection \ref{ex:bundle:ideals:transitive}, principal connections can be seen as examples of multiplicative Ehresmann connections.

\begin{proposition}
\label{prop:curvature:multiplicative}
The curvature of a multiplicative Ehresmann connection $E\subset T\G$ is a multiplicative form $\Omega\in\Omega^2_\mult(\G,\ka)$, which for horizontal vector fields $X,Y\in\Gamma(E)$ is given by:
\begin{equation}
    \label{eq:curvature:bracket}
    \d L_g\cdot \Omega(X,Y)_g=\big(h([X,Y])-[X,Y]\big)_g.
\end{equation}
In particular, if $\ka$ is induced by a Lie groupoid submersion $\Phi:\G\to\H$, and $\widetilde{U},\widetilde{V}\in \X(\G)$ are the horizontal lifts of $U,V\in \X(\H)$, then \[ \d L_{g}\cdot \Omega(\tilde{U},\tilde{V})_g=\big(\widetilde{[U,V]}-[\widetilde{U},\widetilde{V}]\big)_{g}.\]
\end{proposition}


\begin{corollary}
A multiplicative Ehresmann connection is involutive if and only if its curvature 2-form vanishes identically.
\end{corollary}

\begin{proof}
Let $\nabla^\s$ be the pullback of the connection $\nabla$ to $\s^*\ka$. The multiplicative connection 1-form $\alpha$ vanishes on horizontal vectors. Hence, by applying the definition of $\d^{\nabla}$, we find for any horizontal vector fields $X,Y\in\Gamma(E)\subset \X(\G)$:
\begin{align*}
    \Omega(X,Y)&=\d^{\nabla}\al(X,Y)\\
        &=\nabla^\s_{X}(\al(Y))-\nabla^\s_{Y}(\al(X))-\al([X,Y])\\
        &=\al(h([X,Y])-[X,Y]).
\end{align*}
Since $h([X,Y])-[X,Y]\in K$, the definition of $\al$ implies \eqref{eq:curvature:bracket}.

Let $\bO\in \Omega^2(\G,K)$ be the $K$-valued 2-form on $\G$ defined by
\[ \bO(X,Y)_g:=\d L_g\cdot  \Omega(X,Y)_g. \]
By the first part, if $X,Y\in\X(\G)$ are horizontal vector fields, we have:
\begin{equation}
    \label{eq:aux:multiplicative:curvature:1}
    \bO(X,Y)=h([X,Y])-[X,Y].
\end{equation} 
On the other hand, the multiplicativity of $\Omega$ is equivalent to:
\begin{align}
    \label{eq:aux:multiplicative:curvature:2}
    \bO(\d m(v_1,v_2),\d m(w_1,w_2))
    &=\d R_{g_2}\cdot\bO(v_1,w_1)+\d L_{g_1}\cdot\bO(v_2,w_2)\notag\\
    &=\d m(\bO(v_1,w_1),\bO(v_2,w_2)),
\end{align}
for any pair of composable arrows $(g_1,g_2)\in\G^{(2)}$ and any two pairs of composable tangent vectors $(v_1,v_2),(w_1,w_2)\in T_{(g_1,g_2)}\G^{(2)}$. Decomposing each tangent vector into horizontal and vertical component, and using the fact that $\bO(v,w)$ vanishes if either $v$ or $w$ is vertical, we see that \eqref{eq:aux:multiplicative:curvature:2} holds if and only if holds for horizontal tangent vectors. 

Now observe that, since multiplication in a groupoid is a submersion, given any composable tangent vectors $(v,w)\in T_{(g_1,g_2)}\G^{(2)}$, we can find (local) vector fields $V,W,V_i,W_i\in\X(G)$ such that 
\begin{align*}
    &V=m_*(V_1,V_2),  & &W=m_*(W_1,W_2),\\
    &(v_1,v_2)=(V_1,V_2)|_{(g_1,g_2)}, &  &(w_1,w_2)=(W_1,W_2)|_{(g_1,g_2)}.
\end{align*}
Moreover, $h:T\G\to E$ is a groupoid morphism, so it follows that:
\[ h(V)=m_*(h(V_1),h(V_2)),\quad h(W)=m_*(h(W_1),h(W_2)). \]
Hence, to prove the multiplicativity property \eqref{eq:aux:multiplicative:curvature:2}, it is enough to prove that for any \emph{horizontal} (local) vector fields $V,W,V_i,W_i\in\X(G)$ such that $V=m_*(V_1,V_2)$ and 
$W=m_*(W_1,W_2)$ one must have:
\[
\bO(m_*(V_1,V_2),m_*(W_1,W_2))=m_*(\bO(V_1,W_1),\bO(V_2,W_2)).
\]
Since all vector fields are horizontal we can use \eqref{eq:aux:multiplicative:curvature:1}, and we find:
\begin{align*}
    \bO(m_*(V_1,V_2),&m_*(W_1,W_2))=\\
    =&h([m_*(V_1,V_2),m_*(W_1,W_2)])-[m_*(V_1,V_2),m_*(W_1,W_2)]\\
    =&h(m_*([(V_1,V_2),(W_1,W_2)]))-m_*([(V_1,V_2),(W_1,W_2)])\\
    =&m_*(h([V_1,W_1]),h([V_2,W_2]))-m_*([V_1,W_1],[V_2,W_2])\\
    =&m_*(\bO(V_1,W_1),\bO(V_2,W_2)),
\end{align*}
as desired.
\end{proof}

To list further properties of the curvature, given $\ka$-valued  forms $\be\in\Omega^k(\G,\ka)$ and $\ga\in\Omega^l(\G,\ka)$ we will denote by $[\be,\ga]_\ka$ the $\ka$-valued form of degree $k+l$ given by:
\[ [\be,\ga]_\ka(X_1,\dots,X_{k+l})=
\sum_{\sigma\in S_{k+l}}(-1)^{|\sigma|}[\be(X_{\sigma(1)},\dots,X_{\sigma(k)}),\ga(X_{\sigma(k+1)},\dots,X_{\sigma(k+l)})]_\ka.\]
We have the following suggestive result:

\begin{proposition}
The curvature $\Omega$ of a multiplicative Ehresmann connection with multiplicative connection 1-form $\al$ satisfies:
\begin{enumerate}[(i)]
    \item Structure equation: $\Omega=\d^\nabla\al+\frac{1}{2}[\al,\al]_\ka$;
    \item Bianchi's identity: $\D\Omega=0$.
\end{enumerate}
\end{proposition}

\begin{proof}
To prove (i), observe that from the definitions: 
\[ 
\big(\d^\nabla\al+\tfrac{1}{2}[\al,\al]_\ka\big)(X,Y)
=\nabla^{\s}_X\al(Y)-\nabla^{\s}_Y\al(X)-\al([X,Y])+[\al(X),\al(Y)]_\ka. 
\]
We consider various cases:
\smallskip

$\bullet$  \emph{$X=\xi^L$ and $Y=\eta^L$,   $\xi,\eta\in \Gamma(\ka)$, are  vertical and left invariant:} then $\Omega(X,Y)=0$. On the other hand, $\xi=\alpha(\xi^L)|_M$, $\eta=\alpha(\eta^L)|_M$ and:
    \[\al(\xi^L)=s^*\xi,\quad \al(\eta^L)=s^*\eta,\quad s_*\xi^L=s_*\eta^L=0. \] 
    Hence, $\nabla^{\s}_X\al(Y)=\nabla^{\s}_Y\al(X)=0$, and $[X,Y]=[\xi,\eta]^L$, so 
    \[ \al([X,Y])=\s^*[\xi,\eta]_{\ka}=[\al(X),\al(Y)]_\ka. \]
    Hence:
    \[
    \big(\d^\nabla\al+\tfrac{1}{2}[\al,\al]_\ka\big)(X,Y)=-\al([X,Y])+\al([X,Y])=0=\Omega(X,Y).
    \]

$\bullet$ \emph{$X$ and $Y$ horizontal:} then $\al(X)=\al(Y)=0$ and we find:
        \[
        \big(\d^\nabla\al+\tfrac{1}{2}[\al,\al]_\ka\big)(X,Y)=
        -\al([Y,X])=\Omega(X,Y),
        \]
    where the last identity follows from the previous proposition.
\smallskip

$\bullet$ \emph{$X=\xi^L$, $\xi\in\Gamma(\ka)$, vertical and left invariant, and $Y$ horizontal and $\s$-projectable to $U=\s_*(Y)\in \X(M)$:} Then $\Omega(X,Y)=0$ and $\al(Y)=0$. We claim that the following relation holds: 
    \begin{equation}\label{eq:connection:between:connections}
    \nabla^{\s}_Y\al(X)=\al([Y,X]),
    \end{equation}
    which implies (i) in this case: 
    \[
        \big(\d^\nabla\al+\tfrac{1}{2}[\al,\al]_\ka\big)(X,Y)=
        -\al([Y,X])+\al([X,Y])=0=\Omega(X,Y).
        \]
To prove \eqref{eq:connection:between:connections}, we calculate Lie bracket using the commutator of the flows:
\[ [X,Y]_g=\frac{\d}{\d \eps}\Big|_{\eps=0}\frac{\d}{\d t}\Big|_{t=0}c_{t,\eps}(g), \]
where:
\[c_{t,\eps}(g):=\phi_{\xi^L}^{-\eps}\circ \phi_Y^{-t}\circ \phi_{\xi^L}^{\eps}\circ \phi_Y^{t}(g)=\phi_Y^{-t}\big(\phi_Y^t(g)\cdot \exp(\eps \xi|_{\phi_U^t(x)})\big)\cdot \exp(-\eps \xi|_x),\]
with $x=\s(g)$ and $\exp:\ka\to \G$ the exponential map of the bundle of Lie algebras $\ka$. Note that in the non-Hausdorff case, the flows might not be unique - however, they do exist in local charts, and this suffices to prove \eqref{eq:connection:between:connections}. We find first:
\begin{align*}
\frac{\d}{\d t}\Big|_{t=0}c_{t,\eps}(g)=
\d R_{k_{\eps}^{-1}}\Big(-Y|_{g\cdot k_{\eps}}+\d m(Y|_g,\frac{\d }{\d t}\Big|_{t=0}\exp(\eps \xi|_{\phi_U^t(x)}))\Big)\in T_{g}\G,
\end{align*}
where $k_{\eps}:=\exp(\eps \xi|_x)$ and $R_{k}$ denotes right translation by $k$. Using that $\alpha$ is multiplicative and vanishes on $E$, we then have: 
\begin{align*}
\al\Big(
\frac{\d}{\d t}\Big|_{t=0}c_{t,\eps}(g)\Big)&=\Ad_{k_{\eps}}\circ\, \alpha \Big(\d m\big(Y|_g,\frac{\d }{\d t}\Big|_{t=0}\exp(\eps \xi|_{\phi_U^t(x)})\big)\Big)\\
&=\Ad_{k_{\eps}}\circ\, \alpha\Big( \frac{\d }{\d t}\Big|_{t=0}\exp(\eps \xi|_{\phi_U^t(x)})\Big).
\end{align*}
The exponential map factors as $\exp=i\circ \exp^{\ka}$, where $\exp^{\ka}:\ka\to \G(\ka)$ is the exponential map of the bundle of 1-connected Lie groups $\G(\ka)$ integrating $\ka$, and  $i:\G(\ka)\hookrightarrow \G$ the induced immersion. By Corollary \ref{cor:connection:partial:split}, the linear connection $\nabla$ on $\ka$ is obtained by differentiating the Ehresmann connection $\d i^{-1}(E)\subset T\G(\ka)$ and, by Proposition \ref{prop:connection:partial:split}, $\exp^{\ka}:\ka\to \G(\ka)$ is connection preserving around the zero-section. So we have the following decomposition into horizontal and vertical component:
\[\frac{\d }{\d t}\Big|_{t=0}\exp\big(\eps \xi|_{\phi_U^t(x)}\big)
= \widetilde{U}_{k_{\eps}}\oplus \frac{\d}{\d t}\Big|_{t=0}\exp\big(\eps(\xi+t\nabla_U\xi)|_x\big)\in E_{k_{\eps}}\oplus K_{k_{\eps}},\]
for some $\widetilde{U}_{k_{\eps}}\in E_{k_{\eps}}$ that projects to $U|_{x}=\d \s(Y|_g)$. Therefore, we have that:
\begin{align*}
\al\Big(
\frac{\d}{\d t}\Big|_{t=0}c_{t,\eps}(g)\Big)&=
\frac{\d}{\d t}\Big|_{t=0}\exp\big(\eps(\xi+t\nabla_U\xi)|_x\big)\exp\big(-\eps\xi|_x\big)\in \ka_x.
\end{align*}
and so we conclude that:
\[
\al([X,Y]_g)=\al\Big(\frac{\d}{\d \eps}\Big|_{\eps=0}\frac{\d}{\d t}\Big|_{t=0}c_{t,\eps}(g)\Big|_{t=\eps=0}\Big)=(\nabla_U\xi)\big|_x.\]
But then:
\[ \big(\nabla^{\s}_Y\alpha(X)\big) \big|_g=\s^*(\nabla_{\s_*(Y)}\xi)|_{g}=\nabla_U\xi|_{x}=\alpha([Y,X]|_g). \]
so we obtained \eqref{eq:connection:between:connections}.
\smallskip

Since the vectors fields of the types considered in the cases above span all tangent vectors, item (i) is proven.
\smallskip


To prove (ii), note that $\D\Omega(X,Y,Z)$ vanishes if one of the vectors is vertical. So we can assume that $X,Y,Z\in \X(\G)$ are horizontal vector fields, and we find:
\begin{align*}
    \D\Omega(X,Y,Z)&=\d^\nabla\Omega(X,Y,Z)\\
    &=(\d^\nabla)^2\al(X,Y,Z)+[\d^\nabla\al,\al]_\ka(X,Y,Z)\\
    &=R^\nabla(X,Y)\al(Z)+R^\nabla(Z,X)\al(Y)+R^\nabla(Y,Z)\al(X)=0,
\end{align*}
since $\al(X)=\al(Y)=\al(Z)=0$.
\end{proof}

\begin{remark}
For connections on a Lie groupoid extension $\Phi:\G\to\H$ the authors of \cite{LSX09} also introduce the curvature 2-form and prove analogues of the structure equations and Bianchi identity. See loc.\ cit.\ Section 4.6.
\end{remark}

Consider now a surjective, submersive, Lie groupoid morphism $\Phi:\G\to\H$. Given a multiplicative Ehresmann connection $E$ for $\Phi$ we have the induced connection $E^\K$ on the kernel $p:\K\to M$ and the linear connection $\nabla:=\nabla^E$ on $\ka=\ker\d p|_M$. The curvature of $E^\K$ is the multiplicative 2-form $i^*\Omega$, where $i:\K\to\G$ is the inclusion and $\Omega$ is the curvature of $E$. By Proposition \ref{prop:curvature:multiplicative}, this form is multiplicative.

\begin{proposition}
\label{prop:flat:kernel:connection}
The IM 2-form associated with  $i^*\Omega\in\Omega^2_\mult(\K,\ka)$ is given by $(R^\nabla,0)\in\Omega^2_\imult(\ka,\ka)$. In particular, if $E^{\K}$ is involutive then $R^\nabla=0$, and the converse holds if $\K$ has connected fibers.
\end{proposition}

\begin{proof}
Denote the horizontal lift of a vector field $X\in \X(M)$ to $\K$ by $\widetilde{X}\in \Gamma(E^{\K})$. Then Proposition \ref{prop:curvature:multiplicative}, implies that
\[(i^*\Omega)(\widetilde{X},\widetilde{Y})_k=\d L_{k^{-1}}\big(\widetilde{[X,Y]}-[\widetilde{X},\widetilde{Y}]\big)\in \ka_{p(k)}.\]
To find the IM form corresponding to $i^*\Omega$ we apply \eqref{eq:M:IM:E:forms} and \eqref{eq:M:IM:E:forms:2}. For the symbol we obtain, for $\xi\in\Gamma(\ka)$,
\[ l(\xi)=(i_{\xi^L}i^*\Omega)|_{TM}=0,\]
since $\xi^L\in\X(\K)$ is a vertical vector field. 
On the other hand, the flow of $\xi^L$ is given by $\phi^t_{\xi^L}(k)=k\exp(t\xi)$, so for $X,Y\in\X(M)$ we find
\begin{align*}
(i^*\Omega)(\d \phi_{\al^{L}}^{t}(X),\d \phi_{\al^{L}}^{t}(Y))
&=(i^*\Omega)(\d R_{\exp(t\xi)}(X),\d R_{\exp(t\xi)}(Y))\\
&=(i^*\Omega)(h(\d R_{\exp(t\xi)}(X)),h(\d R_{\exp(t\xi)}(Y)))\\
&=(i^*\Omega)(\widetilde{X},\widetilde{Y})_{\exp(t\xi)}.
\end{align*}
Then \eqref{eq:M:IM:E:forms:2} reduces to:
\[
L(\xi)(X,Y)=\frac{\d}{\d t}\Big|_{t=0} \d R_{\exp(-t\xi)}\big(\widetilde{[X,Y]}-[\widetilde{X},\widetilde{Y}]\big)_{_{\exp(t\xi)}}.
\]
The exponential $\exp:\ka\to \K$ maps the $\nabla$-horizontal lift $\overline{X}\in\X(\ka)$ to the $E^{\K}$-horizontal lift $\widetilde{X}\in\X(\K)$ (see Proposition \ref{prop:connection:partial:split}). 
It follows then that:
\[
    L(\xi)(X,Y)=\frac{\d}{\d t}\Big|_{t=0} \big(\overline{[X,Y]}-[\overline{X},\overline{Y}]\big)_{t\xi}=R^\nabla(X,Y)(\xi).
\]
The multiplicative 2-form $i^*(\Omega)$ vanishes on the bundle $\K^{0}$ of connected components of the identities if and only if the corresponding IM 2-form $(R^{\nabla},0)$ vanishes. This implies the last part of the statement. 
\end{proof}

\subsection{Flat multiplicative connections}
Given a multiplicative Ehresmann connections $E$ for a morphism $\Phi:\G\to\H$, there are various levels of flatness one can require.

We start by looking at the case where the induced Ehresmann connection $E^{\K}$ is involutive, and hence, by Proposition \ref{prop:flat:kernel:connection}, the linear connection $\nabla$ on $\ka$ is flat.


\begin{definition}
A multiplicative Ehresmann connection $E$ for a groupoid morphism $\Phi:\G\to\H$ is called {\bf kernel flat} if $E^\K$ is involutive.
\end{definition}

In order to state the basic properties of kernel flat connections, let us make the following general observation about $\Phi:\G\to\H$. The $\G$-action on $\ka$ restricts to a $\K$-action on $\ka$ which is just the fiberwise adjoint action. Hence, its fixed point set is a collection of vector subspaces of the center:
\[ \ka^\K\subset z(\ka). \]
In general, this fails to be a vector subbundle. It is one if either:
\begin{enumerate}[(a)]
    \item $\ka$ is locally trivial and $\K$ has connected fibers, in which case $\ka^\K=z(\ka)$, or
    \item $\K$ is locally trivial, in which case one can have $\ka^\K\subsetneq z(\ka)$.
\end{enumerate}
The $\G$-action preserves $\ka^\K$ and descends to an $\H$-action on $\ka^\K$. So when $\ka^\K$ is subbundle we obtain a representation of $\H$. 

\begin{lemma}
If $\ka^\K$ is a subbundle then the linear connection $\nabla$ restricts on $\ka^\K$ to an $\H$-invariant connection.
\end{lemma}

\begin{proof}
Consider the center of the bundle of Lie groups $p:\K\to M$, denoted
\[ 
Z(\K):=\{z\in\K: kzk^{-1}=z,\ \forall\ k\in p^{-1}(p(z))\},
\]
and its connected components of the identity, denoted $Z(\K)^{\circ}$. Under the assumption that $\ka^\K$ is a subbundle of $\ka$, this is a bundle of Lie groups integrating $\ka^\K$
\[ \ka^\K_x=T_{1_x} Z(\K)^{\circ},\quad  (x\in M). \]
We claim that for all $z\in Z(\K)^{\circ}$ we have:
\[ E^\K_z\subset T_zZ(\K)^{\circ}. \]
This claim implies that $\nabla$ restricts to a linear connection on $\ka^\K$. 

To prove the claim, we consider any smooth curve $z_t\in Z(\K)^{\circ}$ and show that, given the decomposition into vertical and horizontal vectors:
\[ \dot{z}_0=\dot{z}_0^h\oplus\dot{z}_0^v\in E^\K_{z_0}\oplus K_{z_0},\]
we have $\dot{z}_0^v\in T_{z_0}Z(\K)^{\circ}$. Then $\dot{z}_0^h\in T_{z_0}Z(\K)^{\circ}$ and the claim follows. Let $k\in\K_{p(z_0)}$ and $k_t$ be the parallel transport of $k$ along the curve $p(z_t)$. Then:
\[ z_t=k_tz_tk^{-1}_t. \]
Differentiating both sides at $t=0$ and decomposing into horizontal and vertical components, using the multiplicativity of $E^\K$, we obtain:
\[ \dot{z}_0^v=0_{k}\cdot\dot{z}_0^v\cdot 0_{k^{-1}}, \]
where multiplication and inverse are the operations in the tangent group $T(\K_{p(z_0)})$. Writing $\dot{z}_0^v=0_{z_0}\cdot w$, with $w\in \ka_{p(z_0)}$, and using that $z_0\in Z(\K)$, we obtain: 
\[w=\Ad_k(w),\quad \forall\, k\in \K.\]
Hence $w\in \ka^{\K}$, and so $\dot{z}_0^v\in T_{z_0}Z(\K)^{\circ}$. 

We now show that $\nabla$ restricts on $\ka^\K$ to a $\H$-invariant connection. For this, it is enough to show that if $z_t\in Z(\K)^{\circ}$ is a horizontal curve and $h_t\in \H$ is any curve with $\s(h_t)=p(z_t)$ then $h_t\cdot z_t$ is a horizontal curve. For this observe that 
\[ h_t\cdot z_t=\widetilde{h}_t\cdot z_t, \]
where $\widetilde{h}_t\in\G$ is any curve with $\Phi(\widetilde{h}_t)=h_t$. In particular, we can take $\widetilde{h}_t$ to be the horizontal lift of $h_t$ and multiplicativity of $E$ gives that $h_t\cdot z_t$ is horizontal.
\end{proof}

Let us now assume that $E$ is kernel flat:

\begin{proposition}
\label{prop:curvature:basic}
Let $E$ be a kernel flat multiplicative Ehresmann connection for a morphism $\Phi:\G\to\H$. Then:
\begin{enumerate}[(i)]
    \item The curvature $\Omega$ takes values in $\ka^\K$;
    \item If $\ka^\K$ is a vector bundle, there is a multiplicative form $\underline{\Omega}\in\Omega^2_\mult(\H,\ka^\K)$ such that:
    \[ \Omega=\Phi^*\underline{\Omega},\]
    and which is $\d^{\nabla}$-closed.
\end{enumerate}
\end{proposition}

\begin{proof}
Let $g_1,g_2\in \G$ such that $\Phi(g_1)=h=\Phi(g_2)$. Let $X,Y\in T_h\H$, and consider their horizontal lifts $\widetilde{X}_{g_i},\widetilde{Y}_{g_i}\in T_{g_i}\G$, for $i=1,2$. First, write $g_2=g_1k$, for $k\in \K$. Let $\overline{X}_k, \overline{Y}_k\in T_k\K $ be the horizontal lift of $\d_h\s X$ and $\d_h\s Y$ with respect to the connection $E^{\K}$. Since $E$ is a subgroupoid of $T\G$, we have that: 
\[\d m(\widetilde{X}_{g_1},\overline{X}_k)=\widetilde{X}_{g_2}, \quad \d m(\widetilde{Y}_{g_1},\overline{Y}_k)=\widetilde{Y}_{g_2}.\]
Since $\Omega$ is multiplicative, this yields:
\[\Omega(\widetilde{X}_{g_2},\widetilde{Y}_{g_2})=k^{-1}\cdot \Omega(\widetilde{X}_{g_1},\widetilde{Y}_{g_1})+\Omega(\overline{X}_k,\overline{Y}_k)=k^{-1}\cdot \Omega(\widetilde{X}_{g_1},\widetilde{Y}_{g_1}),\]
where in the last equation we used that $E^{\K}$ is flat, so $i^*\Omega=0$. Next, write $g_2=lg_1$, with $l\in \K$ and let $\overline{X}_l, \overline{Y}_l\in T_l\K $ be the horizontal lift of $\d_h\t (X)$ and $\d_h\t (Y)$ with respect to the connection $E^{\K}$. Then we also have that: 
\[\d m(\overline{X}_l,\widetilde{X}_{g_1})=\widetilde{X}_{g_2}, \quad \d m(\overline{Y}_l,\widetilde{Y}_{g_1})=\widetilde{Y}_{g_2},\]
and so, since $\Omega$ is multiplicative:
\[\Omega(\widetilde{X}_{g_2},\widetilde{Y}_{g_2})=g_1^{-1}\cdot \Omega(\overline{X}_{l},\overline{Y}_{l})+\Omega(\widetilde{X}_{g_1},\widetilde{Y}_{g_1})=\Omega(\widetilde{X}_{g_1},\widetilde{Y}_{g_1}).\]
Since $\Omega$ vanishes on vertical vectors, the obtained equations show that $\Omega$ takes values in $\ka^{\K}$. They also show that we obtain a well-defined $\ka^\K$-valued 2-form on $\H$, by setting for $X,Y\in T_h\H$:
\[ \underline{\Omega}(X,Y)=\Omega(\widetilde{X}_{g},\widetilde{Y}_{g}) \quad (g\in \Phi^{-1}(h)).\]
When $\ka^\K$ is a vector bundle the $\G$-action factors through the $\H$-action, and the multiplicativity of $\underline{\Omega}$ follows from the multiplicativity of $\Omega$.

The fact that $\underline{\Omega}$ is $\d^{\nabla}$-closed follows from the Bianchi identity for $\Omega$.
\end{proof}

\begin{remark}
Kernel flat connections play an important role in the theory of gerbes developed in \cite{LSX09}. There, for a groupoid extension, the authors obtain a version of the previous proposition (see loc.\ sit.\ Theorem 6.38).
\end{remark}

We now describe another type of flatness. Observe that, if $\Phi:\G\to\H$ is a surjective, submersive, morphism covering the identity, for each $x\in M$ the restriction:
    \[ \Phi_x:=\Phi|_{\s_{\G}^{-1}(x)}: \s_{\G}^{-1}(x)\to \s_{\H}^{-1}(x). \]
is a principal $\K_x$-bundle. Given a multiplicative Ehresmann connection $E$ for a $\Phi:\G\to\H$ one obtains an Ehresmann connection on the restriction:
\[ \ker\d_g\s_{\G}=K_g\oplus E_g^x,\quad E^x_g:=\ker\d_g\s_{\G}\cap E\quad (g\in \s_{\G}^{-1}(x)). \]

\begin{lemma}
For each $x\in M$, the restricted connection $E^x$ is a principal bundle connection.
\end{lemma}

\begin{proof}
Using the multiplicativity of $E$ one finds that if $k\in \K_x$ and $v\in E^x_g$, for $g\in \s_{\G}^{-1}(x)$, then:
\[ v\cdot k=\d_{(g,k)}m(v,0_k)\in E^x_{gk}. \]
Hence, $E^x$ is $\K_x$-invariant.
\end{proof}

If $\al\in\Omega^1_\mult(\G,\ka)$ and $\Omega \in\Omega^2_\mult(\G,\ka)$ are the multiplicative connection  1-form and curvature form of $E$ then the principal bundle connection $E^x$ has connection 1-form $\al^x\in \Omega^1(\s_{\G}^{-1}(x),\ka_x)$ and curvature form $\Omega^x\in \Omega^2(\s_{\G}^{-1}(x),\ka_x)$ the restrictions:
\[ \al^x=i^*_x\al,\quad \Omega^x=i^*_x\Omega, \]
where $i_x:\s_{\G}^{-1}(x)\hookrightarrow\G$ is the inclusion.

\begin{definition}
A multiplicative Ehresmann connection $E$ for a groupoid morphism $\Phi:\G\to\H$ is called {\bf leafwise flat} if $\Omega^x=0$ for every $x\in M$.
\end{definition}

We have the following characterizations:
\begin{lemma}
Given a multiplicative Ehresmann connection $E$ for a groupoid morphism $\Phi:\G\to\H$ the following are equivalent:
\begin{enumerate}[(i)]
    \item $E$ is leafwise flat;
    \item $\ker \d \s\cap E\subset T\G$ is involutive;
    \item $B':=(\ker \d \s\cap E)|_M\subset A$ is a Lie subalgebroid.
\end{enumerate}
\end{lemma}
\begin{proof}
Clearly, (i) and (ii) are equivalent. Equivalence between (ii) and (iii) uses that $\ker \d \s\cap E$ is spanned by right invariant extensions of elements in $B'$. 
\end{proof}

\begin{remark}
The pair $(\G,E)$ is an example of a Pfaffian groupoid (see \cite{MSalazarThesis}).
\end{remark}

The lemma shows that if $\Phi:\G\to \H$ admits a leafwise flat multiplicative Ehresmann connection, then the induced Lie algebroid map $\phi:A\to B$ admits a Lie algebroid splitting with image $B'$. We briefly recall the groupoid version of this condition. 

Let $\H\tto M$ be a Lie groupoid which acts on a bundle of Lie groups $p:\K\to M$ preserving the group structure on the fibers. The \textbf{semi-direct product} is the Lie groupoid $\G:=\H\times_M\K\tto M$ with multiplication is given by:
\[ (h_1,k_1)(h_2,k_2)=(h_1 h_2,(h_2^{-1}\cdot k_1)k_2). \]
This fits into a short exact sequence of groupoids:
\[
\begin{tikzcd}
\K \ar[r, hook] & \H \times_M\K \ar[r, two heads, "\pr_1"]& \H
\end{tikzcd}
\]
which is canonically split by the groupoid morphism:
\[ \sigma:\H\longrightarrow \H\times_M \K, \quad h\mapsto (h,1_{\s(h)}). \]

Conversely, given a surjective, submersive, Lie groupoid map $\Phi:\G\to\H$ covering the identity, which admits a \textbf{Lie groupoid splitting}:
\[
\begin{tikzcd}
\K\ar[r, hook] & \G \ar[r, two heads, "\Phi"] & \H\ar[l, dashed, bend left, "\sigma"]
\end{tikzcd}
\]
we obtain an action of $\H\tto M$ on the bundle of Lie groups $p:\K\to M$:
\[ h\cdot k=\sigma(h)k\sigma(h)^{-1}, \]
preserving the group structure on the fibers, and a Lie groupoid isomorphism:
\[ \H\times_M \K \simeq \G,\quad (h,k)\mapsto \sigma(h)k. \]

Given a leafwise flat multiplicative Ehresmann $E$ for a groupoid map $\Phi:\G\to \H$, the induced Lie algebroid splitting integrates to a groupoid morphism on a possibly different Lie groupoid with the same Lie algebroid as $\H$. To explain this, let $\widetilde{\H}\tto M$ be the universal covering groupoid of $\H^0\tto M$, the connected component of the identity of $\H\tto M$. Recall that it consists of $\s$-leafwise path-homotopy classes of paths in $\H$ starting at the identity:
\[ \widetilde{\H}=\big\{[\gamma]~|~ \gamma:[0,1]\to \H,\, \s(\gamma(t))=x,\, \gamma(0)=1_x,\, x\in M\big\}. \]

Recall that principal bundle connections are always complete, so the horizontal lifts of $\s$-leafwise paths are always defined, and we can set:

\begin{definition}\label{def:holonomy:map}
Given a leafwise flat, multiplicative, Ehresmann connection $E$ for $\Phi:\G\to\H$ the {\bf holonomy} of $E$ is the map:
\[ \Hol:\widetilde{\H}\to \G,\quad [\gamma]\mapsto \widetilde{\gamma}(1), \]
where $\widetilde{\gamma}:[0,1]\to \G$ is the horizontal lift of $\gamma$ starting at $1_{\s(\gamma(0))}$.
\end{definition}

\begin{proposition}
The holonomy of a leafwise flat, multiplicative Ehresmann connection $E$ for $\Phi:\G\to\H$ is a Lie groupoid morphism making the following diagram commute:
\[ \xymatrix@R=15pt{ & \widetilde{\H}\ar[dl]_{\Hol}\ar[d]\\ \G\ar[r]_{\Phi} & \H}\]
\end{proposition}

\begin{proof}
The smoothness of $\Hol$ follows from smoothness of flows of horizontal vector fields. Clearly, from the definition $\Hol(1_x)=1_x$. The fact that $\Phi$ preserves multiplication follows from the fact that products of horizontal lifts are horizontal and the definition of multiplication on $\widetilde{\H}$.
\end{proof}

Notice that the kernel of $\Hol:\widetilde{\H}\to \G$ is contained in the kernel of the covering map $\widetilde{\H}\to\H$. 
We will say that a leafwise flat, multiplicative Ehresmann connection $E$ for $\Phi:\G\to\H$ has {\bf trivial holonomy} if the two kernels coincide. Then we find:

\begin{corollary}
\label{cor:flat:splitting}
Let $\Phi:\G\to \H$ be a morphism covering the identity and assume $\H$ is source connected. If $\Phi:\G\to \H$ admits a leafwise flat, multiplicative, Ehresmann connection with trivial holonomy then it admits a Lie groupoid splitting.
\end{corollary}

\begin{proof}
If the holonomy is trivial, then the holonomy morphism $\Hol:\widetilde{\H}\to \G$ descends to a groupoid splitting $\sigma:\H\to\G$.
\end{proof}

We will see later (Corollary \ref{cor:proper:semi-direct}) that when $\G$ is proper and $\ker \Phi$ is abelian, then the converse holds: if $\Phi:\G\to \H$ admits a Lie groupoid splitting then it admits a leafwise flat, multiplicative Ehresmann connection with trivial holonomy.

Next, we look for the conditions which ensure that a semi-direct product admits a leafwise flat, multiplicative Ehresmann connection. For that, we introduce an extension of Definition \ref{definition:invariant:connection} to non-linear connections.

\begin{definition}\label{definition:invariant:Ehresmann:connection}
Let $\H\tto M$ be a Lie groupoid which acts on a submersion $p:N\to M$. An Ehresmann connection $E$ on $p:N\to M$ is called {\bf $\H$-invariant} if the diffeomorphism:
\[ 
\vcenter{\xymatrix@C=15pt@R=15pt{\H\timessp N\ar[rr]^{\Psi}\ar[dr] & & \H\timestp N\ar[dl]\\ & \H }}
\qquad (g,p)\mapsto (g,g\cdot p),
\]
induces an isomorphism $\d\Psi:(\d\s)^*E\diffto (\d\t)^*E$.
\end{definition}

Note that in the distributions $(\d\s)^*E$ and $(\d\t)^*E$ appearing in this definition are Ehresmann connections for the projections in the diagram. We now have:

\begin{proposition}\label{prop:invariant:connection}
Let $\Phi:\G\to \H$ be a surjective, submersive groupoid map, and let $\K:=\ker\Phi$. The following are equivalent:
\begin{enumerate}[(i)]
    \item There is a multiplicative Ehresmann connection $E$ on $\G$ and $\Phi$ admits a groupoid splitting $\sigma:\H\to \G$ which is horizontal with respect to $E$.
    \item $\G$ is isomorphic to a semidirect product $\G\simeq \H\times_M\K$, where $\K$ admits an $\H$-invariant multiplicative connection.
\end{enumerate}
If these hold, then $\Phi:\G\to \H$ admits a leafwise flat multiplicative Ehresmann connection with trivial holonomy.
\end{proposition}

\begin{proof}
We first prove (i) $\Rightarrow$ (ii). For this, we use $\sigma$ to identify $\G\simeq \H\times_M\K$. Then any path of the form $t\mapsto(h_t,1_{\s(h_t)})$ is $E$-horizontal. We claim that the multiplicative connection $E^\K=T\K\cap E$ on $\K$ is $\H$-invariant. To see this let $h:[0,1]\to \H$ be any path, and let $k:[0,1]\to \K$ be an $E^\K$-horizontal path covering $t\mapsto\s(h_t)$. Then $t\mapsto (h_t,k_t)\in \H\times_M \K$ is horizontal for the connection $(\d\s)^*E^\K$. Moreover, all $(\d\s)^*E^\K$-horizontal paths are of this form. It suffices to show that the pair $t\mapsto (h_t,h_t\cdot k_t)$ is horizontal for $(\d\t)^*E^\K$, or equivalently, that $t\mapsto h_t\cdot k_t$ is horizontal for $E^\K$. This follows because $t\mapsto(h_t,1_{\s(h_t)})$ and $t\mapsto(1_{\s(h_t)},k_t)$ are $E$-horizontal, $E$ is multiplicative, and so also the following product is $E$-horizontal:
\[t\mapsto (h_t,1_{\s(h_t)})\cdot (1_{\s(h_t)},k_t) \cdot (h_t,1_{\s(h_t)})^{-1}=(1_{\t(h_t)},h_t\cdot k_t).\]

Conversely, to prove that (ii) $\Rightarrow$ (i). If $E^\K$ is an $\H$-invariant multiplicative Ehresmann connection on $p:\K\to M$, then $E=(\d\s)^*E^\K$ is an Ehresmann connection on $\H\times_M \K$ for which a path $(h_t,k_t)\in \H\times_M \K$ is $E$-horizontal if and only if $k_t$ is $E^\K$-horizontal. In particular, the splitting $h\to (h,1_{\s(h)})$ is horizontal. To show that $E$ is multiplicative it suffices to show that for any pair of paths $(h_t,k_t)$ and $(\tilde{h}_t,\tilde{k}_t)$ in $\H\times_M \K$ that are pointwise composable and $E$-horizontal, their product is also $E$-horizontal. But $E^\K$ is $\H$-invariant, so we have that $\tilde{h}_t^{-1}\cdot k_t$ is $E^\K$-horizontal. Since $E^{\K}$ is multiplicative, also $(\tilde{h}_t^{-1}\cdot k_t)\,\tilde{k}_t$ is $E^\K$-horizontal, or equivalently, $(h_t,k_t)\cdot(\tilde{h}_t,\tilde{k}_t)$ is $E$-horizontal. Since 

 For the last part of the proposition, notice that the connection $E=(\d\s)^*E^\K$ is leafwise flat with trivial holonomy.
 \end{proof}

For semi-direct products with abelian kernel we obtain: 

\begin{corollary}\label{corollary:semi-direct:abelian:kernel}
Let $\H\tto M$ be a Lie groupoid which acts on a bundle of abelian groups $p:A\to M$ preserving addition on the fibers. Then  $\pr_1:\H\times_M A\to \H$ admits a multiplicative Ehresmann connection if and only if $p:A\to M$ admits an $\H$-invariant multiplicative connection.
\end{corollary}

\begin{proof}
Assume that $\pr_1:\H\times_M A\to \H$ admits a multiplicative Ehresmann connection with connection 1-form $\alpha$. We can insure that $\sigma^*\al=0$, by replacing $\alpha$ by $\alpha - \pr_1^*\sigma^*(\alpha)$. This form is multiplicative because $A$ being abelian implies that the action of $\H\times_M A$ on $A$ is the pullback of the action of $\H$ on $A$. The corollary now follows from the proposition.
\end{proof}

For semi-direct products which have linear kernel, we obtain the following: 

\begin{corollary}\label{corollary:partial:split:for:semi-direct}
Given a Lie groupoid $\H\tto M$ and a representation of $\H$ on a vector bundle $p:V\to M$, the projection $\pr_1:\H\times_M V\to \H$ admits a multiplicative Ehresmann connection if and only if $V$ admits an $\H$-invariant connection $\nabla$.
\end{corollary}

\begin{proof}
Note that an Ehresmann connection $E^V\subset TV$ is multiplicative if and only if it defines a linear connection $\nabla$. Also note that, for linear connections on $V$, the two notions of $\H$-invariance from Definitions \ref{definition:invariant:Ehresmann:connection} and \ref{definition:invariant:connection} coincide. Hence, the corollary follows from the previous one.
\end{proof}

\medskip

Finally, we discuss the situation where a multiplicative Ehresmann connection $E$ for a morphism $\Phi:\G\to \H$ is involutive, in which case we say that $E$ is {\bf totally flat}. The following proposition shows that when the connection is complete then, up to a cover, $\G$ is a trivial product and $\Phi$ is the projection:

\begin{proposition}
\label{prop:totally:flat:grpd}
Let $\Phi:\G\to \H$ be a morphism which admits a complete, totally flat, multiplicative, Ehresmann connection, and assume that $\G$ is target connected. Then there is a covering space of the base $p:\widetilde{M}\to M$ and a commutative diagram of Lie groupoid morphisms
\[
\xymatrix@R=15pt@C=30pt{
p^*\widetilde{\H}\times G\ar[d]\ar[r]^--{\pr} & p^*\widetilde{\H}\ar[d]\\
\G\ar[r]_{\Phi} & \H
}
\]
where $G\simeq \K_x$ is a Lie group, the vertical maps are surjective local diffeomorphisms covering $p:\widetilde{M}\to M$, and the pullback of the Ehresmann connection under the morphism on the left is $T(p^*\widetilde{\H})\times 0_G$.
\end{proposition}




\begin{proof}
By completeness, the parallel transport exists along a path $h:[0,1]\to \H$: 
\[\tau_h:\Phi^{-1}(h(0))\diffto \Phi^{-1}(h(1)).\]
We note the following equalities, for any $h:[0,1]\to \H$, with $h(0)=1_x$ and $\s\circ h= x$:
\begin{equation}\label{eq:parllel:transports}
\tau_h(k)=\tau_h(1_x)k=\tau_{1_{\t\circ h}}(k)\tau_h(1_x), \quad (k\in \K_x)    
\end{equation}
These follow because $E$ is multiplicative, so all the paths above are $E$-horizontal, they cover $h$ and start at $k$. 

Since $E$ is leafwise flat, the holonomy map of Definition \ref{def:holonomy:map} is a groupoid map:
\[\Hol:\widetilde{\H}\to \G, \quad [h]\mapsto \tau_h(1_{x}), \quad  x=\s([h])\]
where $\widetilde{\H}$ is the target 1-connected cover of $\H$ - which has connected $\t$-fibers, because $\G$ has. We have an induced action of $\widetilde{\H}$ on $\K$ by group automorphisms, and so a surjective groupoid morphism 
\[\widetilde{\H}\times_M\K\to \G, \quad ([h],k)\mapsto \tau_h(1_x)k.\]
From \eqref{eq:parllel:transports} we deduce that the action of $\widetilde{\H}$ on $\K$ is given by:
\[[h]\cdot k := \tau_h(1_x)k\tau_h(1_x)^{-1}=\tau_{1_{\t\circ h}}(k), \quad k\in \K_x, \quad x=\s ([h]).
\]
In other words, $[h]$ acts via the parallel transport on ${\K}$ along its base path $[\t\circ h]$ with respect to the flat connection $E^{\K}$. Consider the holonomy cover $p:\widetilde{M}\to M$ for $E^\K$, i.e., the smallest cover where $E^\K$ becomes the trivial connection. Upon pullback to $\widetilde{M}$, $p^*\K=\widetilde{M}\times G$, where $G\simeq\K_x$, and the action of $p^*\widetilde{\H}$ on $p^*\K$ becomes the trivial action; hence we have a groupoid isomorphism: $p^*(\widetilde{\H}\times_M\K)\simeq p^*\widetilde{\H}\times G$. By comparing the parallel transport maps, one obtains the claim about the pullback of $E$ to $p^*\widetilde{\H}\times G$.
\end{proof}

\section{Examples and applications (groupoids)}
\label{examples:groupoids}

We discuss the partially split condition for several classes of groupoids with bundles of ideals.

\subsection{Lie groups}
\label{ex:Lie:groups}
Let $G$ be a connected Lie group with Lie algebra $\gg$. We claim that an ideal $\ka\subset \gg$ is partially split for $G$ if and only if we have a decomposition $\gg=\ka\oplus\mathfrak{h}$, for some ideal $\mathfrak{h}\subset \gg$. To see this, using left translations, we identify $TG$ with the semi-direct:
\[ TG\simeq G\ltimes \gg, \]
where $G$ acts on $\gg$ via the adjoint action. Under this isomorphism, $K\subset TG$ is identified with the subgroup
\[ K\simeq G\ltimes\ka. \]
Moreover, one checks easily that a distribution $E\subset TG$ is a subgroup if and only if it takes the form:
\[ E\simeq G\ltimes \mathfrak{h}, \]
where $\mathfrak{h}\subset \gg$ is an ideal. Hence, the claim follows. 

When $\mathfrak{k}$ integrates to a closed subgroup of $G$, we have a surjective Lie group morphism $\Phi:G\to H$ inducing $\ka$. The distributions $E\subset TG$, induced by ideals $\mathfrak{h}\subset \gg$ as above, are then multiplicative Ehresmann connections for $\Phi$, which are complete by Theorem \ref{thm:completeness}.


\subsection{Products}\label{example:os:products}
Let $\H\tto M$ be a Lie groupoid with Lie algebroid $B\Ato M$ and $G$ be a Lie group with Lie algebra $\gg$. The product
\[ \G:=\H \times G\tto M\]
comes with the bundle of ideals $\ka := 0_M\times \gg\subset B\times \gg$, which corresponds to the Lie groupoid morphism $\pr_1:\G\to \H$. This is partially split with complete, multiplicative  Ehresmann connection $E=T\H\times G\subset T(\H\times G)$. 

\subsection{Bundle of groups}
As we already pointed out before, a  bundle of groups $p:\G\to M$ can be thought of as a groupoid submersion onto the identity groupoid $M\tto M$. The corresponding bundle of ideals is its Lie algebroid $\ka=A$. By Corollary \ref{corollary:locally:trivial:groupoid}, if $\ka$ is partially split for $\G$, then $\ka$ must be locally trivial. The converse, in general, may fail, but it holds if $\G$ is assumed to be a bundle of simply connected Lie groups. This follows from Corollary \ref{cor:connected:Lie:groups}, since in this case $\G$ is locally trivial (it follows also from the global to infinitesimal correspondence given in Theorem \ref{thm:Lie:functor:connections}).




A simple example of a bundle of simply connected groups that is not locally trivial is $\G:=\R\times \R^2 \to \R$, with multiplication given at $x\in \R$ by:
\[ (u_1,v_1)\cdot(u_2,v_2)= (u_1+u_2,v_1+e^{xu_1}v_2).\]
Therefore, this does not admit multiplicative Ehresmann connections.

We note that, in the case of a vector bundle $\G=V\to M$, viewed as a bundle of Lie groups, a multiplicative Ehresmann connection is the same as a linear connection on $V$.

\subsection{Transitive groupoids}
\label{ex:bundle:ideals:transitive}
For a transitive Lie groupoid $\G\tto M$ the bundle of isotropy Lie algebras $\ka:= \ker \rho$ is a bundle of ideals, with 
\[K=\ker \d \s \cap \ker  \d \t \subset T\G.\] 
This is the bundle of ideals determined by the groupoid submersion 
\[ \Phi:=(\t, \s):\G\to M\times M \]
and it is always partially split. Indeed, recall that $\G$ can be identified with a gauge groupoid
\[\G(P):= P\times_GP\tto M,\] 
where, for a fixed $x\in M$, $G$ denotes the isotropy group $\G_{x}$, and $P:=\s^{-1}(x)$ is a principal $G$-bundle with projection $\t:P\to M$. Under this identification, the bundle $\ka$ coincides with the adjoint bundle:
\[ \ka=P[\gg]:=P\times_G \gg.\]
If we choose a principal bundle connection $\eta\in\Omega^1(P;\gg)$, the 1-form $\pr_2^*\eta-\pr_1^*\eta\in \Omega^1(P\times P;\gg)$ descends to a 1-form $\alpha\in\Omega^1(\G(P);P[\gg])$:
\[ q^*\al=\pr_2^*\eta-\pr_1^*\eta,\]
where $q:P\times P \to \G(P)$ is the projection. This gives a multiplicative connection 1-form $\alpha\in\Omega^1_\mult(\G(P);P[\gg])$, and so $\ka=P[\gg]$ is partially split. 

In fact, this assignment yields a 1-to-1
correspondence:
\[ 
\left\{\txt{principal bundle\\ connections $\eta\in\Omega^1(P;\gg)$\,} \right\}
\tilde{\longleftrightarrow}
\left\{\txt{multiplicative connections\\ {  $\alpha\in\Omega^1_{\mult}(\G(P);P[\gg])$ } \,}\right\}
\]
Since principal bundle connections are always complete, we see that any multiplicative Ehresmann connection for $\Phi:=(\t, \s):\G\to M\times M$ is complete. This agrees, of course, with Theorem \ref{thm:completeness} since the kernel of $\Phi$ is the locally trivial bundle of groups $P[G]=P\times_GG\to M$ ($G$ acts by conjugation).

\subsection{Principal type}
\label{ex:princ:type:part:split}
Let $\G(P)=P\times_GP\tto M$ be the gauge groupoid of a principal $G$-bundle $P\to M$, and $\H \tto M$ be any Lie groupoid. Consider the fiber product:
\[\G:=\{(g,h)\in \G(P)\times\H\, :\, \s_{\G(P)}(g)=\s_{\H}(h),\, \t_{\G(P)}(g)=\t_{\H}(h)\}\tto M,\]
with groupoid structure such that the inclusion $\G\hookrightarrow \G(P)\times \H$ is a groupoid map. Smoothness of $\G$ follows because the anchor map  $(\t_{\G(P)},\s_{\G(P)}):\G(P)\to M\times M$ is a surjective submersion. This also implies that the groupoid morphism: 
\[\Phi=\pr_2:\G\to \H\]
is a surjective submersion. This Lie groupoid map admits a multiplicative Ehresmann connection, constructed as follows. Let $A(P)=TP/G$ be the Lie algebroid of $\G(P)$, and $B$ that of $\H$. Identifying the Lie algebroid of $\G$ with  $A=A(P)\times_{TM} B$, we see that the bundle of ideals $\ka$ corresponding to $K=\ker \d \Phi$ is identified with the isotropy bundle of $A(P)$ which, as we saw in the previous example, is $P[\gg]$. Also as in the previous example, any principal connection $\eta\in \Omega^1_{\mult}(P;\gg)$ gives rise to a multiplicative connection 1-form $\alpha\in \Omega^1(\G(P),P[\gg])$. Using the projection $\pr_{1}:\G\to \G(P)$, which is a groupoid map, we obtain a multiplicative 1-form on $\G$:
\[\pr_{1}^*\alpha\in \Omega_{\mult}^{1}(\G,\ka).\]
We conclude that $\ka$ is partially split. 

The groupoids endowed with a bundle of ideals obtained via this construction will be called of \textbf{principal type}. Some of the previous examples fit into this setting: 
\begin{itemize}
\item[\tiny$\bullet$] When $P$ is the trivial principal bundle, $P= G\times M$, one obtains the product $\G=G\times \H$ from Subsection \ref{example:os:products}.
\item[\tiny$\bullet$] For a pair groupoid $\H=M\times M$, we obtain back the transitive Lie groupoids $\G=\G(P)$ with $K=\ker \d\s \cap \ker\d\t$ from Subsection \ref{ex:bundle:ideals:transitive}.
\item[\tiny$\bullet$] When $\H$ is a bundle of groups, we obtain the bundle of groups which is the fiberwise product $\G=P[G]\times_M\H$, where $P[G]=P\times_GG$.
\end{itemize}

The kernel of the groupoid morphism $\Phi=\pr_2:\G\to\H$ is $P[G]=P\times_GG\to M$, which is locally trivial. Hence, by Theorem \ref{thm:completeness}, the multiplicative Ehresmann connections for this class of examples are complete.

\subsection{Action groupoids}\label{action:groupoids}

Consider an action groupoid 
\[ \G:=G\ltimes M\tto M\]
associated with an action of a Lie group $G$ on a manifold $M$. The Lie algebroid is the action algebroid:
\[ A:=\gg\ltimes M\implies M\]
associated with the infinitesimal action $\rho:\gg\to\X(M)$. We have $(v,x)\in\ker\rho_x$ if and only if $v$ lies in the isotropy Lie algebra $\gg_x$ of the infinitesimal action. Moreover, the action of an arrow $(g,x)\in G\ltimes M$ on an element $(v,x)\in\ker\rho_x$ is given by:
\[ (g,x)\cdot (v,x)=(\Ad_g v,gx). \]
It follows that a subbundle $\ka\subset \gg\ltimes M$ is a bundle of ideals if and only if it satisfies:
\begin{enumerate}[(i)]
\item $\ka_x\subset \ker\rho_x$;
\item $\Ad_g(\ka_x)=\ka_{gx}$.
\end{enumerate}

The following is immediate:
\begin{proposition}
\label{prop:action}
If the inclusion of a bundle of ideals $\ka\hookrightarrow \gg\times M$ admits a $G$-equivariant splitting $l:\gg\times M\to \ka$, then $\ka$ is partially split with multiplicative connection 1-form $\al\in\Omega^1_\mult(G\ltimes M;\ka)$ given by:
\[\al_{(g,x)}(\dd L_g(v),w):=l(v,x).\]
In particular, if the action of $G$ on $M$ is proper, or if $\gg$ admits a $G$-invariant inner product, then $\ka$ is partially split.
\end{proposition}
\subsection{Non-integrable quotient}
\label{example:Lie-Poisson:sphere}
Let $G$ be a compact semi-simple Lie group, and restrict the adjoint action of $G$ to the unit sphere $M\subset \gg$, with respect to an invariant inner product. The action groupoid $\G:=G\ltimes M\tto M$ has the bundle of ideals 
\[\ka\subset A:=\gg\ltimes M,\quad  \ka|_{x}:=\R x,\  x\in M.\]
By Proposition \ref{prop:action}, $\ka$ is partially split. However, for $\gg\not\simeq \mathfrak{so}(3,\mathbb{R})$, the Lie algebroid $B:=A/\ka$ is not integrable. To see this, it is enough to consider the case when $G$ is simply-connected. If $\H\tto M$ is a Lie groupoid integrating $B$, then we have an induced groupoid morphism $\Phi:\G\to \H$. So for each $x\in M$, we have an induced Lie group map between the isotropy groups $\Phi_x:G_x\to \H_x$. The kernel of $\Phi_x$ is a 1-dimensional group whose connected component is $\exp(\R x)\subset H_x$. However, for $\gg\not\simeq \mathfrak{so}(3,\mathbb{R})$, there are $x\in M$ for which $\exp(\R x)$ is not a closed submanifold of $G$. This is a contradiction.  
\subsection{Groupoids with bi-invariant metrics}
\label{ex:bi:invariant:metrics}
For a general groupoid $\G$ the notion of bi-invariant metric does not make sense: for this we need to lift the right and left actions of $\G$ on itself to its tangent bundle. So assume that $\G\tto M$ is equipped with a Cartan connection, i.e., a multiplicative distribution complementary to the source fibers:
\[ T\G=\ker\d\s\oplus H. \]
Then one can lift the left and right actions of $\G$ on itself to $T\G$ by setting:
\begin{align*}
\lambda^L_g:T_h\G\to T_{gh}\G,&\quad v\mapsto g v:=\d_{(g,h)}m(w,v)\\
\lambda^R_h:T_g\G\to T_{gh}\G,&\quad v'\mapsto v'h:=\d_{(g,h)}m(v',w').
\end{align*}
where $w\in H_g$ and $w'\in H_h$ are the unique tangent vectors such that $\d\s(w)=\d\t(v)$ and $\d\s(v')=\d\t(w')$.
These actions extend the left and right actions of $\G$ on $\ker\d\t$ and $\ker\d\s$ and they are related by the groupoid inversion.

Now, given a groupoid $\G$ with a multiplicative connection $H\subset T\G$, one says that a Riemannian metric $\eta$ on $\G$ is \textbf{bi-invariant} if it is invariant under left translations, right translations and inversion (see \cite{KotovStrobl18}). 
If $\G$ admits a bi-invariant metric, then any bundle of ideals $\ka$ is partially split. In fact, for a bi-invariant metric $\eta$ the distribution orthogonal to $K$:
\[ E:=K^\perp\subset T\G, \]
is a VB subgroupoid of $T\G$, so satisfies Proposition \ref{prop:partially:split:grpd} (i). We leave the details to the reader. 

\begin{remark}
An action groupoid $G\ltimes M\tto M$ has an obvious multiplicative connection. If the action of $G$ on $M$ is proper, using averaging, one obtains a bi-invariant metric on $G\ltimes M$. This gives an alternative argument that if the action of $G$ on $M$ is proper then $\ka$ is partially split, a fact also deduced in Proposition \ref{prop:action}. We will see in Section \ref{sec:Cartan} the infinitesimal analog of this result.  
\end{remark}

\subsection{Over-symplectic groupoids}\label{example:oversympl:groupoid}
A closed multiplicative 2-form $\omega\in \Omega^2_{\mult}(\G)$ is called an \textbf{over-symplectic structure} on the Lie groupoid $\G\tto M$ \cite{BCWZ04} if it satisfies:
\[ \ker\omega\subset \ker \d \t \cap \ker \d \s.\]
It follows from this condition that rank of $\omega$ is constant, equal to $2\dim M$ (see \cite[Proposition 4.5]{BCWZ04}). Then $\ka:=\ker\omega|_M$ defines a bundle of ideals whose associate subgroupoid \eqref{eq:bundle:K} is $K=\ker \omega$. In this case, the existence of partial splittings plays an important role in the study of local models around Poisson submanifolds. For example, it implies that $(\G,\omega)$ embeds coisotropically in some symplectic groupoid. We refer to \cite{FerMa22} where this class of groupoids is discussed in detail.

\section{Morita invariance and properness}
\label{sec:Morita}

In this section we prove the following fundamental result.

\begin{theorem}[Morita invariance]\label{theo:Morita}
Let $\H_1\tto N_1$ and $\H_2\tto N_2$ be two Lie groupoids. A Morita equivalence $\H_1\cong\H_2$ induces a one-to-one correspondence between bundles of ideals in $\H_1$ and bundles of ideals in $\H_2$, under which partially split bundles of ideals are sent to partially split bundles of ideals.
\end{theorem}

We start by deducing from this result the following important fact, which implies the theorem from the introduction. 

\begin{theorem}
\label{thm:proper:partially:split}
A bundle of ideals in a proper Lie groupoid is partially split.
\end{theorem}


\begin{proof}
Let $x\in M$ and denote by $\O$ the orbit of $\G$ through $x$. By the slice theorem for proper Lie groupoids  \cite{CrSt13,Weinstein02,Zung06}, there is a transversal $T$ to $\O$ such that the restriction 
$\G|_T\tto T$ is isomorphic to an action Lie groupoid $G_x\ltimes V\tto V$, for some open set $0\in V\subset \nu_x(\O)$. Since $\G$ is proper, the isotropy group $G_x$ is compact so this is a proper groupoid. On the other hand, $\G|_T\tto T$ is Morita equivalent to $\G|_{U}\tto U$, where $U$ is the saturation of $T$  (see, e.g., \cite{CrSt13}). Therefore, it follows from Proposition \ref{prop:action} and Theorem \ref{theo:Morita} that $\ka|_{U}$ is a partially split bundle of ideals of $\G|_{U}$

Hence, we can cover $M$ by saturated open sets $\{U_i\}_{i\in I}$ such that, for each $i\in I$, there exists a multiplicative connection 1-form $\al_{U_i} \in\Omega^1_\mult(\G|_{U_i},\ka)$. Since $\G$ is proper, there exists a $\G$-invariant partition of unity $\{\rho_i\}_{i\in I}$ subordinate to the cover $\{U_i\}_{i\in I}$ (see, e.g., \cite[Proposition 8]{CrMe18}). Then the $\ka$-valued 1-form:
\[\al:=\sum_{i\in I}(\rho_i\circ\s)\al_{U_i}\in \Omega^1(\G;\ka).\]
is multiplicative and satisfies:
\[  \al|_\ka=\id. \]
Hence, $\al$ is a multiplicative connection 1-form, so $\ka$ is partially split.
\end{proof}

The theorem also has the following consequence:

\begin{corollary}
\label{cor:proper:semi-direct}
Let $\Phi:\G\to \H$ be a groupoid morphism with abelian kernel and assume that $\G$ is target-connected and proper. Then $\Phi:\G\to \H$ admits a Lie groupoid splitting if and only if it admits a leafwise flat, multiplicative Ehresmann connection with trivial holonomy.
\end{corollary}

\begin{proof}
Corollary \ref{cor:flat:splitting} gives one direction. To prove the other direction, we can assume that $\Phi=\pr_1:\H\times_M A\to \H$ where $p:A\to M$ is an abelian bundle of Lie groups. Since $\G=\H\times_M A$ is assumed proper, Theorem \ref{thm:proper:partially:split} gives a multiplicative Ehresmann connection for $\pr_1:\H\times_M A\to \H$, and then Corollary \ref{corollary:semi-direct:abelian:kernel} shows that $p:A\to M$ admits an $\H$-invariant multiplicative Ehresmann connection $E^A$. Then $E=(\d\s)^*E^A$ is a leafwise flat, multiplicative Ehresmann connection for $\pr_1:\H\times_M A\to \H$ with trivial holonomy (see Proposition \ref{prop:invariant:connection}).
\end{proof}

\subsection{Morita invariance}
We now turn to the proof of Morita invariance.
For the general theory of Morita equivalences see \cite{dH13} and \cite{MM04}. 
We take the point of view that a Morita equivalence is given by a Lie groupoid $\G$ and Morita maps $\Phi_1:\G\to\H_1$ and $\Phi_2:\G\to\H_2$ which are surjective submersions. This is equivalent to the bibundle definition via the following constructions. The base of $\G$ is a principal $(\H_1,\H_2)$-bibundle, and to a principal bibundle $\H_1\ract M\lact \H_2$ one associates the groupoid \[(\G\tto M):=(M\times_{N_1} \H_1\times_{N_1} M\tto M)\simeq (M\times_{N_2} \H_2\times_{N_2} M\tto M).\]

We will call a Morita map that is a surjective submersion a \textbf{Morita fibration}. Such maps can be characterized as follows:

\begin{lemma}\label{lemma:iso:morita:fib}
A groupoid map $\Phi:\G\to\H$ is a Morita fibration if and only if it covers a surjective submersion $\phi:M\to N$ and $\G\tto M$ is isomorphic to the pullback of $\H\tto N$ along $\phi$:
\[\phi^*(\H)=M\times_N\H\times_NM\tto M\]
 via the map:
\[\overline{\Phi}:\G\to \phi^*(\H), \quad \overline{\Phi}(g)=(\t(g),\Phi(g),\s(g)).\]
\end{lemma}

Given a Morita fibration $\Phi:\G\to\H$, covering a submersion $\phi:M\to N$, it follows from the properties above that its kernel:
\[ \K=\{g\in \G:\Phi(g)=1_x, x\in N\},\]
is a submersion groupoid: an element $g\in \K$ is uniquely determined by its source and target, and the restriction of $\overline{\Phi}$ gives an isomorphism of Lie groupoids:
\[ \K\diffto M\times_N M,\quad g\mapsto (\t(g),\s(g)).\]
On submersion groupoids, multiplicative forms are multiplicatively exact:

\begin{lemma}\label{lemma:submersion:groupoid}
Let $\phi:M\to N$ be a surjective submersion, and $E\to N$ be a vector bundle. Any multiplicative form on the submersion groupoid
\[\alpha\in \Omega^{\bullet}_{\mult}(M\times_NM;\phi^*(E))\] is multiplicativly exact, i.e., there exists $\theta\in \Omega^{\bullet}(M;\phi^*(E))$ such that $\alpha=\delta \theta$, where 
\[\delta:\Omega^{\bullet}(M;\phi^*(E))\to\Omega^{\bullet}_\mult(M\times_NM;\phi^*(E))\] 
is the simplicial differential of $M\times_NM$ with coefficients in $\phi^*(E)$ (see \eqref{eq:simplicial:E:forms:0}).
\end{lemma}

\begin{proof}
Let $U\subset N$ be an open set on which there exists a local section $\sigma:U\to M$ of $\phi$. Define $\theta\in \Omega^{\bullet}(\phi^{-1}(U);\phi^*(E))$ by $\theta(v):=\alpha(v, \d (\sigma\circ \phi)(v))$, for $v\in \wedge^{\bullet}T\phi^{-1}(U)$. Multiplicativity of $\alpha$ implies that $\alpha|_{\phi^{-1}(U)}=\delta \theta$:
\begin{align*}
\alpha(v_1,v_2)&=m^*(\alpha)\big((v_1,\d\sigma(u)),(\d\sigma(u),v_2)\big)=\alpha(v_1,\d\sigma(u))+\alpha(\d\sigma(u),v_2)\\
&=\theta(v_1)-\theta(v_2)=\delta\theta (v_1,v_2),
\end{align*}
for all $v_1,v_2 \in \wedge^{\bullet}T\phi^{-1}(U)$ with $\d\phi (v_1)=u=\d\phi(v_2)$. 

In general, consider an open cover $\{U_i\}_{i\in I}$ of $N$ for which local sections of $\phi$ exist. Then the corresponding local primitives $\theta_i\in \Omega^{\bullet}(\phi^{-1}(U_i);\phi^*(E))$ of $\alpha$ can be glued, using a partition of unity $\{\chi_i\}_{i\in I}$ subordinated to the cover, to a global primitive: 
$\theta:=\sum_{i\in I} \chi_i\circ \phi\cdot \theta_i$.
\end{proof}

The following two propositions prove Theorem \ref{theo:Morita}.

\begin{proposition}
A Morita fibration $\Phi:\G\to\H$ gives a 1-to-1 correspondence between bundles of ideals in $\G$ and bundles of ideals in $\H$.
\end{proposition}

\begin{proof} For each $x\in M$, $\Phi$ restricts to an isomorphism of Lie groups: $\G_x\diffto \H_{\phi(x)}$. So we can define a map of bundles of ideals: 
\[ A(\H)\supset \ka \longmapsto \Phi^*(\ka), \quad \text{where}\quad \Phi^*(\ka)_x=\{v\in \ker\rho_x: \d_x\Phi(v)\in\ka|_{\phi(x)}\}\subset A(\G),\]

For the inverse map, given a bundle of ideals $\ka$ in $\G$, one needs to show that if $y=\phi(x)=\phi(x')$ then $\d\Phi(\ka_x)=\d\Phi(\ka_{x'})$. As we observed above, there is a unique $k\in\K$ such that $\s(k)=x$ and $\t(k)=x'$, and we have:
\[ \ka_{x'}=k\cdot \ka_x. \]
Since $k$ belongs to the kernel, we have $\Phi(kgk^{-1})=\Phi(g)$, and so we conclude that:
\[ \d\Phi(\ka_{x'})=\d\Phi(k\cdot \ka_x)=\d\Phi(\ka_x). \]
Therefore we get a map  of bundles of ideals: 
\[ A(\G)\supset \ka \longmapsto \Phi_*(\ka), \quad \text{where}\quad \Phi_*(\ka)_{\phi(x)}=\d\Phi(\ka_x)\subset A(\H), \]
which is the inverse of the map above.
\end{proof}

Now if $\al\in \Omega^1_\mult(\H,\ka)$ is a multiplicative connection 1-form for $\ka$, then the pullback form:
\[ \Phi^*\al\in\Omega^1_\mult(\G,\Phi^*\ka), \]
is easily seen to be a multiplicative connection 1-form for $\Phi^*\ka$. This establishes one half of the following proposition:

\begin{proposition}
Given a Morita fibration $\Phi:\G\to\H$, a bundle of ideals $\ka$ in $\H$ is partially split if and only if the bundle of ideals $\Phi^*\ka$ in $\G$ is partially split.
\end{proposition}

\begin{proof}
We already now that if $\ka$ is a partially split ideal in $\H$ then $\Phi^*\ka$ is a partially split ideal in $\G$.

For the converse, let $\ka$ be a bundle of ideals in $\H$ so that $\Phi^*\ka$ is partially split. Fix a multiplicative connection 1-form $\al\in\Omega^1(\G,\phi^*\ka)$. We will ``correct'' $\al$ so that it is the pullback of some multiplicative 1-form in $\H$. As we mentioned already, the kernel $\K$ of a Morita fibration $\Phi$ is a submersion groupoid. So by Lemma \ref{lemma:submersion:groupoid} applied to the restriction $\al|_{\K}$, there is a $\ka$-valued 1-form $\omega\in\Omega^1(M,\phi^*\ka)$ such that:
\[ \al|_\K=\delta\omega, \]
where $\delta:\Omega^1(M,\phi^*\ka)\to\Omega^1_\mult(\K,\phi^*\ka)$ is the simplicial differential of $\K$ with coefficients in $\phi^*\ka$ (see \eqref{eq:simplicial:E:forms:0}). Now consider the 1-form:
\[ \widetilde{\al}:=\al-\delta \omega\in \Omega^1_\mult(\G,\phi^*\ka), \]
where $\delta:\Omega^1(M,\phi^*\ka)\to\Omega^1_\mult(\G,\phi^*\ka)$ is the simplicial differential of $\G$ with coefficients in $\phi^*\ka$. Then the pullback of $\widetilde{\al}$ to $\K$ vanishes. We claim that there exists a unique multiplicative form $\beta\in\Omega^1_\mult(\H,\ka)$ such that: 
\[\widetilde{\al}=\Phi^*\be.\]
For this, we use the identification $\G\simeq \phi^*(\H)$ from Lemma \ref{lemma:iso:morita:fib}. Then we have to show that $\widetilde{\al}$ takes the same value on any two vectors of the form:
\[(X_{i},Y,Z_{i})\in T_{p_i}M\times_{T_{\t(g)}N} T_{g}\H\times_{T_{\s(g)}N}T_{q_i}M, \quad i=1,2.\]
Since $\d\phi(X_1)=\d\t(Y)=\d\phi(X_2)$ and $\d\phi(Z_1)=\d\s(Y)=\d\phi(Z_2)$, we have that:
\[(X_{1},Y,Z_{1})=(X_{1},\d\phi(X_1),X_{2})*(X_{2},Y,Z_{2})*(Z_{2},\d\phi(Z_1),Z_{1}),\]
where $*=\d m$ is the multiplication in $T\phi^*(\H)$. Using that the first and the last element on the right-hand side are in $T\K$, and that $\widetilde{\al}$ is multiplicative, we obtain the claimed equality: 
\[\widetilde{\alpha}(X_{1},Y,Z_{1})=\widetilde{\alpha}(X_{2},Y,Z_{2})\in \ka_{\s(g)}.\]
Finally, $\widetilde{\al}|_{\Phi^*\ka}=\al|_{\Phi^*\ka}=\id$ implies that $\be|_{\ka}=\id$. Hence, $\ka$ is partially split.
\end{proof}

\subsection{Bundle gerbes and connections}
\label{sec:gerbes}
As an application of the results discussed so far, we discuss now how one can recover the curving and 3-curvature for bundle gerbes out of a multiplicative Ehresmann connection. Our aim is not to discuss the general theory of gerbes and multiplicative connections, for which we refer the reader to \cite{LSX09}, but rather to give a simple illustration of our theory. For that reason we consider only $S^1$-gerbes over manifolds (see, \cite{Hitchin01,LSX09,Murray96}). We will see how one can recover from our results the classical theorem of Murray \cite{Murray96} that the class of the curvature 3-form is the image in $H^3(N,\R)$ of the Dixmier-Douady class in $H^3(N,\Z)$ (see Theorem \ref{thm:gerbe}).

Let $N$ be a smooth manifold. By an {\bf $S^1$-central extension} over $N$ we mean a surjective submersion $\phi:M\to N$ together with a groupoid morphism onto  the corresponding submersion groupoid $\Phi:\G\to M\times_\phi M$, with kernel the trivial $S^1$-bundle $\K=S^1_M=M\times S^1$:
\[
\xymatrix{
1\ar[r] & S^1_M\ar[r] & \G\ar[r]^---{\Phi} & M\times_\phi M\ar[r] & 1,
}
\]
such that for all $g\in\G$ and $\theta\in S^1$:
\[ g\cdot (\s(g),\theta)=(\t(g),\theta)\cdot g. \]

A {\bf Morita equivalence} of $S^1$-central extensions $\Phi_1:\G_1\to M_1\times_{\phi_1} M_1$ and $\Phi_2:\G_2\to M_2\times_{\phi_2} M_2$ is given by a principal $(\G_1,\G_2)$-bibundle
\[
\xymatrix{
 \G_1 \ar@<0.25pc>[d] \ar@<-0.25pc>[d]  & \ar@(dl, ul) & P \ar[dll]\ar[drr]  & \ar@(dr, ur) & \G_2 \ar@<0.25pc>[d] \ar@<-0.25pc>[d]  \\
M_1&  & & & M_2}
\]
such that orbits of the actions of the kernels $S^1_{M_1}$ and $S^1_{M_2}$ on $P$ coincide. The orbit space of these actions is then a manifold $Q$, which can be canonically identifies with the fiber product $Q=M_1\tensor[_{\phi_1}]{\times}{_{\phi_2}} M_2$. Moreover, it is a principal bi-bundle for the submersion groupoids:
\[
\xymatrix{
 M_1\times_{\phi_1} M_1 \ar@<0.25pc>[d] \ar@<-0.25pc>[d]  & \ar@(dl, ul) & M_1\tensor[_{\phi_1}]{\times}{_{\phi_2}} M_2 \ar[dll]\ar[drr]  & \ar@(dr, ur) & M_2\times_{\phi_2} M_2 \ar@<0.25pc>[d] \ar@<-0.25pc>[d]  \\
M_1&  & & & M_2}
\]
and one obtains a map of principal bi-bundles:
\[
\xymatrix{ M_1\ar@{=}[d]& P\ar[l]\ar[r]\ar[d] & M_2\ar@{=}[d]\\ M_1& M_1\tensor[_{\phi_1}]{\times}{_{\phi_2}} M_2\ar[l]\ar[r] & M_2}
\]

\begin{definition}
An {\bf $S^1$-gerbe over a manifold} $N$ is a Morita equivalence class of $S^1$-central extensions over $N$.
\end{definition}

We recall that an $S^1$-gerbe is completely characterized by its \emph{Dixmier-Douady class} which can be defined as follows. Given an $S^1$-central extension defined by a submersion $\phi:M\to N$ and a groupoid morphism $\Phi:\G\to M\times_\phi M$, one chooses a good cover $\{U_i\}$ of $N$ for which there exist sections $s_i:U_i\to M$ of $\phi:M\to N$. Then one can find maps $g_{ij}:U_{ij}\to \G$ such that
\[ \Phi(g_{ij}(x))=(s_i(x),s_j(x)). \]
Since, for each $x\in U_{ijk}$, the composition $g_{ij}(x)\cdot g_{jk}(x)\cdot g_{jk}(x)$ is an element of the kernel $S^1_M$, we have that:
\[ g_{ij}(x)\cdot g_{jk}(x)\cdot g_{ki}(x)=(s_i(x),c_{ijk}(x)), \]
for a 3-cocycle $c_{ijk}:U_{ijk}\to S^1$. The class in sheaf cohomology:
\[ c_2(\G):=[c_{ijk}]\in H^2(N,\underline{S^1}), \]
only depends on the Morita equivalence class of the extension and is called the {\bf Dixmier-Douady class of the $S^1$-gerbe}. Using the exponential sequence:
\[
\xymatrix{
1\ar[r] & \Z\ar[r] & \R\ar[r]^---{\exp} & S^1\ar[r] & 1,
}
\]
one can view the Dixmier-Douady class as a class in integer cohomology:
\[ c_2(\G)\in H^3(N,\Z). \]

In this way, one can think that $S^1$-central extensions give geometric representatives of integer cohomology classes in degree 3, the same way as principal $S^1$-bundles give geometric representatives of integer cohomology classes in degree 2, via their Chern class. The same way one can use the curvature of a principal connection to obtain representatives in real cohomology of the Chern class, we will see now that one can use multiplicative Ehresmann connections to obtain representatives in real cohomology of the Dixmier-Douady class.

We start by the following proposition:

\begin{proposition}
Every $S^1$-central extension 
\[
\xymatrix{
1\ar[r] & S^1_M\ar[r] & \G\ar[r]^---{\Phi} & M\times_\phi M\ar[r] & 1,
}
\]
admits a multiplicative Ehresmann connection $E$. The induced multiplicative Ehresmann connection $E^\K$ on the kernel $\K=S^1_M$ is the canonical flat connection.
\end{proposition}

\begin{proof}
Notice that $\G$ is necessarily a proper groupoid. Hence, it follows from Theorem \ref{thm:proper:partially:split} that multiplicative Ehresmann connections exist. 

By Proposition \ref{prop:connection:partial:split}, a multiplicative Ehresmann connection $E^{S^1}$ on the trivial $S^1$-bundle $S^1_M$ is related to a linear connection $E^{\nabla}$ on the trivial line bundle $\R_M$ by the exponential map:
\[ (\d\exp)(E^\nabla)=E^{S^1_M}.\]
It follows that constant sections of $\R_M$, $x\mapsto (x,n)$ ($n\in\Z$), are flat. Then any constant section is flat, and so $E^{S^1_M}$ must be the trivial connection. 
\end{proof}

Hence, given a multiplicative Ehresmann connection $E$ for an $S^1$-central extension $\Phi:\G\to M\times_\phi M$, the associated linear connection on $\ka=\R_M$ is the canonical flat connection:
\[ \nabla_X=\Lie_X. \]
Then, by Proposition \ref{prop:curvature:basic}, the curvature $\Omega\in\Omega^2_\mult(\G)$ of $E$ satisfies
\[ \Omega=\Phi^*\underline{\Omega},\]
for a unique closed, multiplicative, 2-form $\underline{\Omega}\in\Omega^2_\mult(M\times_\phi M)$. By Lemma \ref{lemma:submersion:groupoid}, there exists a 2-form $F\in\Omega^2(M)$ (not unique) such that:
\[ \underline{\Omega}=\pr_1^*F-\pr_2^*F,\]
where $\pr_i:M\times_\phi M\to M$ are the source and target of the submersion groupoid.
Equivalently:
\begin{equation}
    \label{eq:curving}
    \Omega=\t^*F-\s^*F.
\end{equation}

\begin{definition}
A form $F\in\Omega^2(M)$ satisfying \eqref{eq:curving} is called a {\bf curving} of the connection $E$.
\end{definition}

Finally, observe that:
\[ 
0=\d\underline{\Omega}=\pr_1^*\d F-\pr_2^*\d F. 
\]
Therefore, there exists a unique 3-form $G\in\Omega^3(N)$ such that:
\begin{equation}
    \label{eq:curvature:3:form}
    \d F=\phi^* G.
\end{equation}

\begin{definition}
The form $G\in\Omega^2(N)$ satisfying \eqref{eq:curvature:3:form} is called the {\bf curvature 3-form} of the curving.
\end{definition}

Note that two curvings $F$ and $F'$ differ by a form $B=F-F'$ such that:
\[ \pr_1^*B-\pr_2^*B=0. \]
Hence, $B=\phi^*\varphi$ for a 2-form $\varphi\in\Omega^2(N)$. The corresponding curvature 3-forms are then related by:
\[ G-G'=\d\varphi.\]
In particular, the class $[G]\in H^3(N)$ does not depend on the choice of curving and is intrinsically associated with the multiplicative Ehresmann connection $E$. We can now state and prove Murray's result:

\begin{theorem}
\label{thm:gerbe}
Let $\Phi:\G\to M\times_\phi M$ be an $S^1$-central extension representing an $S^1$-gerbe over a manifold $N$. Given any multiplicative Ehresmann connection $E$ for $\Phi$, the class $[G]$ of a curvature 3-form is the image of the Dixmier-Douady class of the gerbe $c_2(\G)$ under the canonical map:
\[ H^3(N,\Z)\to H^3(N,\R). \]
\end{theorem}

\begin{proof}
As above, we pick a good cover $\{U_i\}$ of $N$ for which there exist sections $s_i:U_i\to M$ of $\phi:M\to N$, and maps $g_{ij}:U_{ij}\to \G$ such that
\[ \Phi(g_{ij}(x))=(s_i(x),s_j(x)). \]
The Dixmier-Douady class is then $c_2(\G):=[c_{ijk}]$ where the 3-cocycle $c_{ijk}:U_{ijk}\to S^1$ is given by:
\[ (s_i,c_{ijk})=g_{ij}\cdot g_{jk}\cdot g_{ki}. \]
This is an equality of maps $U_{ijk}\to \G$ and
we are going to pull back the connection 1-form $\al\in\Omega^1(\G)$ under both sides.
For the left-side, it follows from the condition $\al(\xi^R)=\xi$ that
\[ (s_i,c_{ijk})^*\al=c_{ijk}^{-1}\d c_{ijk}=\d f_{ijk}, \]
where we set $c_{ijk}=\exp(f_{ijk})$, for a map $f_{ijk}:U_{ijk}\to\R$, well-defined up to the addition of an integer. On the other hand, for the right-hand side, using that $\al$ is multiplicative, we find:
\begin{align*}
    (g_{ij}\cdot g_{jk}\cdot g_{ki})^*\al&=
    g_{ij}^*\al+g_{jk}^*\al+g_{ki}^*\al\\
    &=A_{ij}+A_{jk}+A_{ki},
\end{align*}
where 
\[ A_{ij}:=g_{ij}^*\al\in\Omega^1(U_{ij}). \] 
We conclude that:
\begin{equation}
    \label{eq:class:1}
    A_{ij}+A_{jk}+A_{ki}=\d f_{ijk}.
\end{equation}
Now, if $F\in\Omega^2(M)$ is a curving, and we let:
\[ F_{i}:=s_{i}^*F\in\Omega^2(U_i), \]
it follows from \eqref{eq:curving} that on $U_{ij}$ we must have:
\begin{align*} 
F_i-F_j=s_i^*F-s_j^*F&=(s_i,s_j)^*\underline{\Omega}=\\
&=g_{ij}^*\Omega
=g_{ij}^*\d\al=\d g_{ij}^*\al.
\end{align*}
In other words,
\begin{equation}
    \label{eq:class:2}
    F_i-F_j=\d A_{ij}.
\end{equation}
On the other hand, the definition \eqref{eq:curvature:3:form} of the curvature 3-form gives:
\begin{equation}
    \label{eq:class:3}
    G|_{U_i}=\d F_i.
\end{equation}
Equations \eqref{eq:class:1}, \eqref{eq:class:2} and \eqref{eq:class:3} show that $[G]\in H^3(N,\R)$ corresponds to the class $[\d f_{ijk}]\in H^2(N,\underline{\Omega^1_{\textrm{cl}}})$, so using that $c_{ijk}=\exp(f_{ijk})$ the result follows.
\end{proof}

\section{IM Ehresmann connections}
\label{section: IM Ehresmann}

\subsection{IM connection for a surjective algebroid morphism}

We now turn to the infinitesimal version of multiplicative Ehresmann connections. If one starts with a surjective Lie algebroid morphism $\phi:A\to B$ covering the identity, its differential is a morphism of VB algebroids:
\[
\xymatrix@C=10pt@R=10pt{
TA \ar@{=>}[dd] \ar[rr]^{\dd\phi} \ar[dr]&  &  TB  \ar@{=}[d] \ar[dr]\\
 & 
A\ar[rr]^(.3){\phi}  \ar@{=>}[dd] & \ar@{=>}[d] &B  \ar@{=>}[dd] \\
TM \ar[dr] \ar@{-}[r]^---{\id} &\ar[r]  & TM \ar[dr]\\
 & M\ar[rr]^(.3){\id} &  &M
}
\]
We recall that double arrows $\implies$ represent Lie algebroids. The kernel of $\dd\phi$ is a VB subalgebroid of $TA$:
\[
\xymatrix{
\ker(\dd\phi) \ar[r] \ar@{=>}[d] & A \ar@{=>}[d]\\
0_M\ar[r] & M
}
\]
This leads naturally to the infinitesimal version of a multiplicative Ehresmann connection:

\begin{definition}
\label{def:IM:connection}
An {\bf IM Ehresmann connection} for $\phi:A\to B$ is a VB subalgebroid $E\Ato TM$ of $TA\Ato TM$ which is ``horizontal":
\[ TA=\ker(\dd\phi)\oplus E. \]
\end{definition}

\subsection{Infinitesimal partially split bundles of ideals}
\label{sec:local:model:algbrd}

A surjective Lie algebroid morphism $\phi:A\to B$ covering the identity is completely determined by its kernel:
\[ \xymatrix{0\ar[r]& \ka\ar[r]& A\ar[r]^{\phi}& B\ar[r]& 0}.\]
So the notion of IM Ehresmann connection for $\phi$ can be rephrased in terms of the bundle of ideals $\ka$, without any mentioning of $\phi$ and $B$. To do this we need to express $\ker(\d\phi)$ in terms of $\ka$, which can be done as follows.

Let $A\Ato M$ be a Lie algebroid and $\ka\subset A$ a \textbf{bundle of ideals}, i.e., $\ka \to M$ is a vector subbundle included in $\ker\rho_A$ and satisfying 
\[
 \al\in\Gamma(A),\ \ga\in\Gamma(\ka) \quad \Longrightarrow\quad [\al,\ga]_A\in\Gamma(\ka).
\]
Then we have the canonical representation of $A$ on $\ka$: 
\begin{equation}
\label{eq:representation:algbrd} 
\nabla^{\ka}_\al\gamma:=[\al,\gamma]_{A},\quad \al\in\Gamma(A), \ga\in\Gamma(\ka).
\end{equation}
This is the infinitesimal version of the groupoid representation \eqref{eq:representation:grpd} and it gives rise to the infinitesimal version of the groupoid $\G\times_M\ka\tto M$. Namely, we have the semi-direct product Lie algebroid $A\times_M \ka\Ato M$, with Lie bracket
\[ [(\al_1,\gamma_1),(\al_2,\gamma_2)]:=([\al_1,\al_2]_{A},\nabla^\ka_{\al_1}\gamma_2-\nabla^\ka_{\al_2}\gamma_1), \]
and anchor $\rho_{A}\circ\pr_{A}$. This is a VB algebroid:
\[
\xymatrix{
A\times_M\ka \ar[r] \ar@{=>}[d] & A \ar@{=>}[d]\\
M\ar[r] & M
}
\]
and, in fact, we have:

\begin{lemma}
Given a bundle of ideals $\ka\subset A$ the image of the inclusion
\begin{equation}\label{inclusion:VB:subalgebroid}
A\times_M\ka\hookrightarrow TA,\quad (a,\gamma)\mapsto \frac{\d}{\d t}\Big|_{t=0}(a+t\gamma),
\end{equation}
coincides with $\ker(\d\phi)$, where $\phi:A\to B:=A/\ka$ is the quotient map.
\end{lemma}

\begin{proof}
Since $\phi:A\to B$ is surjective and covers the identity, $\d_a\phi:T_aA\to T_{\phi(a)} B$ is surjective, for any $a\in A$. Also, if $(a,\gamma)\in A\times_M \ka$ then $\phi(a+t\gamma)=\phi(a)$ so under the natural inclusion \eqref{inclusion:VB:subalgebroid} $A\times_M \ka$ is mapped to the kernel of $\d\phi$. Since 
\[ \rank_A T A-\rank_B T B=\rank_M\ka=\rank_A (A\times_M \ka),\] 
the result follows.
\end{proof}

Hence, we can rephrase Definition \ref{def:IM:connection} in terms of bundle of ideals as follows:

\begin{definition}
\label{def:partially:split:algbrd}
A bundle of ideals $\ka\subset A$ is called {\bf partially split} if there is a VB subalgebroid $E\Ato TM$ of $TA\Ato TM$ such that:
\[ TA=(A\times_M\ka)\oplus E. \]
\end{definition}

We emphasize that, unlike the groupoid case, IM Ehresmann connections for a surjective Lie algebroid morphism and partially split bundle of ideals are equivalent notions. From now on we will mostly assume the perspective of bundle of ideals. 

\subsection{The infinitesimal partially split condition}
\label{sec:local:model:description}

We now look for alternative characterizations of (infinitesimal) partially split ideals. 

If we represent the inclusion \eqref{inclusion:VB:subalgebroid} by the diagram of VB algebroids:
\[
\vcenter{
\xymatrix@R=10pt{
A\times_M \ka\  \ar@{=>}[dd] \ar[dr]  \ar[dr]  \ar@{^{(}->}@<-0.10pc>[rr] &  &  TA
\ar@{=>}[dd] \ar[dl]\\
 &  A  \ar@{=>}[dd]  \\
0_M\  \ar[dr]  \ar[dr]  \ar@{^{(}-}[r] &\ar[r] & TM \ar[dl] \\
 & M
}}
\]
then the VB dual to this inclusion is the projection:
\[
\vcenter{
\xymatrix@R=10pt{
A\ltimes\ka^* \ar@{=>}[dd] \ar[dr]  \ar[dr] &  &  T^*A\ar@{->>}[ll] 
\ar@{=>}[dd] \ar[dl]\\
 &  A  \ar@{=>}[dd]  \\
\ka^* \ar[dr]  \ar[dr]& \ar@{->>}[l] & A^* \ar[dl] \ar@{-}[l]\\
 & M
}}
\]
Here $A\ltimes \ka^*\Rightarrow  \ka^*$ is the action algebroid associated with the representation $\nabla^{\ka^*}$ of $A$ on $\ka^*$ dual to the representation \eqref{eq:representation:algbrd}, which is defined by:
\[ \Lie_{\rho_{A}(\al)}(\langle \be,X \rangle)=\langle \nabla^{\ka}_\al\be, X \rangle+\langle \be, \nabla^{\ka^*}_\al X \rangle. \]

We have the following infinitesimal analogue of Proposition \ref{prop:partially:split:grpd} giving alternative characterizations of partial splittings (for the terminology, see the appendix):

\begin{proposition}
\label{prop:partially:split:algbrd}
Given a bundle of ideals $\ka\subset A$, the following structures are in 1-to-1 correspondence:
\begin{enumerate}[(i)]
\item VB subalgebroids $E\subset TA$ that are complementary to $A\times_M\ka$:
\[ TA=(A\times_M\ka)\oplus E;\]
\item VB algebroid morphisms $\theta:A\ltimes \ka^* \to T^*A$ that are splittings of the natural projection $p:T^*A\to A\ltimes \ka^*$:
\[ p\circ\, \theta=\id_{T^*A};\]
\item $\ka$-valued, IM 1-forms $(L,l)\in\Omega^1_\imult(A,\ka)$ that restrict to the identity on $\ka$:
\[ l|_\ka=\id_{\ka}; \]
\item linear, closed, IM 2-forms $\mu\in\Omega^2_\imult(A\ltimes \ka^*)$ that along $\ka\subset (A\ltimes \ka^*)|_M$  satisfy:
\[ \big(\pr_{\ka}\circ  \mu|_M\big)|_\ka=\id_{\ka}. \]
\end{enumerate}
\end{proposition}

\begin{proof}
The equivalence between (i) and (ii) follows by passing from the VB algebroid $TA$ to its dual VB algebroid $T^*A$, as explained above. 

A vector bundle map 
\[
\xymatrix@C=10pt@R=10pt{
A\ltimes \ka^* \ar@{=>}[dd] \ar[rr]^{\theta} \ar[dr]&  & T^*A  \ar@{=}[d] \ar[dr]\\
 & 
A\ar[rr]^(.3){\id}  \ar@{=>}[dd] & \ar@{=>}[d] &A  \ar@{=>}[dd] \\
\ka^* \ar[dr] \ar@{-}[r] &\ar[r]  & A^* \ar[dr]\\
 & M\ar[rr] &  &M
}
\]
is a VB algebroid morphism covering the identity if and only if its dual is a VB algebroid morphism:
\[
\xymatrix@C=10pt@R=10pt{
TA \ar@{=>}[dd] \ar[rr]^{\theta^\vee} \ar[dr]&  &  A\oplus \ka  \ar@{=}[d] \ar[dr]\\
 & 
A\ar[rr]^(.3){\id}  \ar@{=>}[dd] & \ar@{=>}[d] &A  \ar@{=>}[dd] \\
TM \ar[dr] \ar@{-}[r] &\ar[r]  & 0_M \ar[dr]\\
 & M\ar[rr] &  &M
}
\]
But a Lie algebroid morphism $\theta^\vee:TA\to A\oplus \ka$ is the same thing as an $\ka$-valued, IM 1-form $(L,l)\in\Omega^1_\imult(A;\ka)$ (see Example \ref{ex:degree1:E-valued:IM:forms}). This establishes the equivalence between (ii) and (iii), since the extra conditions along $M$ correspond to each other.

Next, a closed IM 2-form $\mu:A\ltimes \ka^*\to T^*\ka^*$ can be viewed, alternatively, as a VB algebroid morphism:
\[
\xymatrix{
T(A\ltimes \ka^*)\ar[r]^{(\omega_\mu)^\flat} \ar@{=>}[d] & T^*(A\ltimes \ka^*) \ar@{=>}[d] \\
T\ka^*\ar[r] & (A\ltimes \ka^*)^*
}
\]
where $\omega_\mu\in\Omega^2(A\ltimes \ka^*)$ is the fiberwise linear 2-form:
\[ \omega_\mu=\mu^*\omega_\can. \]
The algebroid morphism $(\omega_\mu)^\flat$ restricts to a VB algebroid morphism:
\[ \theta^\vee:TA\to A\oplus \ka\] 
covering the identity $\id:A\to A$, such that the following diagram commutes:
\[
\xymatrix@C=10pt@R=10pt{
TA\ar@{=>}[dd]\ar[rr]^{\theta^\vee} \ar@{^{(}->}[dr]&  &  A\oplus \ka\ar@{=}[d] \ar@{^{(}->}[dr]\\
 & 
T(A\ltimes \ka^*)\ar[rr]^(.4){(\omega_\mu)^\flat} \ar@{=>}[dd] &\ar@{=>}[d] &T^*(A\ltimes \ka^*) \ar@{=>}[dd] \\
TM \ar@{^{(}->}[dr] \ar@{-}[r] & \ar[r] & M \ar@{^{(}->}[dr]\\
 & T\ka^*\ar[rr] &  &(A\ltimes \ka^*)^*
}
\]
Moreover, $\theta^\vee$ (or equivalently $\theta$) completely determines $\mu$: using the fact that $\mu$ is a linear IM form, one finds that it coincides with the pullback under the algebroid morphism $\theta:A\ltimes \ka^*\to T^* A$ of the canonical IM 2-form $\mu_\can$:
\[ \mu=\theta^*\mu_\can. \]
This proves the equivalence between (ii)-(iv) since the additional conditions along $M$ correspond to each other.
\end{proof}

\begin{definition}
Given a partially split bundle of ideals $\ka$ with an IM Ehresmann connection $E\subset TA$, the corresponding $\ka$-valued IM 1-form $(L,l)\in\Omega^1_\imult(A;\ka)$, given by Proposition \ref{prop:partially:split:algbrd} (iii), is called
the \textbf{IM connection 1-form}.
\end{definition}

We can now make precise the statement that IM Ehresmann connections are the infinitesimal counterpart of multiplicative Ehresmann connections:

\begin{theorem}
\label{thm:Lie:functor:connections}
Let $\G\tto M$ be a target 1-connected Lie groupoid with Lie algebroid $A\Ato M$. Given a bundle of ideals $\ka\subset A$ there is a 1:1 correspondence:
\[ 
\left\{\txt{multiplicative Ehresmann\\ connections $E\subset T\G$\, \\
for $\G\times_M \ka$} \right\}
\tilde{\longleftrightarrow}
\left\{\txt{IM Ehresmann\\ connections $E_*\subset TA$\, \\
for $A\times_M \ka$
}\right\}
\]
\end{theorem}

\begin{proof}
The proof follows from the description of multiplicative and IM Ehresmann connections in terms of their connections 1-forms given by Propositions \ref{prop:partially:split:grpd} and \ref{prop:partially:split:algbrd}, and by observing that for a target 1-connected Lie groupoid one has a canonical isomorphism \cite{CrSaSt15}
\[ \Omega^1_\mult(\G;\ka)\simeq \Omega^1_\imult(A;\ka), \]
under which the extra conditions on the 1-forms correspond to each other.
\end{proof}

If $\G$ is not target 1-connected the differentiation of multiplicative forms still exists, and one can still associate to a multiplicative Ehresmann connection a IM Ehresmann connection, but not conversely.

Note that for a IM connection 1-form $(L,l)\in\Omega^1_\imult(A;\ka)$, the symbol $l:A\to\ka$ satisfies $l|_\ka=\id$ and hence gives a splitting of the short exact sequence:
\[ 
\xymatrix{0\ar[r]& \ka\ar[r]& A\ar@/^/@{-->}[l]^{l}\ar[r]^{\phi}& B\ar[r]& 0}.
\]
In general, the induced splitting $B\to A$ is not a Lie algebroid morphism.

The partially split condition implies certain properties of the bundle of ideals. A first consequence is the following infinitesimal versions of Corollaries \ref{corollary:locally:trivial:groupoid} and \ref{corollary:complement:groupoid}:

\begin{corollary}
\label{corollary:Lie:alg:bundle}
Let $\ka\subset A$ be a partially split bundle of ideals with IM connection 1-form $(L,l)\in\Omega^1_\imult(A;\ka)$, then
\begin{equation} 
\label{eq:connection:ka}
\nabla^L_X\xi:=i_XL(\xi),
\end{equation}
defines a connection on $\ka$ which preserves the Lie bracket:
\[\nabla^{L}_{X}[\xi,\eta]_{\ka}=[\nabla^{L}_{X}\xi,\eta]_{\ka}+[\xi,\nabla^{L}_{X}\eta]_{\ka},\quad X\in\X(S),\xi,\eta\in\Gamma(\ka). \]
In particular, $\ka$ is a locally trivial bundle of Lie algebras.
\end{corollary}

\begin{proof}
Since $l|_\ka=\id$, the symbol equation \eqref{eq:symbol:IM:form} implies that \eqref{eq:connection:ka} defines a linear connection $\nabla^L$ on $\ka$. On the other hand, the third equation in the IM condition (\ref{eq:compatibility:IM:E:form}) shows that one has 
\[ \Lie_\xi\eta=[\xi,\eta],\quad \xi,\eta\in\Gamma(\ka). \] 
Then the second equation in (\ref{eq:compatibility:IM:E:form}) shows that $\nabla^L$ preserves the Lie bracket.
\end{proof}

\begin{corollary}\label{corollary:isotropy:splits}
Let $\ka\subset A$ be a partially split bundle of ideals and fix $x\in M$. The symbol of any IM connection 1-form $(L,l)\in\Omega^1_\imult(A;\ka)$ gives a decomposition of the  isotropy Lie algebra $\gg_x:=\ker \rho_{A}|_x$ into a direct sum of ideals: 
\[ \gg_x\simeq (\gg_x\cap \ker l)\oplus \ka_x. \]
\end{corollary}

\begin{proof}
If $\al\in\Gamma(A)$ and $\be\in\Gamma(\ker l)$ are such that $\al_x,\be_x\in\gg_x$, then the third equation in the IM condition (\ref{eq:compatibility:IM:E:form}) shows that 
\[ [\al_x,\be_x]=[\al,\be]|_x\in\ker l.\]
This proves that $\gg_x\cap \ker l$ is an ideal in $\gg_x$, so the corollary follows.
\end{proof}

An interesting characterization of the partial split condition can be obtained as follows. Given a (usual) connection $\nabla$ on a Lie algebroid $A\Ato M$ one has the following associated $A$-connections on $A$ and $TM$:
\begin{align*}
\overline{\nabla}_\al\be&:=\nabla_{\rho(\be)}\al+[\al,\be], \\
\overline{\nabla}_\al X&:=\rho(\nabla_X\al)+[\rho(\al),X]
\end{align*}
They satisfy 
\[ \rho(\overline{\nabla}_\al\be)=\overline{\nabla}_\al\rho(\be). \]
One defines the {\bf basic curvature} of $\nabla$ to be the tensor
\begin{equation}
\label{eq:basic:curvature:alg} 
R^{\bas}_\nabla(\al,\be)(X):=\nabla_X([\al,\be])-[\nabla_X\al,\be]-[\al,\nabla_X\be]-\nabla_{\overline{\nabla}_\be X}\al+\nabla_{\overline{\nabla}_\al X}\be,
\end{equation}
where $X\in\X(M)$ and $\al,\be\in\Gamma(A)$. When $R^{\bas}_\nabla\equiv 0$ one calls $\nabla$ a Cartan connection and this is precisely the infinitesimal version of a Cartan connection on a Lie groupoid (see, e.g., \cite{AC13,Blaom06,Blaom12}).

\begin{proposition}
\label{prop:Cartan:connection:splitting}
A bundle of ideals $\ka\subset A$ is partially split if and only if there is a vector bundle splitting $l:A\to \ka$ of the associated short exact sequence 
\[ 
\xymatrix{0\ar[r] & \ka \ar[r] & A\ar[r]^{\phi} & B\ar[r] & 0}
\]
and a (usual) connection $\nabla$ on $A$ such that:
\begin{equation}
    \label{eq:partial:splitting:basic:curvature}
    \overline{\nabla}l=0,\qquad l(R^{\bas}_\nabla)=0.
\end{equation} 
In this case, the pair $(l,\nabla)$ determines a IM connection 1-form $(L,l)\in\Omega^1_\imult(A;\ka)$ by setting:
\[ 
L:\Gamma(A)\to \Omega^1(M;\ka), \quad i_X L(\al):=l(\nabla_X\al). 
\]
Moreover, every IM connection 1-form $(L,l)$ takes this form for some pair $(l,\nabla)$ satisfying \eqref{eq:partial:splitting:basic:curvature}.
\end{proposition}

\begin{proof}
Let $l:A\to \ka$ be any splitting of the short exact sequence. Given a map $L:\Gamma(A)\to \Omega^1(M;\ka)$, we claim that the pair $(L,l)$ satisfies the IM-condition \eqref{eq:symbol:IM:form}:
\[ L(f\al)=f L(\al)+\d f\wedge l(\al) \]
if and only if 
\[ i_X L(\al)=l(\nabla_X\al)\]
for some connection $\nabla$ in $A$. Clearly, if $L$ takes this form then \eqref{eq:symbol:IM:form} holds. Conversely, given $(L,l)$ satisfying \eqref{eq:symbol:IM:form} choose any connection $\nabla'$ on $A$ and consider the difference:
\[ D(\al,X):=i_X L(\al)-l(\nabla'_X\al). \]
This is $C^\infty$-linear in both entries and takes values in $\ka\subset A$. Therefore, we can correct the connection $\nabla'$ by setting:
\[ \nabla_X\al:=\nabla'_X\al+D(\al,X), \]
so that $(L,l)$ takes the desired form.

It remains to check that a pair $(L,l)$ defined by $(l,\nabla)$ satisfies the IM-conditions \eqref{eq:compatibility:IM:E:form} if and only if \eqref{eq:partial:splitting:basic:curvature} hold. Applying the definition of $\overline\nabla$, we obtain:
\begin{align*} 
(\overline\nabla_\al l)(\be)&=\overline\nabla_\al (l(\be))-l(\overline\nabla_\al\be)\\
&=[\al,l(\be)]-l(\nabla_{\rho(\be)}\al+[\al,\be])\\
&=\Lie_\al l(\be)-i_{\rho(\be)}L(\al)-l([\al,\be]),
\end{align*}
for any $\al,\be\in\Gamma(A)$. This shows that the condition $\overline\nabla l=0$ amounts to the last IM condition in \eqref{eq:compatibility:IM:E:form}. It also shows that this condition can be written as:
\[ l([\al,\be])=[\al,l(\be)]+l(\nabla_{\rho(\be)}\al), \]
for any $\al,\be\in\Gamma(A)$. Using this last identity, the definition of $L$ and the expression for the basic curvature, we now find:
\begin{align*} 
i_X L([\al,&\be])-l(R^{\bas}_\nabla(\al,\be)(X))=\\
&=l(\nabla_X[\al,\be])-l(R^{\bas}_\nabla(\al,\be)(X))\\
&=l([\nabla_X\al,\be])+l([\al,\nabla_X\be])+l(\nabla_{\overline{\nabla}_\be X}\al)-l(\nabla_{\overline{\nabla}_\al X}\be)\\
&=l([\nabla_X\al,\be])+l([\al,\nabla_X\be])+l(\nabla_{\rho(\nabla_X\be)+[\rho(\be),X]}\al)-l(\nabla_{\rho(\nabla_X \al)+[\rho(\al),X]}\be)\\
&=[\al,l(\nabla_X\be)]-l(\nabla_{[\rho(\al),X]}\be)
-[\be,l(\nabla_X\al)]+l(\nabla_{[\rho(\be),X]}\al)\\
&=i_X\big(\Lie_\al L(\be)-\Lie_\be L(\al)\big).
\end{align*}
So the second IM condition in \eqref{eq:compatibility:IM:E:form} holds if and only if $l(R^{\bas}_\nabla)=0$.
\end{proof}

\begin{remark}
The connection $\nabla$ in the previous proposition is not uniquely determined by the partial IM connection 1-form $(L,l)\in\Omega^1_\imult(A;\ka)$.
\end{remark}

\subsection{Couplings}
\label{sec:obstructions:partial:split}

In this section we give a coupling description of algebroids with a partially split bundle of ideals. This is inspired by (and in fact generalizes) the classical coupling description of symplectic fibrations \cite{McDuffSalamon17} and horizontally non-degenerate Poisson structures \cite{Vorobjev01,Vorobjev05}.

To state the main result in this section we consider a partially split bundle of ideals $\ka\subset A$ with a fixed choice of IM connection 1-form $(L,l)\in \Omega^1_\imult(A,\ka)$. The base map gives a vector bundle splitting of the associated short exact sequence
\[ \xymatrix{0\ar[r]& \ka\ar[r]& A\ar@/^/@{-->}[l]^{l}\ar[r]^{\phi}& B\ar[r]& 0}\]
which we use to identify $A\simeq B\oplus \ka$. Then we associate the following data:
\begin{enumerate}[(i)]
    \item A linear connection $\nabla^L$ on the vector bundle $\ka\to M$, by setting:
    \[ \nabla^L_X\xi:=i_XL(\xi). \]
    \item A tensor $U\in \Gamma(B^*\otimes T^*M\otimes \ka)$ given by:
    \[ U(\alpha,X):=-i_XL(\alpha). \]
\end{enumerate}
That this is a connection and a tensor follows from the symbol equation \eqref{eq:symbol:IM:form} and the fact that $l|_{\ka}=\id$. Notice that the connection $\nabla^L$ has already appeared in Corollary \ref{corollary:Lie:alg:bundle}.

\begin{proposition}
\label{prop:structure:eqs}
The data $(\nabla^L,U)$ associated with a IM connection 1-form $(L,l)\in \Omega^1_\imult(A,\ka)$ satisfies the structure equations:
\begin{enumerate}
\item[(S1)] the connection $\nabla^{L}$ preserves the Lie bracket $[\cdot,\cdot]_{\ka}$, i.e.,
\[\nabla^{L}_{X}[\xi,\eta]_{\ka}=[\nabla^{L}_{X}\xi,\eta]_{\ka}+[\xi,\nabla^{L}_{X}\eta]_{\ka};\]
\item[(S2)] the curvature of $\nabla^{L}$ is related to $\mathrm{ad}_{U}$ as follows:
\[\nabla^{L}_{\rho_B(\alpha)}\nabla^{L}_{X}-\nabla^{L}_{X}\nabla^{L}_{\rho_B(\alpha)}-\nabla^{L}_{[\rho_B(\alpha),X]}=[U(\alpha,X),\cdot ]_{\ka};\]
\item[(S3)] $U$ satisfies the ``mixed'' cocyle-type equation:
\begin{align*}
\nabla^{L}_{\rho_B(\alpha)}U(\beta,X)&-\nabla^L_{\rho_B(\beta)}U(\alpha,X)+\nabla^L_{X}U(\alpha,\rho_B(\beta))\\
&+U(\alpha,[\rho_B(\beta),X])-U(\beta,[\rho_B(\alpha),X])=U([\alpha,\beta]_{B},X),
\end{align*}
\end{enumerate}
for all $X\in \X^1(M)$, $\alpha,\beta\in \Gamma(B)$, $\xi,\eta\in \Gamma(\ka)$. 
\end{proposition}

The proof of this proposition follows immediately from the IM conditions (\ref{eq:compatibility:IM:E:form}) and is left to the reader.

\begin{definition}
Let $ B\Ato M$ be a Lie algebroid and $(\ka,[\cdot,\cdot]_{\ka})$ a Lie algebra bundle over $M$. A \textbf{coupling data} is a pair $(\nabla^L,U)$, where $\nabla^L$ is a connection on $\ka$ and $U\in \Gamma(B^*\otimes T^*M\otimes \ka)$ is a tensor field satisfying the structure equations (S1), (S2) and (S3).
\end{definition}

The following result shows that the coupling data gives a way of codifying algebroids with a partially split bundle of ideals. 

\begin{proposition}\label{prop:operators:splitting}
Let  $B\Ato M$ be a Lie algebroid, $(\ka,[\cdot,\cdot]_{\ka})$ a Lie algebra bundle over $M$ and $(U,\nabla^L)$ coupling data. Then $A:=B\oplus \ka$ is a Lie algebroid, with anchor $\rho_{A}(\alpha,\xi):=\rho_B(\alpha)$ and Lie bracket:
\begin{equation}\label{equation:bracket:splitting}
[(\alpha,\xi),(\beta,\eta)]_{A}:=([\alpha,\beta]_{B},U(\alpha,\rho_B(\beta))+\nabla^{L}_{\rho_B(\alpha)}\eta-\nabla^{L}_{\rho_B(\beta)}\xi+[\xi,\eta]_{\ka}).
\end{equation}
Then $\ka\subset A$ is partially split, with IM connection 1-form $(L,l)$ given by:
\begin{equation}\label{eq:L:nabla:U}
l=\mathrm{pr}_{\ka}, \quad i_XL(\alpha,\xi)=\nabla^L_{X}\xi-U(\alpha,X).
\end{equation}

Moreover, any Lie algebroid $A\Ato M$ together with a partially split bundle of ideals $\ka\subset A$ is isomorphic to one of this type.
\end{proposition}
\begin{proof}
Assume we are given coupling data $(U,\nabla^L)$. The skew-symmetrization with respect to $\alpha$ and $\beta$ of (S3) gives:
\[\nabla^L_X\big(U(\alpha,\rho_B(\beta))+ U(\beta,\rho_B(\alpha))\big)=0,\]
which yields:
\begin{equation}\label{eq:U:skew} 
U(\alpha,\rho_B(\beta))=- U(\beta,\rho_B(\alpha)).
\end{equation}
Hence \eqref{equation:bracket:splitting} defines a skew-symmetric bracket. It is immediate to check that it satisfies the Leibniz identity and that its Jacobiator is given by:
\begin{align*}
     J((0,\xi),(0,\eta),(0,\nu))&=(0,J_{\ka}(\al,\eta,\nu))\\
     J((\al,0),(0,\xi),(0,\eta))&=(0,C_1(\al,\xi,\eta))\\
     J((\al,0),(\be,0),(0,\xi))&=(0,C_2(\al,\be,\xi)\\
     J((\al,0),(\be,0),(\ga,0))&=(J_B(\al,\be,\ga),C_3(\al,\be,\ga))
\end{align*}
where $J_\ka$ and $J_B$ are the Jacobiators of the brackets $[\cdot,\cdot]_\ka$ and $[\cdot,\cdot]_B$ and
\begin{align*}
    C_1(\al,\xi,\eta)&:=\nabla^{L}_{\rho_B(\al)}[\xi,\eta]_{\ka}-[\nabla^{L}_{\rho_B(\al)}\xi,\eta]_{\ka}-[\xi,\nabla^{L}_{\rho_B(\al)}\eta]_{\ka}\\
    C_2(\al,\be,\xi)&:=\nabla^{L}_{\rho_B(\alpha)}\nabla^{L}_{\rho_B(\be)}\xi-\nabla^{L}_{\rho_B(\be)}\nabla^{L}_{\rho_B(\alpha)}\xi-\nabla^{L}_{[\rho_B(\al),\rho_B(\be)]}\xi-[U(\alpha,\rho_B(\be)),\xi]_{\ka}\\
    C_3(\al,\be,\ga)&:=\nabla^{L}_{\rho_B(\alpha)}U(\beta,\rho_B(\gamma))+U(\alpha,[\rho_B(\beta),\rho_B(\gamma)])
    -\nabla^L_{\rho_B(\beta)}U(\alpha,\rho_B(\gamma))\\ &\qquad\quad -U(\beta,[\rho_B(\alpha),\rho_B(\gamma)])+\nabla^L_{\rho_B(\gamma)}U(\alpha,\rho_B(\beta))-U([\alpha,\beta]_{B},\rho_B(\gamma))
\end{align*}
If in the structure equations (S1), (S2) and (S3) we replace $X$ by $\rho_B(\al)$, $\rho_B(\be)$ and $\rho_B(\ga)$, respectively, we obtain that these 3 expressions vanish. Hence, the bracket \eqref{eq:L:nabla:U} satisfies the Jacobi identity, so we obtain a Lie algebroid.
If we now define $l:A\to\ka$ and $L:\Gamma(A)\to \Omega^1(M,\ka)$ by \eqref{eq:L:nabla:U}, one checks that the symbol equation \eqref{eq:symbol:IM:form} holds and that the IM conditions (\ref{eq:compatibility:IM:E:form}) are equivalent to (S1)-(S3).

Finally, if we start with a partially split bundle of ideals $\ka\subset A$, as we saw above, a choice of IM connection 1-form $(L,l)\in \Omega^1_\imult(A,\ka)$ determines coupling data satisfying the structure equations and a vector bundle isomorphism $A\simeq B\oplus\ka$ where $B\simeq\ker l$. The anchor of $A$ is then given by $\rho_{A}(\alpha,\xi):=\rho_B(\alpha)$ and all we have to check is that, under this isomorphism, the Lie bracket becomes \eqref{eq:L:nabla:U}. For this we use the third equation in the IM condition (\ref{eq:compatibility:IM:E:form}) for $(L,l)$ and the definitions of $\nabla^L$ and $U$ to conclude that:
\begin{itemize}
    \item[\tiny$\bullet$] If $\al\in\Gamma(B)\simeq\Gamma(\ker l)$ and $\xi\in\Gamma(\ka)$ then: \[ [\al,\xi]_{A}=\nabla^L_{\rho_B(\al)}\xi;\]
    \item[\tiny$\bullet$] If $\al,\be\in\Gamma(B)\simeq\Gamma(\ker l)$ then:
    \[ [\al,\be]_{A}=[\al,\be]_{B}+U(\al,\rho_B(\be)). \]
\end{itemize}
Formula \eqref{eq:L:nabla:U} for the bracket follows from these since $\ka$ is a bundle of ideals.
\end{proof}

\subsection{Flat IM Ehresmann connections}
\label{sec:flat}

Let $\ka\subset A$ be a bundle of ideals and let $B=A/\ka$. As we saw in Section \ref{sec:obstructions:partial:split}, a choice of IM connection 1-form $(L,l)\in \Omega^1_\imult(A,\ka)$ gives rise to the  coupling data $(\nabla^L,U)$, where $U\in\Gamma(B^*\otimes T^*M\otimes \ka)$ is a tensor field and $\nabla^L$ is a usual connection on $\ka$. Using Proposition \ref{prop:operators:splitting}, it is not hard to check that we have a $\ka$-valued IM 2-form
$(\tilde{L},\tilde{l})\in\Omega^2_\imult(A,\ka)$, where:
\begin{equation}
    \label{eq:IM:curvature}
    \tilde{L}(\al,\xi)=R^{\nabla^L}\cdot\xi-\d^{\nabla^L}U(\al),\qquad \tilde{l}(\al,\xi)=-U(\al),
\end{equation}
for $\al\in\Gamma(B)$ and $\xi\in\Gamma(\ka)$. This is the infinitesimal version of the curvature of the multiplicative Ehresmann connection from Proposition \ref{prop:curvature:multiplicative}. Namely, in the case where $\phi:A\to B$ is induced by a groupoid morphism $\Phi:\G\to \H$ and the IM Ehresmann connection is induced by a multiplicative Ehresmann $E$ connection, \eqref{eq:IM:curvature} is the IM form corresponding to curvature 2-form $\Omega$ of $E$. 

\begin{definition}
An IM Ehresmann connection $(L,l)\in\Omega^1_\imult(A;\ka)$ with associated coupling data $(\nabla^L,U)$ is called:
\begin{enumerate}[(i)]
    \item {\bf totally flat} if $R^{\nabla^L}=0$ and $U=0$;
\item {\bf leafwise flat} if the splitting $B\to A$, $\al\mapsto (\al,0)$, is a Lie algebroid map, i.e., $U(\al,\rho_B(\be))=0$;
        \item {\bf kernel flat} if the associated linear connection $\nabla^L$ is flat: $R^{\nabla^{L}}=0$.
\end{enumerate}
\end{definition}

The first condition is equivalent to involutivity of the distribution $E\subset TA$ corresponding to the IM Ehresmann connection. This condition is very restrictive, as the following infinitesimal analogue of Proposition \ref{prop:totally:flat:grpd} shows: 

\begin{proposition}
If a bundle of ideals $\ka\subset A$ admits a totally flat IM Ehresmann connection, then there is a covering space of the base $p:\widetilde{M}\to M$ such that the pullback of $A$ to $\widetilde{M}$ is isomorphic to a product $p^*A\simeq p^*B\times \gg$, where $p^*\ka\simeq \widetilde{M}\times \gg$.
\end{proposition}

\begin{proof}
Let $(\nabla^L,U=0)$ be a totally flat coupling data. Let $p:\widetilde{M}\to M$ denote the holonomy cover of $\nabla^L$. Parallel transport along $\nabla^L$ induces a trivialization of the bundle of Lie algebras $p^*\ka\simeq \widetilde{M}\times \gg$, where $\gg=\ka_x$, for some $x\in M$. Since $\nabla^L$ becomes the trivial connection on $\widetilde{M}\times \gg$, formulas \eqref{equation:bracket:splitting} show that the algebroid $p^*A$ is a product (see also the example from Subsection \ref{ex:product}).
\end{proof}

The second flatness condition can be understood as a leafwise version of the first:

\begin{proposition}
Given a bundle of ideals $\ka\subset A$, an IM Ehresmann connection $(L,l)\in\Omega^1_\imult(A;\ka)$ is leafwise flat if and only if, for any leaf $\O\subset M$ of the Lie algebroid $A$, the pullback IM Ehresmann connection $(L,l)|_{\O}\in \Omega^1_\imult(A|_{\O};\ka|_\O)$ is totally flat.
\end{proposition}
\begin{proof}
Given a leaf $\O$ of $A$, the coupling data associated to the pullback of $(L,l)$ to $A|_\O$ is $(\nabla^\O,U^\O)$, where $\nabla^\O$ is the pullback of the connection $\nabla^L$ to $\ka|_\O$ and $U^\O$ is the restriction of $U$ to $B|_\O\otimes T\O$. Now, $(L,l)$ being leafwise flat, i.e., \[U(\alpha,\rho_B(\beta))=0,  \quad \textrm{for all}\ \alpha,\beta\in B,\]
is clearly equivalent to $U^\O=0$ for all leaves $\O\subset M$. On the other hand, using that $A|_\O$ is transitive and (S2), we see that $U^\O=0$ implies that $\nabla^\O$ is flat. 
\end{proof}

Note that a leafwise flat IM Ehresmann connection turns $A$ into a semi-direct product, in the sense of the following. 

Let $B\Ato M$ be a Lie algebroid which acts on a bundle of Lie algebras $(\ka,[\cdot,\cdot]_{\ka})\to M$ by infinitesimal Lie algebra automorphisms. The \textbf{semi-direct product} is the Lie algebroid $A:=B\times_M\ka\Ato M$ with Lie bracket: 
\[ [(\al,\xi),(\be,\eta)]=([\al,\be],\nabla^{\ka}_{\al}\eta-\nabla^{\ka}_{\beta}\xi+[\xi,\eta]_{\ka})\]
and anchor $\rho_B\circ \pr_B$. Note that the short exact sequence 
\[
\begin{tikzcd}
\ka\ar[r, hook] & B\times_M\ka \ar[r, two heads, "\pr_1"]& B
\end{tikzcd}
\]
is canonically split by a Lie algebroid morphism:
\[ \sigma:B\longrightarrow B\times_M \ka, \quad \al \mapsto (\al,0).\]

Conversely, given a surjective Lie algebroid map $\phi:A\to B$ covering the identity, which admits a \textbf{Lie algebroid splitting}:
\[
\begin{tikzcd}
\ka\ar[r, hook] & A \ar[r, two heads, "\phi"] & B\ar[l, dashed, bend left, "\sigma"]
\end{tikzcd}
\]
we obtain an action of $B\Ato M$ on the bundle of Lie algebras $(\ka,[\cdot,\cdot]_{\ka})\to M$:
\[ \nabla_{\al}^{\ka}\xi:=[\sigma(\al),\xi]_{A},\]
preserving the Lie algebra structure, and a Lie algebroid isomorphism:
\[ B\times_M \ka \simeq A,\quad (\al,\xi)\mapsto \sigma(\al)+\xi.\]
Notice that the Lie algebroid splitting $\sigma:B\to A$ allows to pull back $\ka$-valued IM forms on $A$ to IM forms on $B$:
\[\sigma^*:\Omega_{\imult}^{\bullet}(A;\ka)\rmap \Omega_{\imult}^{\bullet}(B;\ka).\]
Therefore we conclude the following:

\begin{corollary}
Given a bundle of ideals $\ka\subset A$, with an IM Ehresmann connection $(L,l)\in\Omega^1_\imult(A;\ka)$ that is leafwise flat, we obtain a IM 1-form $(U,0)\in \Omega_{\imult}^1(B;\ka)$.
\end{corollary}

The following infinitesimal version of Proposition \ref{prop:invariant:connection} gives yet another notion of flatness, weaker than totally flat and stronger than leafwise flat: 

\begin{proposition}\label{prop:invarinat:connection:inf}
Let $\ka\subset A$ be a bundle of ideals with quotient $B:=A/\ka$. The following are equivalent:
\begin{enumerate}[(i)]
    \item There is an IM Ehresmann connection $(L,l)\in\Omega^1(A;\ka)$ on $A$ and $\pr:A\to B$ whose associated splitting $\sigma:B\to A$ is a horizontal algebroid splitting: i.e., $\sigma^*(L,l)=0$;
    \item There exists an IM Ehresmann connection $(L,l)\in\Omega^1_\imult(A;\ka)$ with coupling data $(\nabla^L,U=0)$; 
    \item $A$ is isomorphic to a semi-direct product $A\simeq B\times_M\ka$ and $\ka$ admits an $B$-invariant connection $\nabla^L$ preserving the bracket $[\cdot,\cdot]_{\ka}$ (see Definition \ref{def:A-inv:conn}).
\end{enumerate}
\end{proposition}

\begin{proof}
Note that $\sigma^*(L,l)=(U,0)$, so (i) and (ii) are equivalent.

If $U=0$, then the formula for the Lie bracket on $A$ \eqref{equation:bracket:splitting} together with (S1), (S2) are equivalent to $\nabla^L$ being a $B$ invariant connection on $\ka$ which preserves $[\cdot,\cdot]_{\ka}$.

Conversely, if $A\simeq B\times_M\ka$, and the representation of $B$ on $\ka$ admits a $B$-invariant connection $\nabla^L$ which preserves $[\cdot,\cdot]_{\ka}$, then $(\nabla^L,U=0)$ is a coupling data for an IM Ehresmann connection.
\end{proof}

Finally, we look at the kernel flat IM Ehresmann connections. For this, observe that the representation $\nabla^\ka$ of the Lie algebroid $A$ restricts to an $A$-representation on the center $z(\ka)$. While the representation of $A$ on the whole $\ka$, in general, does not factor to a representation of $B$, its restriction to the center $z(\ka)$ does factor. We still denote this representation by $\nabla^\ka$.
In terms of the coupling the representation of $B$ on the center $z(\ka)$ is given by:
\[ \nabla^{\ka}_{\al}\xi=\nabla^L_{\rho_B(\al)}\xi, \]
for all $\al\in\Gamma(B)$, $\xi\in\Gamma(z(\ka))$. This follows from the expression of the bracket given in Proposition \ref{prop:operators:splitting}.

In the kernel flat case, the structure equations (S1)-(S3) of Proposition \ref{prop:structure:eqs} simplify to:

\begin{proposition}
\label{prop:flat:jets}
Let $\ka\subset A$ be a partially split bundle of ideals with coupling data $(\nabla^L,U)$. If $\nabla^L$ is flat, then:
\begin{itemize}
    \item[(S1')] The connection $\nabla^L$ preserves the bracket $[\cdot,\cdot]_\ka$;
    \item[(S2')] The tensor $U$ takes values in the center of $\ka$;
    \item[(S3')] The pair $(\d^{\nabla^L} U,U)$ is a $z(\ka)$-valued IM 2-form on $B$.
\end{itemize}
Conversely, given a Lie algebroid $B\Ato M$, a Lie algebra bundle $\ka\to M$ with a flat connection $\nabla^L$, and a tensor field $U\in\Gamma(B^*\otimes T^*M\otimes \ka)$ satisfying (S1')-(S3'),
then $\ka$ is a partially split bundle of ideals in Lie algebroid $A=B\oplus\ka$ with kernel flat coupling data $(\nabla^L,U)$.
\end{proposition}

\begin{proof}
If $\nabla^L$ is flat, it is obvious that (S1)-(S2) are equivalent to (S1')-(S2').

Now consider the pair $(\d^{\nabla^L} U,U)$. The properties of $\d^{\nabla^L}$ give immediately that:
 \[ \d^{\nabla^L} U(f\al)=f\d^{\nabla^L} U(\al)+\d f\wedge U(\al), \]
 so $U:B\to T^*M\otimes \ka$ is the symbol of $\d^{\nabla^L} U:\Gamma(B)\to\Omega^2(M,\ka)$, i.e., \eqref{eq:symbol:IM:form} holds.
 The skew-symmetry of $U$, given by \eqref{eq:U:skew}, can be written as:
 \[ i_{\rho_B(\be)}U(\al)=-i_{\rho_B(\al)}U(\be), \]
 which is the first equation in the IM conditions \eqref{eq:Lie:derivative:rep}. Then (S3) can be rewritten
using the connection $\nabla^L$ as:
\[ U([\al,\be]_{B})=\Lie^{\nabla^L}_{\rho_B(\al)}U(\be)-i_{\rho_B(\be)}\d^{\nabla^L}U(\al), \]
and since $\Lie_\al=\Lie^{\nabla^L}_{\rho_B(\al)}$ this gives the third equation in the IM conditions \eqref{eq:Lie:derivative:rep}. 

Finally, applying $\d^\nabla$ to this last equation, using that its square is zero and ``Cartan's magic formula'' holds for $\Lie^{\nabla^L}$, we obtain:
\begin{align*}
    \d^{\nabla^L}U([\al,\be]_{B})&=\d^{\nabla^L}\Lie_{\al}U(\be)-\d^{\nabla^L}i_{\rho_B(\be)}\d^{\nabla^L}U(\al)\\
    &=\Lie_{\al}\d^{\nabla^L}U(\be)-\Lie^{\nabla^L}_{\rho_B(\be)}\d^{\nabla^L}U(\al)\\
    &=\Lie_{\al}\d^{\nabla^L}U(\be)-\Lie_{\be}\d^{\nabla^L}U(\al),
\end{align*}
which is the second equation in the IM conditions \eqref{eq:Lie:derivative:rep}. Hence, $(\d^{\nabla^L} U,U)$ is a $z(\ka)$-valued IM 2-form on $B$.

It should clear from this that the converse also holds: a pair $(\nabla^L,U)$, with $\nabla^L$ flat and satisfying (S1')-(S3'), also satisfies (S1)-(S3). Hence, the last part follows from Proposition \ref{prop:operators:splitting}. 
\end{proof}

\begin{remark}\label{rem:cocycle:not:invariant}
For a kernel flat partial splitting of $\ka\subset A$, the proposition yields the IM valued 2-form:
\begin{equation}\label{eq:IM:2:form:for:kernel:flat}
(\d^{\nabla^L} U,U)\in \Omega^2_{\imult}(B;z(\ka)),
\end{equation}
which is the infinitesimal analogue of the class from Proposition \ref{prop:curvature:basic}. Namely, the IM form corresponding the multiplicative form $\underline{\Omega}$ is $(-\d^{\nabla^L}U,-U)$ -- see \eqref{eq:IM:curvature}.

Since $\nabla^L$ is flat, by Proposition \ref{prop:differ:IM:forms} we have an induced differential:
\[\d^{\nabla^L}_{\imult}:  \Omega^{\bullet}_{\imult}(B;z(\ka))\rmap 
\Omega^{\bullet+1}_{\imult}(B;z(\ka)),\]
for which the element \eqref{eq:IM:2:form:for:kernel:flat} is clearly closed. Applying the chain map from Lemma \ref{lemma:chain:map:IM:Lie:alg} we obtain a $B$-cocycle with values in $z(\ka)$:
\[\lambda\in \Omega^2(B;z(\ka)),\quad \lambda(\al,\be):=U(\al,\rho_B(\be)).\]
That $\lambda$ is a cocycle can be seen directly also from (S3). Note that the classes 
\[[(\d^{\nabla^L}U,U)]\in H^2_{\imult}(B;z(\ka))\quad \textrm{and} \quad [\lambda]\in H^2(B;z(\ka))\]
depend on the choice of the IM connection.

\end{remark}

The remark can be developed into a criterion for the existence of kernel flat IM connection for \textbf{abelian bundle of ideals} $\ka\subset A$. In this case, the canonical action of $B:=A/\ka$ is on $\ka=z(\ka)$ and, after choosing a splitting $A\simeq B\oplus \ka$, the Lie bracket takes the form:
\[ 
[(\alpha,\xi),(\beta,\eta)]_{A}:=([\alpha,\beta]_{B},\lambda(\alpha,\beta)+\nabla^\ka_{\alpha}\eta-\nabla^\ka_{\beta}\xi).
\]
Then $\lambda\in\Omega^2(B;\ka)$ is a 2-cocycle, whose cohomology class:\[c_2(A):=[\lambda]\in H^2(B;\ka)\]
is independent of the splitting, and it determines the extension $\ka\to A\to B$ up to isomorphism. We have the following: 

\begin{proposition}
\label{corollary:abelian:flat:type}
An abelian bundle of ideals $\ka\subset A$, with $B:=A/\ka$, admits a kernel flat partial splitting if and only if $\ka$ admits a flat connection $\nabla$ inducing $\nabla^\ka$ and the class $c_2(A)$ is in the image of the map (from Lemma \ref{lemma:chain:map:IM:Lie:alg}):
\[ H_{\imult}^2(B;\ka) \to H^2(B;\ka). \]
\end{proposition}

\begin{proof}
Given a flat coupling data $(\nabla^L,U)$, the connection $\nabla^L$ induces $\nabla^\ka$, and, as remarked, $c_2(A)$ is the image of the class $[(\d^{\nabla^L}U,U)]\in H_{\imult}^2(B;\ka)$.

Conversely, assume one is given a flat connection $\nabla^L$ inducing $\nabla^\ka$ and that $c_2(A)$ is the image of $c\in H^2_{\imult}(B;\ka)$. Note that any representative of $c$ if of the form $(\d^{\nabla^L}U,U)$. Choose a splitting inducing
\[ \lambda(\al,\be)=U(\al,\rho_B(\be)).\]
Since the pair $(\nabla^L,U)$ satisfies (S1')-(S3'), it follows from Proposition \ref{prop:operators:splitting} that it is the coupling data of a kernel flat partial splitting for $A$.
\end{proof}

We also have the infinitesimal version of Corollary \ref{corollary:partial:split:for:semi-direct}:

\begin{corollary}
\label{corollary:abelian:semi-direct product}
A semi-direct product $A=B\times_M \ka$, with $\ka$ abelian is partially split if and only if the representation of $B$ on $\ka$ admits a $B$-invariant connection (in the sense of Definition \ref{def:A-inv:conn}).
\end{corollary}
\begin{proof}
If $\ka$ admits a $B$-invariant connection $\nabla^L$, then $(\nabla^L,U=0)$ is the coupling data for an IM connection on $B\times_M \ka$. 

Conversely, let $(L,l)\in \Omega^1(B\times_M \ka;\ka)$ be an IM Ehresmann connection. Both the projection $\phi:B\times_M \ka \to B$ and the inclusion $\sigma:B\to B\times_M \ka$ are Lie algebroid morphisms which intertwine the actions on $\ka$. Therefore, 
$\phi^*\circ \sigma^*(L,l)$ is an IM 1-form. By subtracting it from $(L,l)$, we may assume that 
$\sigma^*(L,l)=0$. By Proposition \ref{prop:invarinat:connection:inf}, this is equivalent to $(L,l)$ having coupling data $(\nabla^L,U=0)$, with $\nabla^L$ a $B$-invariant connection on $\ka$.
\end{proof}

\section{Examples and applications (algebroids)}
\label{sec:examples}

Using the correspondence from Theorem \ref{thm:Lie:functor:connections}, many of the examples that follow can be seen as infinitesimal counterparts of the examples discussed in Section \ref{examples:groupoids}. Note, however, that we will make no integrability assumption, so often we obtain more general classes of examples. 


\subsection{Lie algebra bundles}
Let $A\to M$ be a Lie algebra bundle. A bundle of ideals $\ka\subset A$ is not necessarily partially split. The obstructions  are precisely those found in Corollaries \ref{corollary:Lie:alg:bundle} and \ref{corollary:isotropy:splits}:

\begin{proposition}\label{prop:Lie:alegebra:bundle}
A bundle of ideals $\ka\subset A$ in a Lie algebra bundle
is partially split if and only if the following two conditions hold: 
\begin{enumerate}[(i)]
\item $\ka$ is a Lie algebra bundle;
\item there exists a splitting $A\simeq B\oplus \ka$, for which the fiberwise Lie bracket is a direct product.
\end{enumerate}
\end{proposition}

\begin{proof}
By Corollaries \ref{corollary:Lie:alg:bundle} and \ref{corollary:isotropy:splits}, a partial splitting for $\ka\subset A$ yields a (usual) connection $\nabla^L$ on $\ka$ which preserves the Lie bracket and a splitting $A\simeq B\oplus \ka$, for which the fiberwise Lie bracket is a direct product. 

Conversely, fix a connection $\nabla^L$ on $\ka$ which preserves the Lie bracket and a splitting $A\simeq B\oplus \ka$ for which the Lie bracket is a direct product. Then we can obtain a partial splitting from Proposition \ref{prop:structure:eqs} by setting $U:=0$.
\end{proof}

\begin{remark}
\label{rem:h2:algebra:bundle}
If $M$ is connected, then for a bundle of Lie algebras $(\ka,[\cdot,\cdot]_{\ka})\to M$ to be a Lie algebra bundle, it suffices that: 
\[H^2(\ka_x,\ka_x)=0,\quad \forall\, x\in M.\]
Indeed, this condition implies that each $\ka_x$ is a \textbf{rigid} Lie algebra: any small deformation of the Lie bracket on $\ka_x$ is isomorphic to $\ka_x$ (see, e.g., \cite{CrSchStr}). Therefore, all Lie algebras in $\ka$ are isomorphic to each other and it follows that $\ka$ is a bundle of Lie algebras. 
\end{remark}

\subsection{Products}
\label{ex:family with no limit}
\label{ex:product}

Consider the Lie algebroid $A\Ato M$ obtained as the product of a Lie algebroid $B\Ato M$ and a Lie algebra $\gg$:
\[ A=B\times \gg. \] 
 Then $\ka:=M\times \gg\subset A$ is a bundle of Lie algebras which is always partially split with a canonical, totally flat partial splitting. This follows by observing that we have canonical isomorphisms:
\[ TA=TB\times (\gg\ltimes \gg), \quad A\times_M\ka=B\times(\gg\ltimes \gg),\]
in terms of which the inclusion \eqref{inclusion:VB:subalgebroid} is
\[ B\times(\gg\ltimes \gg)\to TB\times (\gg\ltimes \gg),\quad (b,v,w)\mapsto (0_b,v,w). \]
Hence, we have the canonical IM Ehresmann connection $E\subset TA$ given by:
\[ E_{(b,v)}:=\{(X,v,0): X\in T_b(B_x), v\in\gg\}. \]
It is also easy to determine the alternative data of specifying this connection given by Proposition \ref{prop:partially:split:algbrd}, namely:
\begin{itemize}
    \item[\tiny$\bullet$] The canonical splitting $\theta:A\ltimes\ka^*\to T^*A$:
    \[ \theta:B\times (\gg\ltimes \gg^*)\to T^*(B)\times (\gg\ltimes \gg^*), \quad (b,v,\xi)\mapsto (0_b,v,\xi). \]
    \item[\tiny$\bullet$] The IM connection 1-form $(L,l)\in\Omega^1_\imult(A,\ka)$:
    \begin{align*}
        &l:B\times\gg\to M\times \gg, \quad L(b_x,v)=(x,v),\\
        &L:\Gamma(B)\times C^\infty(M;\gg)\to \Omega^1(M,\gg), \quad L(\al,f)=\d f;
    \end{align*} 
    \item[\tiny$\bullet$] The linear, closed, IM 2-form $\mu\in\Omega^2_\imult(A\ltimes\ka^*)$:
    \[ \mu:B\times (\gg\ltimes \gg^*)\to T^*(M\times \gg^*)=T^*M\times (\gg\ltimes \gg^*),\quad  (\al,v,\xi)\mapsto (0,v,\xi). \]
\end{itemize}

\subsection{Transitive algebroids}
\label{ex:transitive}
Let $A\Ato M$ be a transitive Lie algebroid. Its isotropy $\ka=\ker\rho$ is a bundle of ideals. We claim that 
this is always partially split, and there is a 1-to-1 correspondence between splittings of the anchor:
\[ \xymatrix{
0\ar[r] & \ka \ar[r] & A \ar[r]^{\rho} \ar@/^/@{-->}[l]^l & TM\ar[r]  \ar@/^/@{-->}[l]^{\tau} & 0}
\]
and IM connection 1-forms $(L,l)\in\Omega^1_\imult(A,\ka)$. 

Indeed, given a splitting $l$, we can define a linear operator $L:\Gamma(A)\to\Omega^1(M,\ka)$ by setting:
\begin{equation}\label{eq:L:intermsof:l}
 i_XL(\al):=l([\tau(X),\al]).    
\end{equation}
One checks easily that the pair $(L,l)$ satisfies \eqref{eq:compatibility:IM:E:form}, so it is a $\ka$-valued, IM 1-form, with $l|_\ka=\id$. 


Conversely, given any IM connection 1-form $(L',l)\in\Omega^1_\imult(A,\ka)$, the bundle map $l:A\to \ka$ determines a splitting of the anchor. We claim that $L'=L$, where $L$ is given by \eqref{eq:L:intermsof:l}. Note that the difference $L-L'$ is a $\ka$-valued, IM 1-form whose symbol vanishes. It follows from the last equation in \eqref{eq:compatibility:IM:E:form} that:
\[ i_{\rho(\be)}(L(\al)-L'(\al))=0, \quad \forall \be\in A. \]
Since $\rho$ is surjective, this means that $L=L'$.

\subsection{Cartan connections}\label{sec:Cartan}
Recall that a \emph{Cartan connection} is a connection $\nabla$ on a Lie algebroid $A$ whose basic curvature vanishes identically: 
\[ R^{\bas}_\nabla\equiv 0. \] 
Hence, as an immediate consequence of Proposition \ref{prop:Cartan:connection:splitting}, we obtain:

\begin{corollary}
\label{cor:Cartan:connection}
Let $\ka\subset A$ be a bundle of ideals. If $A$ admits a Cartan connection $\nabla$ and a splitting $l:A\to \ka$ which is $\overline{\nabla}$-invariant, then $\ka$ is partially split. In particular, this holds when $A$ admits a fiberwise metric that is $\overline{\nabla}$-invariant.
\end{corollary}

\begin{proof}
The second part of the statement follows by choosing the splitting $l:A\to \ka$ to be the orthogonal projection relative to the invariant metric.
\end{proof}

In Subsection \ref{ex:bi:invariant:metrics} we saw that, for a Lie groupoid admitting a bi-invariant metric, every bundle of ideals is partially split. At the Lie algebroid level, the corresponding notion of a bi-invariant metric (see \cite{KotovStrobl18}) is given by a Cartan connection and a pair of metrics $(\eta^A,\eta^M)$ on $A$ and $M$ satisfying:
\[ \overline{\nabla}\eta^A=0,\quad \overline{\nabla}\eta^M=0. \]
The previous corollary shows that for a Lie algebroid $A$ admitting a bi-invariant metric any bundle of ideals $\ka\subset A$ is partially split. Notice that the metric $\eta^M$ plays here no role.

\subsection{Action Lie algebroids}
Let $A=\gg\ltimes M\Ato M$ be the action Lie algebroid associated with a Lie algebra action $\rho:\gg\to\X(M)$. The canonical flat connection $\nabla$ on $A$ has vanishing basic curvature, hence it is a Cartan connection.
Given a bundle of ideals $\ka\subset A$ a splitting $l:A\to \ka$ is $\overline{\nabla}$-invariant if and only if it is $\gg$-equivariant:
\[ l([v,w]_\gg)=[v,l(w)]_{\gg\ltimes M}, \]
for all $v,w\in\gg$ (here we identify elements of $\gg$ with constant sections of $A$). Hence, we deduce from Corollary \ref{cor:Cartan:connection} the infinitesimal version of Proposition \ref{prop:action}: 

\begin{corollary}
A bundle of ideals on an action algebroid $\gg\ltimes M\Ato M$ admitting a $\gg$-equivariant splitting is partially split.
\end{corollary}

For example, if $\gg$ is a Lie algebra of compact type then it admits an $\ad$-invariant scalar product $\langle\cdot,\cdot\rangle$. Such an inner product defines a fiberwise metric $\eta$ on $A$ which is $\overline{\nabla}$-invariant. Hence, in this case, any bundle of ideals on $A=\gg\ltimes M$ admits a $\gg$-equivariant splitting, so it is partially split.

\subsection{Infinitesimal principal type}
\label{ex:principal:type:bundle:local:model}
\label{ex:principal:type:Lie:algebroids}

Let $A'\Ato M$ be a transitive Lie algebroid and $B\Ato M$ any Lie algebroid. Consider the Lie algebroid fibre product:
\[A:=A'\times_{TM}B:=\{(\alpha,\beta)\, :\, \rho_{A'}(\alpha)=\rho_B(\beta)\}, \]
where the structure is such that the inclusion in the product $A\hookrightarrow  A'\times B$ is a Lie algebroid morphism. The projection $\phi:=\pr_{B}:A\to B$ is a Lie algebroid morphism, which is surjective  because we assume $B$ to be transitive. The kernel $\ka:=\ker\phi$ is then a bundle of ideals which can be identified with $\ker\rho_{A'}$ via $\pr_{A'}$. A bundle of ideals $\ka\subset A$ obtained via this construction will be called of \textbf{principal type}. This is of course the infinitesimal counterpart of groupoids with bundle of ideals of principal type, discussed in Subsection \ref{ex:princ:type:part:split}.

Since $A' \Ato M$ is transitive, as discussed in Subsection \ref{ex:transitive}, a splitting $l_{A'}:A'\to\ka$ of $\rho_{A'}$ determines an IM connection 1-form $(L_{A'},l_{A'})\in\Omega^1_\imult({A'},\ka)$. Pulling back this connection 1-form to $A$, we obtain a multiplicative $\ka$-valued 1-form 
\[(L:=L_{A'}\circ\pr_{A'}, l:=l_{A'}\circ \mathrm{pr}_{A'})\in \Omega^1_{\imult}(A,\ka),\]
which satisfies $l|_{\ka}=\id_{\ka}$. So bundles of ideals of principal type are partially split. 

The IM connection 1-form $(L,l)\in\Omega^1_\imult(A,\ka)$ can also be described using the fact that the splitting $l_{A'}$ gives an identification ${A'}\simeq TM\oplus \ka$, where the anchor becomes $\pr_{TM}$ and the bracket is given by:
\begin{equation}\label{eq:bracket:transitive}
[(X,\xi),(Y,\eta)]_{{A'}}=([X,Y],\Omega(X,Y)+\nabla_{X}\eta-\nabla_{Y}\xi+[\xi,\eta]_{\ka}),
\end{equation}
for all $X,Y\in \X^1(M)$, $\xi,\eta\in \Gamma(\ka)$. Here: 
\begin{enumerate}
\item[-] $\Omega$ is $C^{\infty}(M)$-bilinear, so that $\Omega\in \Omega^2(M;\ka)$;
\item[-] $\nabla$ is a connection on $\ka$ preserving $[\cdot,\cdot]_\ka$ with curvature $R^{\nabla}=\ad(\Omega)$.
\end{enumerate}
Then, one finds that, for $\al\in \Gamma(B)$ and $\xi\in \Gamma(\ka)$:
\[i_XL(\alpha,\xi)=\nabla_X(\xi)+\Omega(X,\rho_{B}(\alpha)).\]
In the notation of Proposition \ref{prop:operators:splitting}, we have that 
\[
U(\alpha,X)=\Omega(\rho_{B}(\alpha),X),\qquad \nabla^{L}=\nabla.
\]

Here is a class of bundle of ideals which are of principal type:

\begin{proposition}\label{prop:isotropy:rigid:implies:linearization}
Let $\ka\subset A$ be a bundle of ideals such that:
\begin{enumerate}[(i)]
\item $\ka$ is a locally trivial Lie algebra bundle;
\item the typical fiber $(\gg,[\cdot,\cdot])$ of $\ka$ satisfies:
\[H^0(\gg,\gg)=H^1(\gg,\gg)=0.\]
\end{enumerate}
Then $\ka\subset A$ is of principal type, and in particular it is partially split. 
\end{proposition}

\begin{corollary}
If $\ka\subset A$ is a bundle of semi-simple Lie algebras, then it is of principal type.  
\end{corollary}

\begin{proof}
For a semi-simple Lie algebra $\gg$, we have that $H^i(\gg,\gg)=0$, for all $i$. Also, $\ka$ is automatically a locally trivial Lie algebra (see Remark \ref{rem:h2:algebra:bundle}), hence Proposition \ref{prop:isotropy:rigid:implies:linearization} applies.
\end{proof}

\begin{proof}[Proof of Proposition \ref{prop:isotropy:rigid:implies:linearization}]
Consider the Lie algebroid $\gl(\ka)$ whose sections are the derivations of the vector bundle $\ka$, i.e., the bundle maps $D_X:\Gamma(\ka)\to\Gamma(\ka)$ satisfying:
\[ D_X(f\xi)=f D_X(\xi)+X(f)\xi, \]
for all $f\in C^\infty(M)$. The anchor of $\gl(\ka)$ assigns to the section $D_X$ its symbol $X\in\X(M)$, and the Lie bracket is the commutator of derivations. We let $A'\subset \gl(\ka)$ be the Lie subalgebroid whose sections are the derivations of the bracket $[\cdot,\cdot]_{\ka}$
\[ \Gamma(A')=\mathrm{Der}(\ka,[\cdot,\cdot]_{\ka}). \] 
The fact that $A'$ is a Lie subalgebroid and transitive, follows from (i). 

Now, to prove the proposition, let $B=A/\ka$ and consider the map
\[ \psi:A\to A'\times B, \quad \al\mapsto (\nabla^{\ka}_{\alpha}, [\al]). \]
We claim that this is an injective Lie algebroid morphism, with image the fiber product $A'\times_{TM} B$. That $\psi$ is a Lie algebroid morphism follows the fact that  both components $\nabla^\ka:A\to A'$ and $\pr:A\to B$ are Lie algebroid maps. On the other hand, if $\al\in A$ is such that $\psi(\al)=0$, then $\al\in\ka_x$ satisfies: 
\[ \nabla^\ka_\al=[\al,-]_{\ka_x}=0. \] 
But (ii) is equivalent to $\ad:\ka_x\to \Der(\ka_x)$ being an isomorphism of Lie algebras. Injectivity of this maps shows that $\al=0$ and so $\psi$ is injective. Finally, surjectivity of this map implies that the map $\psi$ is onto $A'\times_{TM} B$.
\end{proof}


\begin{remark}
As observed in the proof, condition (ii) of Proposition \ref{prop:isotropy:rigid:implies:linearization} is equivalent to $\ad:\gg\to \Der(\gg)$ being a Lie algebra isomorphism. Therefore, $G:=\Aut(\gg)$ is a Lie group integrating $\gg$ and we have the principal $G$-bundle of $\gg$-frames:
\[P:=\{\varphi: \gg \to \ka_x\, :\, x\in M,\, \varphi \, \textrm{is a Lie algebra isomorphism}\}.\]
The transitive Lie algebroid $A'\to M$ in the proof is isomorphic to the Atiyah Lie algebroid of $P$ and hence is integrable.  
\end{remark}

We also have a global version of Proposition \ref{prop:isotropy:rigid:implies:linearization}.

\begin{proposition}
Let $\Phi:\G\to \H$ be a groupoid map covering the identity which is a surjective submersion, and let $\phi:A\to B$ be the induced Lie algebroid map. If $\ka:=\ker \phi$ satisfies the conditions of Proposition \ref{prop:isotropy:rigid:implies:linearization} then $\Phi$ is partially split. 
\end{proposition}

\begin{proof}
The assumption implies that 
\[\G':=\big\{\textrm{Lie algebra isomorphisms}\ \  \ka_x\diffto \ka_y, \ x,y\in M \big\}\tto M\]
is a transitive Lie groupoid. The map \eqref{eq:representation:grpd} $\Ad:\G\to \G'$ together with the map $\Phi$ induce a groupoid map into the fiber product: 
\[\Psi:=(\Ad,\Phi):\G\to \G'\times_{M\times M} \H.\]
Then $\Psi$ integrates the Lie algebroid map $\psi$ from the proof of Proposition \ref{prop:isotropy:rigid:implies:linearization}. Since $\psi$ is an isomorphism, it follows that $\Psi$ is a local diffeomorphism. On the other hand, since $\G'$ is transitive, it follows that $\pr_{2}:\G'\times_{M\times M} \H\to \H$ is of principal type, so by the example of Subsection \ref{ex:princ:type:part:split}, $\pr_{2}$ admits a multiplicative Ehresmann connection $E$. The preimage of $E$ under $\Psi$ is a multiplicative Ehresmann connection for $\Phi$.
\end{proof}

\subsection{Bundles of ideals of principal type with kernel flat IM connections}
\label{example:principal:type:flat}
Let $\ka\subset A$ be an ideal of principal type, for which the 2-form \eqref{eq:bracket:transitive} is center-valued:
\[\Omega\in  \Omega^2(M,z(\ka)).\]
Then the corresponding connection $\nabla$ is flat, so we are in the case of an IM Ehresmann connection which is kernel flat. This kind of partially split ideals can be explicitly described as follows. 

Let $B\Ato M$ be a Lie algebroid and let $\ka\to M$ be a Lie algebra bundle with a flat connection $\nabla^L$ preserving the bracket $[\cdot,\cdot]_{\ka}$. Given any center-valued, $\d^{\nabla^L}$-closed, 2-form $\Omega\in \Omega^2(M,z(\ka))$ one can define a center-valued tensor $U$ by:
\[ U(\al,X):=\Omega(\rho_B(\al),X). \]
One checks that $U$ satisfies (S3), so that (S1')-(S3') hold. Therefore, we obtain a partially split bundle of ideals $\ka\subset A$ with flat coupling data. Moreover, it is of principal type: we have $A\simeq A'\times_{TM} B$, where $A'=TM\oplus \ka$ is the transitive Lie algebroid
with Lie bracket:
\[[(X,\xi),(Y,\eta)]_B=([X,Y],\Omega(X,Y)+\nabla^L_X\eta-\nabla^L_Y\xi+[\xi,\eta]_{\ka}).\]

Note that in this case the $B$-cocycle $\lambda$ in Remark \ref{rem:cocycle:not:invariant} is given by $(\rho_B)^*(\Omega)$, so its class is the image of the map:
\[ (\rho_B)^*:H^2(M,z(\ka))\to H^2(B,z(\ka)). \]
When $\Omega=\d^{\nabla_L}\theta$, with $\theta\in \Omega^1(M,z(\ka))$ we have $[\lambda]=0$ and the resulting bundle of ideals $\ka\subset A$ is isomorphic to a ``trivial'' coupling data $(\nabla^L,0)$, via the map:
\[ A\to B\oplus\ka,\quad (\al,\xi)\mapsto (\al,\theta(\al)+\xi). \]

On the other hand, given any partially split bundle of ideals $\ka\subset A$ with flat coupling data $(\nabla^L,U)$ and a center-valued, $\d^{\nabla^L}$-closed, 2-form $\Omega\in \Omega^2(M,z(\ka))$, we obtain a ``gauge transformed" bundle of ideals with coupling data $(\nabla^L,e^{\Omega}U)$ where:
\[  e^{\Omega}U(\al,X):=U(\al,X)+\Omega(\rho_B(\al),X). \]
The bracket of the underlying algebroid $B$ changes by replacing the cocycle $\lambda$ with the cocycle $\lambda+(\rho_B)^*(\Omega)$.

\subsection{Bundles of ideals of rank one}
We treat in detail the case of bundle of ideals of rank one. In this setting the existence of specific types of IM connections have cohomological interpretations.

Let $\ka\subset A$ be a bundle of ideals of rank one, which, for simplicity, we assume to be orientable. As usual, we set $B=A/\ka$, and after choosing a splitting of the short exact sequence:
\[ \xymatrix{
0\ar[r] & \ka \ar[r] & A \ar[r] & B\ar[r] & 0}
\]
we obtain an identification
\[ A\simeq B\oplus \ka, \]
where the anchor and the Lie bracket on $A$ become: 
\begin{align*}
    &\rho_A(\al,\xi)=\rho_B(\al),\notag\\
    &[(\al,\xi),(\be,\eta)]_A=([\al,\be]_{B},\lambda(\al,\be)+\nabla^\ka_\al \eta-\nabla^\ka_\be \xi).
\end{align*} 
Notice that:
\begin{itemize}
   \item[\tiny$\bullet$] The flat $B$-connection $\nabla^{\ka}$ is independent of the choice of splitting of $A$;
    \item[\tiny$\bullet$] Up to isomorphism, the representation of $B$ on $\ka$ is determined by its characteristic class (see, e.g., Section 11.1 in \cite{CFM21}):
    \[ c_1(\ka)\in H^1(B).\]
    Explicitly, if we choose a trivialization $\ka\simeq M\times \R$, then
    \[\nabla^{\ka}_{\alpha}f=\Lie_{\rho_B(\alpha)}f+i_{V}(\alpha)f,\quad (f\in C^{\infty}(M))\]
    for a cocycle $V\in \Omega^1(B)$, which is a representative of $c_1(\ka)$. If we change the trivialization $\ka\simeq M\times \R$ by multiplying with a non-zero function $h$, then $V$ changes by $\rho_B^*(\d\log|h|)=\d_B\log|h|$.
 \item[\tiny$\bullet$] The $\ka$-valued 2-form $\lambda\in \Omega^2(B;\ka)$ is $\d_B$-closed, so it determines a class:
    \[ c_2(A):=[\lambda]\in H^2(B;\ka).\]
    This class is well-defined: a different splitting of $A$ is determined by an element $Z\in \Omega^1(B;\ka)$, and this has the effect of changing $\lambda$ by $\d_{B}Z$.
\end{itemize} 

The anchor $\rho_B:B\to TM$, being a Lie algebroid morphism, defines a map in cohomology
\[\rho_B^*:H^1(M)\rmap H^1(B). \]

The class $c_1(\ka)$ belongs to the image of this map if and only if there exists a flat connection $\nabla$ on $\ka$ such that, for all $\al\in \Gamma(B)$ and $\xi\in \Gamma(\ka)$:
\begin{equation}\label{eq:lift:of:connection}
\nabla^{\ka}_{\alpha}\xi=\nabla_{\rho_B(\al)}\xi.
\end{equation}
To see this, choose a trivialization $\ka\simeq M\times \R$, giving rise to the representative $V$ of $c_1(\ka)$. We have a one-to-one correspondence between flat connections $\nabla$ on $M\times \R$ and closed forms $\theta\in\Omega^1(M)$, given by:
\[\nabla_Xf=\Lie_Xf+\theta(X)f,\quad (f\in C^{\infty}(M))\]
Under this correspondence, we have that \eqref{eq:lift:of:connection} is equivalent to $V=\rho_B^*\theta$. The class
\[ 
c_1(\nabla):=[\theta]\in H^1(M).
\]
is just the characteristic class of the representation of $TM$ on $\ka$ given by $\nabla$, and we have by definition:
\[ \rho_B^*(c_1(\nabla))=c_1(\ka).\]
Below, whenever $c_1(\ka)$ belongs to the image of $\rho^*_B$, and we choose a flat connection satisfying \ref{eq:lift:of:connection}, we will say that ``$\nabla$ is a flat connection with $\rho_B^*\nabla=\nabla^{\ka}$". In particular, when this happens, $\nabla$ is a  $B$-invariant connection (see Definition \ref{def:A-inv:conn}). Also, we have a map in cohomology with coefficients in $\ka$, for the representations defined by $\nabla$ and $\nabla^\ka$:
\[
\rho_B^*:H^\bullet(M;\ka)\rmap H^\bullet(B;\ka),
\]
and another map in IM cohomology, for the differential $\d^\nabla_\imult$ (see Lemma \ref{lemma:chain:map:IM:Lie:alg}):
\[H^\bullet_{\imult}(B;\ka)\rmap H^\bullet(B;\ka). \]

After these preliminaries, we can now state a result about the various special IM Ehresmann connections one can have for a bundle of ideals of rank one:

\begin{proposition}
\label{characterizations:rank:one}
Let $\ka\subset A$ be a bundle of ideals of rank one, which is orientable. 
\begin{enumerate}[(i)]
    \item The Lie algebroid $A$ is isomorphic to a product, $A\simeq B\times \R$, with $\ka\simeq 0_M\times \R$ (see Subsection \ref{ex:product}), if and only if \[c_1(\ka)=0\quad \textrm{and}\quad c_2(A)=0.\]
    \item There exists a totally flat IM Ehresmann connection for $\ka\subset A$ if and only if
    \[c_1(\ka)\in\im(H^1(M)\to H^1(B))\quad \textrm{and}\quad c_2(A)=0.\]
    \item There exists a leafwise flat IM Ehresmann connection for $\ka\subset A$ if and only if $\ka$ admits a $B$-invariant connection $\nabla$ and $c_2(A)=0$.
    \item There exists a kernel flat IM Ehresmann connection for $\ka\subset A$ if and only if
    \[c_1(\ka)\in\im(H^1(M)\to H^1(B))\quad \textrm{and}\quad c_2(A)\in\im(H^2_{\imult}(B;\ka)\to H^2(B;\ka)),\]
    where the last map is for a flat connection $\nabla$ with $\rho_B^*\nabla=\nabla^{\ka}$.
    \item The Lie bundle of ideals $\ka\subset A$ is of principal type if and only if 
    \[c_1(\ka)\in\im(H^1(M)\to H^1(B))\quad \textrm{and}\quad c_2(A)\in\im(H^2(M;\ka)\to H^2(B;\ka)),\]
    where the last map is for a flat connection $\nabla$ with $\rho_B^*\nabla=\nabla^{\ka}$.
 \end{enumerate}
\end{proposition}

Before we look at the proof, let us discuss the general question of existence of partial splittings for ideals of rank one: 

\begin{proposition}
\label{prop:partial:splittings:codim:1}
A rank one bundle of ideals $\ka\subset A$ admits an IM Ehresmann connection if and only if: 
\begin{enumerate}[(i)]
    \item $\ka$ admits a $B$-invariant connection $\nabla^L$, i.e., for all $\al\in B$:
    \[\nabla^{\ka}_{\al}=\nabla^{L}_{\rho_B(\al)}\quad \textrm{and}\quad  i_{\rho_B(\al)}R^{\nabla^L}=0,\]
    \item and $c_2(A)$ has a representative of the form 
    \[\lambda(\al,\be)=i_{\rho_B(\be)}U(\al),\]
    where $U:B\to T^*M\otimes \ka$ is a linear map satisfying the structure equation: \[ U([\al,\be]_{B})=\Lie^{\nabla^L}_{\rho_B(\al)}U(\be)-i_{\rho_B(\be)}\d^{\nabla^L}U(\al),\]
    for all $\al,\beta\in \Gamma(B)$.
\end{enumerate} 
\end{proposition}

\begin{proof}
From the theory of couplings, discussed in Section \ref{sec:obstructions:partial:split}, $\ka\subset A$ is partially split if and only if
there exists a pair $(\nabla^L,U)$ and a splitting $A\simeq B\oplus \ka$ such that $\nabla^{\ka}_{\alpha}=\nabla^L_{\rho_B(\alpha)}$, $\lambda(\al,\be)=U(\al,\rho_B(\be))$, and the structure equations (S2) and (S3) hold. It follows easily that (S2) is equivalent to invariance of the connection, and that (S3) is equivalent to the equation in (ii) -- as was mentioned already in the proof of Proposition \ref{prop:flat:jets}.
\end{proof}

Assume then that we have an IM Ehresmann connection with coupling data $(\nabla^L,U)$, as in the previous proposition. After fixing a trivialization $\ka\simeq M\times \R$, we can write the connection as
\[\nabla^L_X=\Lie_X+\theta(X),\]
for a unique 1-form $\theta\in \Omega^1(M)$. Then $c_1(\ka)=[V]$ where $V\in \Omega^1(B)$ satisfies:
\[ V(\al)=i_{\rho_B(\al)}\theta.\]
Moreover, regarding $U$ as a map $U:B\to T^*M$, $c_2(A)=[\lambda]$ where $\lambda\in \Omega^2(B)$ satisfies:
\[ \lambda(\al,\be)=i_{\rho_B(\be)}U(\al).\]
The structure equations from (i) and (ii) become:
\begin{align}
    \label{eq:S2:abelian} \tag{S2''}
    & i_{\rho_B(\alpha)}\d\theta=0,\\
    \label{eq:S3:abelian} \tag{S3''}
    & U([\al,\be]_{B})=\Lie_{\rho_B(\al)}U(\be)-i_{\rho_B(\be)}\d U(\al)\\
    &\qquad \qquad \quad +\theta(\rho_B(\al))U(\beta)-i_{\rho_{B}(\beta)}\big(\theta\wedge U(\al)\big),\notag 
\end{align}
for all $\al,\be\in\Gamma(B)$. This follows from the relations: \[ R^{\nabla^L}(X,Y)=\d\theta(X,Y),\quad \d^{\nabla^L}\omega=\d\omega+ \theta\wedge\omega, \quad \Lie^{\nabla^L}_X\omega=\Lie_X\omega +\theta(X)\omega.\]

Recall that a class $c\in H^k(B;\Rep)$ is said to be \textbf{tangential} if it has some representative $P\in \Omega^k(B;\Rep)$ satisfying
\[ \rho_B(\al)=0\quad \Longrightarrow \quad i_{\al}P=0.\]
The previous discussion implies the following obstruction for existence of IM Ehresmann connection:

\begin{corollary}
If a rank one bundle of ideals $\ka\subset A$ admits an IM Ehresmann connection, then the classes $c_1(\ka)\in H^1(B)$ and $c_2(A)\in H^2(B;\ka)$ are tangential.
\end{corollary}

\begin{proof}[Proof of Proposition \ref{characterizations:rank:one}]
(i) If $A$ is isomorphic to a product, then clearly $c_1(\ka)=0$ and $c_2(A)=0$. Conversely, $c_2(A)=0$ means that there exists a splitting $A\simeq B\times \ka$ such that $\lambda=0$, and $c_1(\ka)=0$ means that there exists a trivialization $\ka\simeq M\times \R$ such that $V=0$. Thus $A$ is a product.


(ii) Assume that $\ka\subset A$ admits a totally flat IM connection, with coupling data $(\nabla^L,U=0)$. Then, as remarked above $c_1(\ka)$ belongs to the image of $\rho_B^*$, and $\lambda=0$, so clearly $c_2(A)=0$. Conversely, under the assumptions, there exists a splitting $A\simeq B\oplus \ka$, under which $\lambda=0$ and there is a flat connection $\nabla^L$ on $\ka$ inducing $\nabla^{\ka}$. Then $(\nabla^L,U)$ is the coupling data of a totally flat IM connection for $\ka$.

(iii) Assume that $\ka\subset A$ admits a leafwise flat IM connection, with coupling data $(\nabla^L,U)$. This means that $\lambda=0$, so clearly, $c_2(A)=0$, and as remarked before $\nabla^L$ is $B$-invariant. Conversely, under the assumptions, there exists a splitting $A\simeq B\times \R$, under which $\lambda=0$, and there is a connection $\nabla^L$ which is $B$-invariant. Then the pair $(\nabla^L,U=0)$ is the coupling data of a leafwise flat IM connection for $\ka$.

(iv) Assume that $\ka\subset A$ admits a kernel flat IM connection, with coupling data $(\nabla^L,U)$. This means that $\nabla^L$ is flat, so $c_1(\ka)$ is in the image of $\rho_B^*$. By Proposition \ref{prop:flat:jets}, the pair $(\d^{\nabla^L}U,U)$ is an IM 2-form on $B$ with values in $\ka$, which is clearly closed for the differential $\d_{\imult}^{\nabla^L}$ from Proposition \ref{prop:differ:IM:forms}, and which is sent under the map from Lemma \ref{lemma:chain:map:IM:Lie:alg} to $\lambda$. So $c_2(A)$ it the image of $[\gamma]$. Conversely, the assumptions imply the existence of a flat connection $\nabla^L$ on $\ka$ which is $B$-invariant, and of a splitting $A\simeq B\oplus \ka $, under which $\lambda$ is in the image of $\Omega_{\imult}^2(B;\ka)\to \Omega^2(B;\ka)$ (see Lemma \ref{lemma:chain:map:IM:Lie:alg}) of an element $\gamma$ which is $\d^{\nabla^L}_{\imult}$-closed. This implies that $\gamma=(\d^{\nabla^L}_{\imult}U,U)$. The second part of Proposition \ref{prop:flat:jets} implies that $(\nabla^L,U)$ is the coupling data of a kernel flat IM connection.

(v) Assume that $\ka \subset A$ is of principal type, i.e., there exists a transitive Lie algebroid $\ka\hookrightarrow A' \twoheadrightarrow TM$ so that $A\simeq A'\times_{TM}B$. As discussed in Subsection \ref{ex:principal:type:Lie:algebroids}, after choosing a splitting $A'\simeq TM\oplus \ka$, the Lie bracket on $A'$ is given by
\begin{equation}\label{eq:bracket:on:transitive:Lie}
[(X,\xi),(Y,\eta)]_{{A'}}=([X,Y],\Omega(X,Y)+\nabla_{X}\eta-\nabla_{Y}\xi),
\end{equation}
where $\nabla$ is a flat connection on $\ka$ and $\Omega\in \Omega^2(M;\ka)$ is $\d^{\nabla}$-closed. Using the induced splitting $A\simeq B\oplus \ka$, we obtain that $\nabla^{\ka}_{\al}=\nabla_{\rho_B(\al)}$ and $\lambda=\rho_B^*(\Omega)$; so $c_1(\ka)$ and $c_2(A)$ are in the image of $\rho^*_B$. Conversely, if this happens, we have a flat connection $\nabla$ on $\ka$ such that $\nabla^{\ka}_{\al}=\nabla_{\rho_B(\al)}$, and a $\d^{\nabla}$-closed 2-form  $\Omega\in 
\Omega^2(M;\ka)$ such that $\lambda=\rho_B^*(\Omega)$, for some splitting. This means precisely that $A\simeq A'\times_{TM}B$, where $A'=TM\oplus \ka$ is the transitive Lie algebroid with bracket \eqref{eq:bracket:on:transitive:Lie}.
\end{proof}

\appendix
\section{Multiplicative and infinitesimal multiplicative forms}

We recall here a few results about multiplicative and infinitesimal multiplicative (IM) forms that are required throughout the paper. For us, it will be specially important forms with coefficients, since they appear as ``connection 1-forms'' for Ehresmann connections.  More details on the results mentioned in this section can be found in \cite{BuCa12,BCdH16, CaDr17, CrSaSt15,DrEg19} and we only provide proofs of the results that cannot be found in those references.

\subsection{Multiplicative forms and IM forms} 
\label{app:IM:forms} 
Let $\G\tto M$ be a Lie groupoid with source/target $\s,\t:\G\to M$ and multiplication $m:\G\timesst \G\to \G$. 

\begin{definition}
\label{def:mult:forms}
A differential form $\omega\in\Omega^k(\G)$ is called  {\bf multiplicative} if it lies in the kernel of the simplicial differential:
\[ \delta\omega:=\pr_1^*\omega-m^*\omega +\pr_2^*\omega=0 \]
where $\pr_i:\G\timesst \G\to \G$ are the projections on the factors.
\end{definition}
 
The differential of a multiplicative form is again a multiplicative form, so we have a complex of multiplicative differential forms $(\Omega^{\bullet}_\mult(\G),\d)$. A basic fact is that two multiplicative forms $\omega,\omega'\in \Omega^k_\mult(\G)$ such that $\omega|_M=\omega'|_M$ and $(\d\omega)|_M=(\d\omega')|_M$ actually coincide, assuming that $\G$ has connected target fibers.

Denote by $(A,[\cdot,\cdot],\rho)$ the Lie algebroid of $\G\tto M$. Given a multiplicative form $\omega\in\Omega^k_\mult(\G)$ one defines two vector bundle maps  $\mu:A\to \wedge^{k-1}T^*M$ and $\zeta:A\to \wedge^kT^*M$ by setting:
\begin{equation}
\label{eq:M:IM:forms}
 \mu(a)=i_a\omega|_{TM},\quad \zeta(a)=i_a\d\omega|_{TM}.
\end{equation}
The pair $(\mu,\zeta)$ satisfies the following set of equations for any $\al,\be\in\Gamma(A)$:
\begin{align}
\label{eq:mult:form}
i_{\rho(\be)} \mu(\al) &=-i_{\rho(\al)} \mu(\be),\notag\\
\mu([\al,\be])&=\Lie_{\rho(\al)}\mu(\be)-i_{\rho(\be)}\d \mu(\al)-i_{\rho(\be)}\zeta(\al),\\
\zeta([\al,\be])&=\Lie_{\rho(\al)}\zeta(\be)-i_{\rho(\be)}\d \zeta(\al).\notag
\end{align}
This leads to the notion of {\bf infinitesimal multiplicative form} on an arbitrary Lie algebroid $A\Ato M$, integrable or not:

\begin{definition}
An {\bf IM $k$-form} on a Lie algebroid $(A,[\cdot,\cdot],\rho)$ is a pair $(\mu,\zeta)$, where $\mu:A\to \wedge^{k-1}T^*M$ and $\zeta:A\to \wedge^kT^*M$ are bundle maps satisfying \eqref{eq:mult:form}. 
\end{definition}

The space of IM forms is denoted by $\Omega_\imult^k(A)$ and it becomes a cochain complex with differential given by
\[ \d_\imult:\Omega_\imult^k(A)\to \Omega_\imult^{k+1}(A), \quad \d_\imult(\mu,\zeta):=(\zeta,0). \]
For a \emph{target 1-connected} Lie groupoid $\G\tto M$ with Lie algebroid $A\Ato M$ the assignment $\omega\mapsto (\mu,\zeta)$ given by \eqref{eq:M:IM:forms} is an isomorphism of complexes:
\[ (\Omega^\bullet_\mult(\G),\d)\simeq (\Omega^\bullet_\imult(A),\d_\imult). \]

\subsection{Multiplicative forms and IM forms with coefficients} 
This subsection follows \cite{CrSaSt15}, where all claims are proven. Let $\G$ be a Lie groupoid and $\Rep$ a $\G$-representation. We will work with differential forms on $\G$ with coefficients in $\s^*\Rep$ (instead of $\t^*\Rep$, as in \cite{CrSaSt15}), which we denote by
\[\Omega^{\bullet}(\G;\Rep):=\Omega^{\bullet}(\G;\s^*\Rep).\]
Similarly, denote by $\Omega^{\bullet}(\G^{(k)};\Rep)$ the space of differential forms on the manifold $\G^{(k)}$ of composable $k$-strings of arrows, with values in the vector bundle $(\s\circ \pr_k)^*(\Rep)\to \G^{(k)}$, where $\pr_i:\G^{(k)}\to \G$ is the projection onto the $i$-th component.

\begin{definition}
\label{def:mult:forms:coeff}
A form $\omega\in\Omega^{\bullet}(\G;\Rep)$ is called {\bf multiplicative} if
\[ \delta\omega=0,\]
where $\delta:\Omega^{\bullet}(\G;\Rep)\to \Omega^{\bullet}(\G^{(2)};\Rep)$
is the simplicial differential defined by:
\begin{equation}
\label{eq:simplicial:E:forms:1} 
\delta\omega|_{(g_1,g_2)}=g_2^{-1}\cdot \pr_1^*\omega|_{g_1}-m^*\omega|_{g_1g_2}+\pr_2^*\omega|_{g_2}.
\end{equation}
We denote by $\Omega^\bullet_\mult(\G;\Rep)$ the space of $\Rep$-valued multiplicative $k$-forms. 
\end{definition}


A simple way to produce multiplicative forms is by using the simplicial differential in degree zero $\delta:\Omega^{\bullet}(M; \Rep)\to \Omega^{\bullet}_\mult(\G;\Rep)$ given by:
\begin{equation}
\label{eq:simplicial:E:forms:0} 
\delta\omega|_g:=g^{-1}\cdot \t^*\omega - \s^*\omega,\quad g\in \G.
\end{equation}
We will see later other examples. 

Let us pass to the infinitesimal level, so denote by $A\Ato M$ the Lie algebroid of $\G\tto M$. Given a representation $p:\Rep\to M$ of $\G$, at the algebroid level we obtain a \textbf{flat $A$-connection} on $\Rep$:
\[ \nabla:\Gamma(A)\times \Gamma(\Rep)\to \Gamma(\Rep), \quad (\al,s)\mapsto \nabla_\al s.  \]
This means that for all $\al,\be\in\Gamma(A)$, $s\in\Gamma(\Rep)$, and $f\in C^\infty(M)$, we have:
%
\[\nabla_{f\al}s=f\nabla_\al s, \ \ \ 
\nabla_\al(fs)=f\nabla_\al s+(\Lie_{\rho(\al)} f) s,\ \ \ 
\nabla_{[\al,\be]}=[\nabla_{\al},\nabla_{\be}]. 
\]
We also call $(\Rep,\nabla)$ a \textbf{representation} of the Lie algebroid $A$.

For a $\Rep$-valued multiplicative $k$-form $\omega\in\Omega^k_\mult(\G;\Rep)$, define a vector bundle map $l:A\to \wedge^{k-1}T^*M\otimes \Rep$ by setting
\begin{equation}
\label{eq:M:IM:E:forms}
l(a):=(i_{a}\omega)|_{TM},
\end{equation}
and a linear operator $L:\Gamma(A)\to \Omega^k(M; \Rep)$ by letting
\begin{equation}
\label{eq:M:IM:E:forms:2}
L(\al)_x(v_1,\dots,v_k):=\frac{\d}{\d \epsilon}\Big|_{\epsilon=0} \phi_{\al^{L}}^{\epsilon}(x)\cdot \omega(\d_x \phi_{\al^{L}}^{\epsilon}(v_1),\dots,\d_x \phi_{\al^{L}}^{\epsilon}(v_k))
\end{equation}
where $\phi_{\al^{L}}^{\epsilon}$ denotes the flow of the left-invariant vector field $\al^L$.
This operator is a kind of differential operator with symbol $l$, in the sense that for any $f\in C^\infty(M)$ and $\al\in\Gamma(A)$ it satisfies:
\begin{equation}
\label{eq:symbol:IM:form}
L(f\al)=fL(\al)+\d f\wedge l(\al).
\end{equation}
Furthermore, the pair $(L,l)$ satisfies the following set of equations:
\begin{align}
\label{eq:compatibility:IM:E:form}
i_{\rho(\al)}l(\be)&=-i_{\rho(\be)} l(\al),\notag \\
L([\al,\be])&=\Lie_{\al} L(\be)-\Lie_{\be} L(\al),\\
l([\al,\be])&=\Lie_{\al} l(\be)-i_{\rho(\be)}L(\al), \notag
\end{align}
where, for $\al\in\Gamma(A)$ and $\gamma\in\Omega^k(M; \Rep)$, we denoted: 
\begin{equation}
    \label{eq:Lie:derivative:rep}
    \Lie_\al\gamma(X_1,\dots,X_k):=\nabla_\al(\gamma(X_1,\dots,X_k))-\sum_{i=1}^k\gamma(X_1,\dots,[\rho(\al),X_i],\dots,X_k).
\end{equation} 

Now observe that all this makes sense for an arbitrary algebroid $A\Ato M$, integrable or not, provided one has a representation $(\Rep,\nabla)$:

\begin{definition}
Let  $(\Rep,\nabla)$ be a representation of the Lie algebroid $A\Ato M$. An {\bf $\Rep$-valued IM $k$-form} is a pair $(L,l)$, where $L:\Gamma(A)\to \Omega^k(M,\Rep)$ is a linear map and $l:A\to \wedge^{k-1}T^*M\otimes \Rep$ is a vector bundle map satisfying \eqref{eq:symbol:IM:form} and \eqref{eq:compatibility:IM:E:form}. 
\end{definition}


We denote by $\Omega^\bullet_\imult(A;\Rep)$ the space of $\Rep$-valued IM $k$-forms. For historical reasons, these are also sometime called \emph{Spencer operators}. For a \emph{target 1-connected} Lie groupoid $\G\tto M$ with Lie algebroid $A\Ato M$ the assignment $\omega\mapsto (L,l)$ given by \eqref{eq:M:IM:E:forms} and \eqref{eq:M:IM:E:forms:2} gives an isomorphism:
\[ \Omega^\bullet_\mult(\G;\Rep)\simeq \Omega^\bullet_\imult(A;\Rep). \]

\begin{example}[$\Rep$-valued forms of degree 1]
\label{ex:degree1:E-valued:M:forms}
\label{ex:degree1:E-valued:IM:forms}
For us the $\Rep$-valued multiplicative forms of degree 1 will be important. Such a form $\omega\in\Omega^1_\mult(\G;\Rep)$ is the same thing as a morphism of VB groupoids $\omega:T\G\to \s^*\Rep=\G\timessp \Rep$ covering the identity morphism $\id:\G\to \G$. Let us explain this in more detail. 

First, the tangent bundle of a Lie groupoid $\G\tto M$ is a VB-groupoid:
\[
\xymatrix{
T\G \ar[r] \ar@<0.15pc>[d] \ar@<-0.15pc>[d] & \G \ar@<0.15pc>[d] \ar@<-0.15pc>[d]\\
TM\ar[r] & M
}
\]
Next, if $p:\Rep\to M$ is a representation of $\G$, then we have the groupoid extension $\G\times_M \Rep:=\G\timessp \Rep\tto M$, where composition of arrows is given by:
\[ (g,v)\cdot (h,w)=(gh,h^{-1}\cdot v+w). \]
It can also be viewed as a VB-groupoid:
\[
\xymatrix{
\G\times_M \Rep \ar[r] \ar@<0.15pc>[d] \ar@<-0.15pc>[d] & \G \ar@<0.15pc>[d] \ar@<-0.15pc>[d]\\
0_M\ar[r] & M
}
\]
Now an $\Rep$-valued multiplicative differential form $\omega\in\Omega^\bullet_\mult(\G;\Rep)$ is just a map:
\[ \omega:T\G\to \s^* \Rep=\G\times_M \Rep, \quad v_g\mapsto (g,\omega(v_g)), \]
which is a morphism of VB groupoids covering the identity. In other words, it is a groupoid morphism, which is a vector bundle map making the diagram commute:
\[
\xymatrix@C=15pt@R=15pt{
T\G \ar@<0.15pc>[dd] \ar@<-0.15pc>[dd] \ar[rr]^{\omega} \ar[dr]&  &  \G\times_M \Rep  \ar@<0.15pc>@{-}[d] \ar@<-0.15pc>@{-}[d] \ar[dr]\\
 & 
\G\ar[rr]^(.3){\id}  \ar@<0.15pc>[dd] \ar@<-0.15pc>[dd] & \ar@<0.15pc>[d] \ar@<-0.15pc>[d] &\G  \ar@<0.15pc>[dd] \ar@<-0.15pc>[dd] \\
TM \ar[dr] \ar@{-}[r] &\ar[r]  & 0_M \ar[dr]\\
 & M\ar[rr] &  &M
}
\]

A similar discussion holds at the infinitesimal level for $\Rep$-valued IM forms of degree 1: a 1-form $(L,l)\in\Omega^1_\imult(A; \Rep)$ is the same thing as a VB algebroid morphism $\theta^\vee:TA\to A\oplus \Rep$ covering the identity morphism $\id:A\to A$, as we now explain. 

First, the tangent bundle of a Lie algebroid $A\Ato  M$ is a VB-algebroid:
\[
\xymatrix{
TA \ar[r] \ar@{=>}[d] & A \ar@{=>}[d]\\
TM\ar[r] & M
}
\]
where the double arrow is used to distinguish the algebroid projections from the algebroid morphisms. Next, given a representation $(\Rep,\nabla)$ of $A$, then we have the algebroid extension $A\oplus \Rep$ with anchor $\rho\circ\pr_A$ and bracket:
\[ [(\al,s),(\be,t)]=([\al,\be],\nabla_\al  t-\nabla_\be s). \]
It can also be viewed as a VB-algebroid:
\[
\xymatrix{
A\oplus \Rep \ar[r] \ar@{=>}[d] & A \ar@{=>}[d]\\
0_M\ar[r] & M
}
\]

Now, given a morphism of VB algebroids $\theta^\vee:TA\to A\oplus \Rep$ covering the identity morphism $\id:A\to A$:
\[
\xymatrix@C=15pt@R=15pt{
TA \ar@{=>}[dd] \ar[rr]^{\theta^\vee} \ar[dr]&  &  A\oplus \Rep  \ar@{=}[d] \ar[dr]\\
 & 
A\ar[rr]^(.3){\id}  \ar@{=>}[dd] & \ar@{=>}[d] &A  \ar@{=>}[dd] \\
TM \ar[dr] \ar@{-}[r] &\ar[r]  & 0_M \ar[dr]\\
 & M\ar[rr] &  &M
}
\]
we can define an $\Rep$-valued IM form $(L,l)\in\Omega^1_\imult(A,\Rep)$ by setting for any section $\al\in\Gamma(A)$ and $v\in TM$:
\begin{align*}
L(\al)(v)&:=\pr_\Rep(\theta^\vee(\d\al(v))),\\
l(\al)&:=\pr_\Rep\theta^\vee(\hat{\al})|_M,
\end{align*}
where $\hat{\al}\in\X(A)$ is the vertical lift of the section $\al$. Conversely, given an $\Rep$-valued IM form $(L,l)\in\Omega^1_\imult(A;\Rep)$ there exists a unique VB algebroid morphism $\theta^\vee:TA\to A\oplus \Rep$ which covers the identity and corresponds to $(L,l)$ under the above assignment. 
\end{example}

\subsection{Differentiation of multiplicative forms with coefficients}
\label{sec:differentiation:forms:coefficients}
In this section we discuss differentiation of multiplicative forms with coefficients. The material here seems to be new, so we include detailed proofs.

On a manifold $M$, in order to be able to differentiate forms with coefficients in a vector bundle $p:\Rep\to M$ one needs to choose a connection $\nabla$. This gives a $\nabla$-de Rham operator $\d^\nabla:\Omega^k(M;\Rep)\to \Omega^{k+1}(M;\Rep)$ defined by
\begin{align*}
    \d^\nabla\omega(X_0,\dots,X_k)&:=\sum_i (-1)^i\nabla_{X_i}\omega(x_0,\dots,\widehat{X_i},\dots,X_k)+\\
    &\qquad+\sum_{i<j}(-1)^{i+j}\omega([X_i,X_j],X_0,\dots,\widehat{X_i},\dots,\widehat{X_j},\dots,X_k).
\end{align*}
Note that $\d^\nabla$ is a differential if and only if $\nabla$ is flat.

Now, consider a representation $p:\Rep\to M$ of a Lie groupoid $\G\tto M$, with Lie algebroid $A\Ato M$. If we fix a connection $\nabla$ on $\Rep$, we can take the pullback connection $\nabla^\s$ on $\s^*\Rep\to \G$ and use it to differentiate $\Rep$-valued forms. However, in general, given a multiplicative form $\omega\in\Omega^k_\mult(\G;\Rep)$, its differential $\d^{\nabla^{\s}}\omega$ will not be a multiplicative form. We need the connection $\nabla$ to be compatible with the representation.

\begin{definition}\label{definition:invariant:connection}
Let $\G\tto M$ be a Lie groupoid and $p:\Rep\to M$ a $\G$-representation. A (usual) connection $\nabla$ on $\Rep$ is called {\bf $\G$-invariant} if the isomorphism of vector bundles:
\[ \s^*\Rep\to \t^*\Rep,\quad (g,v)\mapsto (g,g\cdot v), \]
preserves the connections $\nabla^\s:=\s^*\nabla$ and $\nabla^\t:=\t^*\nabla$.
\end{definition}

The $\G$-invariance of a connection is equivalent to the commutation relation:
\[ \d^{\nabla^\s}\delta=\delta\,\d^{\nabla},\]
where $\delta:\Omega^{\bullet}(M; \Rep)\to \Omega^{\bullet}_\mult(\G;\Rep)$ is the degree zero simplicial differential \eqref{eq:simplicial:E:forms:0}. It turns out that this is equivalent to require this commutation relation to hold for any simplicial degree. We only need the case of degree 1:

\begin{proposition}
\label{prop:differ:mult:forms}
Let $\G\tto M$ be a Lie groupoid, let $p:\Rep\to M$ be a $\G$-representation, and let $\nabla$ be a connection on $\Rep$. We denote by $\nabla^\s$ both the pullback connections on $\s^*\Rep\to\G$ and on $(\s\circ\pr_2)^*\Rep\to \G^{(2)}$. The $\nabla^\s$-de Rham differential $\d^{\nabla^{\s}}$ and the simplicial differential \eqref{eq:simplicial:E:forms:1} commute:
\[ \d^{\nabla^\s}\delta=\delta\,\d^{\nabla^\s},\]
if and only if the connection $\nabla$ is $\G$-invariant. In particular, if this is the case, then $\d^{\nabla^\s}$ maps multiplicative forms to multiplicative forms.
\end{proposition}

\begin{proof}
Let $E\to M$ be a vector bundle equipped with a linear connection $\nabla$. Given a map $\phi:N\to M$, denote by $\nabla^\phi$ the pullback connection on $\phi^*\Rep$. Then:
\[ \d^{\nabla^\phi}\phi^*=\phi^*\d^\nabla. \]

Now, we can write the simplicial differential with coefficients as:
\[ \delta=\Phi^*\pr_1^*-m^*+\pr_2^*, \]
where $\Phi$ is the vector bundle isomorphism covering the identity: 
\[ \Phi:(\s\circ\pr_1)^*\Rep=(\t\circ\pr_2)^*\Rep\to (\s\circ\pr_2)^*\Rep, \quad (g,h,v)\mapsto (g,h,h^{-1}v). \]
The condition that the connection $\nabla$ is  $\G$-invariant is easily seen to be equivalent to this map preserving the pullback connections:
\[ 
(\t\circ\pr_2)^*\nabla=\pr_2^*\nabla^\t, \quad (\s\circ\pr_2)^*\nabla=\pr_2^*\nabla^\s. 
\]
This condition, in turn, is equivalent to:
\[ \Phi^*\pr_1^*\d^\nabla=\d^{\nabla^\s}\Phi^*\pr_1^*. \]
Hence, we conclude that $\nabla$ is  $\G$-invariant if and only if $\d^{\nabla^\s}\delta=\delta\,\d^{\nabla^\s}$.
\end{proof}

We now pass to infinitesimal level.  First we deduce the infinitesimal analogue of the $\G$-invariance condition on a connection:

\begin{proposition}\label{prop:Invariant:connections}
Let $\G\tto M$ be an $\t$-connected Lie groupoid with Lie algebroid $A\Ato M$ and let $p:\Rep\to M$ be a $\G$-representation. Denote by $\nabla^A$ the corresponding flat $A$-connection on $\Rep$. A connection $\nabla$ is $\G$-invariant if and only if 
\begin{equation}
    \label{eq:connection:A-invariant}
    \nabla^A_a=\nabla_{\rho(a)}\quad \textrm{and}
    \quad R^\nabla(\rho(a),v)=0,
\end{equation}
for all $a\in A_x$, $v\in T_xM$, and $x\in M$.
\end{proposition}

\begin{proof}
We claim that $\G$-invariance of $\nabla$ is equivalent to the statement that for any path $g(t)\in\G$, one has a commutative diagram:
     \begin{equation}
         \label{diag:connection:G-invariant}
     \xymatrix{
     \Rep_{\s(g(0))} \ar[r]^{\tau^{\nabla}_{\s(g(t))}}\ar[d]_{g(0)} & \Rep_{\s(g(t))}\ar[d]^{g(t)}\\
     \Rep_{\t(g(0))} \ar[r]_{\tau^{\nabla}_{\t(g(t))}} & \Rep_{\t(g(t))}
     }
     \end{equation}
where $\tau^\nabla$ denotes parallel transport for the connection $\nabla$. To see this note that the $\G$-invariance condition on $\nabla$ is equivalent to require the map $\Phi:\s^*\Rep\to \t^*\Rep$, $(g,v)\mapsto (g,g\cdot v)$, to preserve parallel transport along any curve $g(t)$:
\[ \Phi\circ\tau^{\nabla^\s}_{g(t)}=\tau^{\nabla^\t}_{g(t)}. \]
The claim now follows by observing that since the connections are obtained by pulling back $\nabla$ along $\s$ and $\t$, we have:
\[ \tau^{\nabla^\s}_{g(t)}=(g(t),\tau^{\nabla}_{\s(g(t))}),\qquad 
\tau^{\nabla^\t}_{g(t)}=(g(t),\tau^{\nabla}_{\t(g(t))}).
\]

Fix $g\in \G$, and let $x:=\t(g)$. Since $\G$ is $\t$-connected we can choose a path $g(\eps)$ in the target fiber $\t^{-1}(x)$ with $g(0)=1_x$ and $g(1)=g$. Then we have the $A$-path $a:[0,1]\to T^*M$, $a(\eps):=g^{-1}(\eps) g'(\eps)$, and parallel transport along $a(\eps)$ for the $A$-connection $\nabla^A$ amounts to acting by the inverse of $g$:
\[ \tau_a^{\nabla^A}:\Rep_{\t(g)}\to \Rep_{\s(g)},\quad \tau_a^{\nabla^A}(v)= g^{-1}v. \]

We prove the proposition by showing the equivalence between $\G$-invariance in the form of the diagram \eqref{diag:connection:G-invariant} and conditions \eqref{eq:connection:A-invariant}. If we assume $\G$-invariance, then given any $a\in A_x$ we choose a path $g(\eps)$ in the target fiber $\t^{-1}(x)$ such that $g(0)=1_x$ and $g'(0)=a$. Then the $A$-path $a(\eps):=g^{-1}(\eps) g'(\eps)$ satisfies
\[ \tau^{\nabla^A}_{a}=\tau^{\nabla}_{\s\circ g}. \]
Since $\s\circ g$ is the base path of $a$, by differentiating at $t=0$, we obtain:
\[ \nabla^A_a=\nabla_{\rho(a)}.\]

Next, given $a\in A_x$ and $v\in T_x M$ we choose a smooth family of paths $g(t,\eps)$ where for each fixed $t$ the path $\eps\mapsto g(t,\eps)$ lies in the target fiber $\t^{-1}(x(t))$ and starts at an identity, $g(t,0)=1_{x(t)}$, and moreover:
\[ 
\frac{\d g}{\d \eps}(0,0)=a,\quad \frac{\d}{\d t}(\s\circ g)(0,0)=v. 
\]
It follows from diagram \eqref{diag:connection:G-invariant} that for the family $\gamma(t,\eps):=\s(g(t,\eps))$ the $\nabla$-parallel transports commute:
\[ 
\tau^{\nabla}_{\eps\mapsto \gamma(t,\eps)}
\circ\tau^{\nabla}_{t\mapsto \gamma(t,0)}
=\tau^{\nabla}_{t\mapsto \gamma(t,\eps)}\circ\tau^{\nabla}_{\eps\mapsto \gamma(0,\eps)}.
\]
By differentiation, this gives:
\[ R^{\nabla}(\rho(a),v)=R^{\nabla}\Big(\frac{\d \gamma}{\d \eps}(0,0),\frac{\d \gamma}{\d t}(0,0)\Big)=0. 
\]

This argument can be reversed, to prove the converse. If the curvature condition in \eqref{eq:connection:A-invariant} holds for all $x\in M$, $a\in A_x$ and $v\in T_x M$, then one obtains that for any smooth family of paths $g(t,\eps)$, where for each fixed $t$ the path $\eps\mapsto g(t,\eps)$ lies in a target fiber, the $\nabla$-parallel transports commute:
\[ 
\tau^{\nabla}_{\eps\mapsto \gamma(t,\eps)}
\circ\tau^{\nabla}_{t\mapsto \gamma(t,0)}
=\tau^{\nabla}_{t\mapsto \gamma(t,\eps)}\circ\tau^{\nabla}_{\eps\mapsto \gamma(0,\eps)},
\]
where $\gamma(t,\eps):=\s(g(t,\eps))$. Then if the relationship between connections in \eqref{eq:connection:A-invariant} holds, we conclude that \eqref{diag:connection:G-invariant} must hold for the path $g(t):=g(t,1)$. Since one can obtain any path $g(t)$ in $\G$ in this way, the result follows.
\end{proof}

This motivates the following:

\begin{definition}\label{def:A-inv:conn}
Let $A\Ato M$ be a Lie algebroid and let $(\Rep,\nabla^A)$ be an $A$-representation. We call a (usual) connection $\nabla$ on $\Rep$ {\bf $A$-invariant} if \eqref{eq:connection:A-invariant} holds. 
\end{definition}

In the presence of an invariant connection we have an operator on IM-forms, which is the infinitesimal version of the operator on multiplicative forms that we saw before:  

\begin{proposition}
\label{prop:differ:IM:forms}
Let $A\Ato M$ be a Lie algebroid and $(\Rep,\nabla^A)$ an $A$-representation. An $A$-invariant connection $\nabla$ on $\Rep$ induces operator on IM-forms, as follows:
\[ \d_\imult^{\nabla}:\Omega^k_{\imult}(A,\Rep)\to \Omega^{k+1}_{\imult}(A,\Rep), \quad (L,l)\mapsto (\d^\nabla L,L-\d^\nabla l).\]
If $\nabla$ is flat, this operator is a differential.
\end{proposition}

\begin{proof}
If $\nabla^A_a=\nabla_{\rho(a)}$ then we obtain:
\[ \Lie_\al=\Lie^\nabla_{\rho(\al)}. \]
On the other, this together with the condition $R^{\nabla}(\rho(a),v)=0$, gives:
\begin{align*}
    \d^\nabla\Lie_\al-\Lie_\al\d^\nabla&=\d^\nabla\Lie^\nabla_{\rho(\al)}-\Lie^\nabla_{\rho(\al)}\d^\nabla\\
    &=(\d^\nabla)^2i_{\rho(\al)}-i_{\rho(\al)}(\d^\nabla)^2\\
    &=R^{\nabla}\wedge (i_{\rho(\al)}\, \cdot)-i_{\rho(\al)}(R^{\nabla}\wedge \cdot)
    =-(i_{\rho(\al)}R^{\nabla})\wedge\cdot=0.
\end{align*}
It follows that if $(L,l)$ satisfies the IM conditions \eqref{eq:compatibility:IM:E:form}, so does $(\d^\nabla L,L-\d^\nabla l)$.
\end{proof}

There is also a natural map:
\begin{align}
\label{eq:map:complexes:IM}
    \Omega_{\imult}^{k}(A,\Rep)\to\, &\Omega^k(A,\Rep),\quad (L,l)\mapsto \omega_l\\
    &\text{where}\quad \omega_l(\al_1,\dots,\al_k):=l(\al_1)(\rho(\al_2),\dots,\rho(\al_k)).\notag
\end{align}
The corresponding operation on the Lie groupoid level is the restriction of multiplicative forms to the target fibers. It is immediate to check:

\begin{lemma}\label{lemma:chain:map:IM:Lie:alg}
Let $A\Ato M$ be a Lie algebroid and $(\Rep,\nabla^A)$ an $A$-representation. If $\nabla$ is an $A$-invariant connection then \eqref{eq:map:complexes:IM} intertwines the differentials:
\[ (\Omega_{\imult}^{\bullet}(A,\Rep), \d^{\nabla}_\imult)\to (\Omega^{\bullet}(A,\Rep),\d^{\nabla^A}). \]
Hence, when $\nabla$ is flat we obtain a map of complexes, and so a map in cohomology:
\begin{equation}
\label{eq:IM:Algebrd:chomology}
 H^{\bullet}_{\imult}(A,\Rep)\to H^{\bullet}(A,\Rep).
\end{equation} 
\end{lemma}

\begin{remark}
In the flat case we also obtain a alternative way of expressing IM-forms with coefficients in $\Rep$, which is similar to (in fact extends) the case of trivial coefficients. Namely, an $A$-invariant connection $\nabla$ gives a 1:1 correspondence between maps $L:\Gamma(A)\to\Omega^k(M;\Rep)$ with symbol $l:A\to \wedge^{k-1}T^*M\otimes \Rep$ and pairs of tensors $(\mu,\zeta)$, where $\mu:A\to \wedge^{k-1}T^*M\otimes \Rep$ and $\zeta:A\to \wedge^{k}T^*M\otimes \Rep$, by setting:
\[ \mu:=l,\quad \zeta:=L-\d^\nabla l. \]
When $\nabla$ is flat, one checks that $(L,l)$ is an IM-form if and only if the pair $(\mu,\zeta)$ satisfies the $\nabla$-analogue of equations \eqref{eq:mult:form}, namely:
\begin{align*}
i_{\rho(\be)} \mu(\al) &=-i_{\rho(\al)} \mu(\be),\\
\mu([\al,\be])&=\Lie^\nabla_{\rho(\al)}\mu(\be)-i_{\rho(\be)}\d^\nabla \mu(\al)-i_{\rho(\be)}\zeta(\al),\\
\zeta([\al,\be])&=\Lie^\nabla_{\rho(\al)}\zeta(\be)-i_{\rho(\be)}\d^\nabla \zeta(\al).
\end{align*}
\end{remark}

\subsection{Linear forms}

A differential form $\omega\in\Omega^k(F)$ on the total space of a vector bundle $F\to M$ is called \textbf{linear} if
\[  m_t^*\omega=t\,\omega,\quad \forall t>0. \]
where $m_t:F\to F$ denotes fiberwise multiplication by $t$. Let $\Omega_{\mathrm{lin}}^k(F)$ the space of linear $k$-forms on $F$.

To any vector bundle map $\Theta:F\to \wedge^{k}T^*M$ covering the identity, one can associate a linear $k$-form: 
\[\Theta^*(\alpha^k_{\can})\in \Omega^k_{\mathrm{lin}}(F),\]
where $\alpha^k_{\mathrm{can}}\in \Omega^k(\wedge^k T^*M)$ is the tautological $k$-form:
\[\alpha^k_{\mathrm{can}}|_{\xi}:=\pr_{\wedge^k T^*M}^*(\xi).\]
We have the following decomposition of linear forms:
\begin{lemma}\label{lemma:decomposition:linear:forms}
Any linear $k$-form $\omega\in\Omega^k_{\mathrm{lin}}(F)$ can be decomposed as:
\[\omega=\Xi^*(\alpha^k_{\can})+\d \Theta^*(\alpha^{k-1}_{\can}),\]
for unique linear maps 
\[\Xi:F\to \wedge^k T^*M, \quad \Theta:F\to \wedge^{k-1}T^*M.\]
Moreover, $\omega$ is closed if and only if $\omega=\d \Theta^*(\alpha^{k-1}_{\can})$, i.e., if and only if $\Xi\equiv 0$.
\end{lemma}

We have the corresponding notion on a VB groupoid:

\begin{definition}\label{definition:linear:multipl}
Given a VB groupoid:
\[
\xymatrix{
F \ar[r] \ar@<0.15pc>[d] \ar@<-0.15pc>[d] & \G \ar@<0.15pc>[d] \ar@<-0.15pc>[d]\\
E\ar[r] & M
}
\]
a form $\omega\in\Omega^k_\mult(F)$ is called a \textbf{linear multiplicative form} if it is multiplicative for the groupoid $F\tto E$ and linear for the vector bundle $F\to \G$.
\end{definition}
One important example is the canonical symplectic form $\omega_\can\in\Omega^2_\mult(T^*\G)$ on the cotangent VB groupoid
\[
\xymatrix{
T^*\G \ar[r] \ar@<0.15pc>[d] \ar@<-0.15pc>[d] & \G \ar@<0.15pc>[d] \ar@<-0.15pc>[d]\\
A^*\ar[r] & M
}
\]
where $A$ denotes the Lie algebroid of $\G$. The canonical primitive of $\omega_\can$, i.e., the Liouville 1-form $\al_{\can}^1\in\Omega^1_\mult(T^*\G)$, is also a linear multiplicative 1-form.
\smallskip

We now turn to the infinitesimal version of these results. 

\begin{definition}\label{definition:linear:IM}
A \textbf{linear IM form} on a VB algebroid
\[ 
\xymatrix{
B\ar@{=>}[d] \ar[r] & A\ar@{=>}[d] \\
E\ar[r] & M}
\]
is an IM form $(\mu,\zeta)\in\Omega^k_\imult(B)$ on the algebroid $B\to E$ satisfying 
\[  m_t^*\mu=t\,\mu,\quad m_t^*\zeta=t\,\zeta,\quad \forall t>0, \]
where $m_t:B\to B$ is the scalar multiplication on the vector bundle $B\to A$.
\end{definition}

Under the correspondence between multiplicative forms and IM forms, linear multiplicative forms on a VB groupoid correspond to linear IM forms on the associated VB algebroid.

For any Lie algebroid $A\to M$, its cotangent algebroid is a VB algebroid:
\[ 
\xymatrix{
T^*A\ar@{=>}[d] \ar[r] & A\ar@{=>}[d] \\
A^*\ar[r] & M}
\]
and carries a canonical closed, linear, multiplicative 2-form $\mu_\can\in\Omega^2_\imult(T^*A)$, namely the canonical isomorphism (called the reverse isomorphism in \cite{Mackenzie05}):
\[ \mu_\can:T^*A\to T^*A^*. \]
This IM 2-form is exact with canonical primitive the linear, multiplicative 1-form $(f_\can,\mu_\can)\in\Omega^1_\imult(T^*A)$, where $f_\can:T^*A\to\R$ is the tautological map
\[ f_\can(\al)=\al\Big(\frac{\d}{\d t}\big|_{t=0} e^tu \Big ),\quad \text{ for }\quad\al\in T^*_u A. \]

\end{document}